\documentclass[A4paper,12pt,twoside]{article}
\usepackage{geometry}
\usepackage{fancyhdr}
\usepackage{graphicx}
\usepackage{amssymb}
\usepackage{amsmath}
\usepackage{amsfonts}
\usepackage{mathrsfs}
\usepackage{mathtools}
\usepackage{bm}
\usepackage{color}
\usepackage{setspace}
\usepackage{exscale}
\usepackage{relsize}
\usepackage{epstopdf}
\usepackage{float}
\usepackage{cite}
\usepackage{nicefrac}
\usepackage{setspace}
\usepackage{algorithm,algorithmic}
\usepackage{tikz}
\usetikzlibrary{positioning}
\usepackage{xcolor}
\usepackage{rotating}
\usepackage{enumerate}
\usepackage{hyperref}
\usepackage{authblk}
\usepackage{soul}
\usepackage{textcomp}
\usepackage[title]{appendix}

\usepackage{cleveref}

\usepackage{theorem}
\usepackage{empheq}
\usepackage{booktabs,multirow} 
\usepackage{array} 
\usepackage{paralist} 
\usepackage{verbatim} 
\usepackage{subfigure} 

\usepackage[normalem]{ulem}

\newtheorem{theorem}{Theorem}[section]
\newtheorem{remark}{Remark}[section]

\newtheorem{definition}{Definition}[section]

\newenvironment{proof}{{\flushleft \bf Proof:}}{}

\newenvironment{acknowledgment}{{\flushleft \bf Acknowledgment:}}{}

\allowdisplaybreaks[1]
\numberwithin{equation}{section}
\numberwithin{figure}{section}
\numberwithin{table}{section}

\usepackage{ae,aecompl}

\def\muu {\bm{u}}

\def\bnabla{\bm{\nabla}}

\newcommand{\dt}{\Delta t}
\newcommand{\dx}{\Delta x}
\newcommand{\dy}{\Delta y}
\newcommand{\kph}{{k+\frac{1}{2}}}
\newcommand{\kmh}{{k-\frac{1}{2}}}
\newcommand{\jph}{{j+\frac{1}{2}}}
\newcommand{\jmh}{{j-\frac{1}{2}}}

\newcommand{\mF}{\bm{F}}

\newcommand{\mG}{\bm{G}}

\newcommand{\bV}{\bm{V}}

\newcommand{\ve}{\varepsilon}

\newcommand\eref[1]{(\ref{#1})}

\makeatletter
\newcommand{\refcheckize}[1]{%
  \expandafter\let\csname @@\string#1\endcsname#1%
  \expandafter\DeclareRobustCommand\csname relax\string#1\endcsname[1]{%
    \csname @@\string#1\endcsname{##1}\wrtusdrf{##1}}%
  \expandafter\let\expandafter#1\csname relax\string#1\endcsname
}
\makeatother

\newcommand*\xbar[1]{%
  \hbox{%
    \vbox{%
      \hrule height 0.5pt 
      \kern0.4ex
      \hbox{%
        \kern-0.05em
        \ensuremath{#1}%
        \kern-0.00em
      }%
    }%
  }%
}

\newcommand{\uvec}[2][3]{\boldsymbol{#2\mkern-#1mu}\mkern#1mu}

\newcommand\abs[1]{\left\lvert#1\right\rvert}

\newcommand{\bU}{\uvec{U}}

\graphicspath{{Figures/}}

\title{A New Asymptotic-Preserving Dual Formulation Finite-Volume Method for the Compressible Euler Equations}
\author{Alina Chertock\thanks{Department of Mathematics, North Carolina State University, Raleigh, NC 27695, USA;
{\tt chertock@math.ncsu.edu}}, Smadar Karni\thanks{Department of Mathematics, University of Michigan, 48109, USA; {\tt karni@umich.edu}},
Alexander Kurganov\thanks{Department of Mathematics and Shenzhen International Center for Mathematics, Southern University of Science and
Technology, Shenzhen, 518055, China; {\tt alexander@sustech.edu.cn}}, and Lorenzo Micalizzi\thanks{Department of Mathematics, North Carolina
State University, Raleigh, NC 27695, USA; {\tt lmicali@ncsu.edu}}}
\date{}

\begin{document}
\maketitle

\begin{abstract}
The paper focuses on the development of numerical methods for the compressible Euler equations. It is well-known that if the Mach number is
small, the system becomes stiff and hence explicit schemes suffer from severe time-step restrictions, making them inefficient or even
impractical. Our objective is to develop an asymptotic preserving (AP) scheme that remains uniformly accurate and stable across all Mach
numbers.

Instead of the conservative hyperbolic flux splitting approach, which is widely used to design AP schemes, we consider a primitive
(nonconservative) formulation and introduce a nonconservative hyperbolic splitting. The resulting system is discretized using a
semi-implicit approach: the stiff part is handled semi-implicitly using second-order central differences, while the nonstiff part is treated
explicitly using a second-order path-conservative central-upwind discretization. A key feature of our method is that the pressure at each
time level is computed by solving a well-posed Poisson-type elliptic equation, thereby enforcing the AP property. Simultaneously, we evolve
the conservative form of the system using a semi-discrete central-upwind (CU) scheme. At the end of each stage of the time discretization,
we perform a special post-processing that selects the appropriate numerical solution depending on the Mach number. This guarantees that in
low-Mach-number regimes, the solution is obtained by the AP nonconservative scheme, while in higher-Mach-number regimes, a sharp and
physically relevant solution is computed by the conservative CU scheme.

Numerical experiments confirm that the proposed AP scheme achieves the expected second order of accuracy and that the time-step constraint
is independent of the Mach number, making it a robust and efficient alternative to conventional explicit methods.
\end{abstract}



\section{Introduction}
The paper focuses on the compressible Euler equations, which, like any other hyperbolic system of PDEs, are characterized by a finite speed
of propagation. This plays a crucial role in the development of explicit numerical methods, for which a major stability requirement is to
keep the time steps inversely proportional to the maximum wave speed over the entire computational domain.

It is well-known that low-Mach-number flows pose several major challenges for numerical simulations. A distinctive feature of such regimes
is the appearance of both slow material waves, which transport quantities like entropy and vorticity, and fast acoustic waves, whose speeds
scale inversely with the Mach number. As the Mach number decreases, the resulting stiffness imposes severe time-step restrictions on
explicit methods and leads to excessive numerical diffusion, making such schemes inefficient or even impractical for real applications.
Fully-implicit methods can address the stiffness, but have their own drawbacks: they tend to oversmear material waves (see, e.g.,
\cite{boscarino2018all}), require the solution of large nonlinear systems, and may fail to capture the correct solution in the
zero-Mach-number limit.

To overcome these difficulties, a widely adopted strategy is to use either implicit-explicit (IMEX) or semi-implicit (SI) methods based on
conservative hyperbolic flux splitting. This approach decomposes the hyperbolic flux into stiff and nonstiff components in a manner that
preserves the conservative structure of the original system. The fast (stiff) part, associated with acoustic waves, is treated
(semi-)implicitly to relax time-step limitations, while the slow (nonstiff) part is handled explicitly to accurately capture the evolution
of material waves without excessive numerical diffusion.

It is also known that, as the Mach number tends to zero, the compressible Euler equations reduce to the incompressible Euler equations. It
is essential to ensure that numerical schemes also exhibit the same limiting behavior at the discrete level and provide a consistent
discretization of the incompressible Euler equations as the Mach number tends to zero. Schemes that maintain this property are called
asymptotic-preserving (AP). They were originally introduced to capture steady-state solutions for neutron transport in the diffusive regime
\cite{LMM,LarMor}, but the specific definition was introduced in \cite{GJL,jin1999efficient,Kla99} in the context of stiff kinetic
equations. In recent years, AP schemes have been extensively studied and applied for simulating low-Mach-number flows; see, e.g.,
\cite{AllVig,boscarino2019high,boscarino2018all,BDLTV,CGK,cordier2012asymptotic,DJL,degond2011all,DLMV,dimarco2017study,haack2012all,Kle95,noelle2014weakly,Tang12,TorVaz,zeifang2020novel}
for a non-exhaustive list of references.

All of the aforementioned AP schemes, which were designed for either the isentropic or full compressible Euler system, are based on 
different flux splitting strategies.  A very simple and robust flux splitting, which was proposed in \cite{haack2012all} for the isentropic 
Euler equations and later extended to the rotating shallow water equations in \cite{LCK}, seems to be rather optimal in the sense that
it very accurately identifies and separates a linear stiff pressure term, which is then discretized implicitly. However, extending this flux
splitting to the full Euler equations presents significant challenges. 

In this paper, we propose an alternative way of accurately identifying and separating a stiff part of the full Euler system: we first
rewrite the studied system in a nonconservative form and then introduce a nonconservative hyperbolic splitting, which may be naturally
viewed as an extension of the flux splitting from \cite{haack2012all}. We integrate the nonconservative system using a SI method 
implemented as follows. The stiff part is approximated semi-implicitly using a simple second-order accurate central-differencing, and the 
nonstiff part is handled explicitly using a second-order {\color{black}path-conservative central-upwind (PCCU)} discretization. The resulting
SI approach is then realized in such a way that the pressure update consists of solving a Poisson-type elliptic equation: this is used to 
enforce the AP property.

However, the resulting SI method can only be applied to low-Mach-number regimes, where the magnitude of discontinuous waves is small.
For large Mach numbers, solving nonconservative formulations of the Euler equations in the presence of discontinuities typically leads to
nonphysical computed solutions, as was demonstrated in \cite{abgrall2010comment,HouLeF}. We therefore apply a dual formulation (DF)
approach, which has been recently introduced in \cite{CFKM} (for other recent works on DF methods, we refer the reader to
\cite{Abg2023,ACKM,AbgLiu,pidatella2019semi}), and solve the nonconservative and conservative formulations simultaneously. The latter one is
discretized in a fully-explicit manner using the second-order semi-discrete central-upwind (CU) discretization from \cite{KLin}. This way,
at each stage of a multi-stage SI time discretization (we have used the deferred correction (DeC) time discretization from
\cite{offner2025analysis}), two copies of the computed solution are evolved: one is AP, but nonconservative, while the second one is
conservative, but non-AP. Hence, upon the completion of each stage of the time evolution, we post-process the obtained solutions to
automatically ensure that in low-Mach-number regimes, the overall numerical solution is obtained by the AP nonconservative scheme, while in
large (intermediate)-Mach-number regimes, the solution reduces to the sharp and conservative solution obtained by the CU scheme.

The rest of the paper is organized as follows. In \S\ref{sec2}, we give the necessary preliminaries: we introduce the governing equations,
namely, non-dimensional conservative and nonconservative (primitive) formulations of the full Euler equations, discuss their
zero-Mach-number limit, and briefly review the considered DF framework. In \S\ref{sec4}, we introduce the novel AP DF finite-volume (DF-FV)
scheme for compressible Euler equations, providing a rigorous proof of its AP character. In \S\ref{sec6}, we demonstrate the performance of
the proposed scheme on a number of challenging numerical examples. Finally, concluding remarks can be found in \S\ref{sec5}.

\section{Preliminaries}\label{sec2}
The main goal of this section is to provide the background needed for presenting the proposed AP scheme. Specifically, we will describe:

\smallskip
\noindent
$\bullet$ a nonconservative reformulation of the Euler equations in terms of primitive variables, which allows for a natural decomposition
of the terms, which are stiff and nonstiff in the low-Mach-number regime;

\smallskip
\noindent
$\bullet$ a formal asymptotic analysis of the studied equations in low-Mach-number regimes, providing the incompressible system that the AP
scheme must accurately approximate in the zero-Mach-number limit;

\smallskip
\noindent
$\bullet$
a DF framework, in which both conservative and nonconservative formulations of the studied system are numerically solved simultaneously
exploiting the advantages of each of them in the corresponding Mach-number regimes.

\subsection{Conservative and Primitive Formulations}\label{sec21}
After suitable non-dimensionalization and rescaling, the two-dimensional (2-D) compressible Euler equations can be written in the
conservative form as
\begin{align}
	&\rho_t+\bnabla\!\cdot\!(\rho\bm u)=0,\label{2.1a}\\
	&(\rho\bm u)_t+\bnabla\!\cdot\!(\rho\bm u\otimes\bm u)+\frac{1}{\ve^2}\nabla p=\bm0,\label{2.1b}\\
	&E_t+\bnabla\!\cdot\!((E+p)\bm u)=0.\label{2.1c}
\end{align}
Here, $\rho$, $\bm u:=(u,v)^\top$, and $E$ denote the density, velocity, and total energy, respectively, $p$ is the pressure, $\ve$ is the 
reference Mach number, and the system is closed by the equation of state, which, in the case of an ideal gas, reads as
\begin{equation}
	E=\frac{p}{\gamma-1}+\frac{\ve^2}{2}\rho(u^2+v^2),
	\label{2.1d}
\end{equation}
with $\gamma$ being the specific heat ratio. This system is hyperbolic and features acoustic waves traveling with (maximum) speed
$|\bm u|+c$, where $c$ is the speed of sound given by $c:=\frac{1}{\ve}\sqrt{\gamma p/\rho}$. Notice that in low-Mach-number regimes, the
acoustic waves travel at a high (maximum) speed proportional to $1/\ve$.

For the purpose of deriving our AP scheme, we also consider an equivalent nonconservative formulation of the system \eref{2.1a}--\eref{2.1d}
in terms of the primitive variables $\rho$, $\bm u$, and $p$:
\begin{align}
	&\rho_t+\bnabla\!\cdot\!(\rho\bm u)=0,\label{2.5a}\\
	&\bm u_t+(\bm u\!\cdot\!\bnabla)\bm u+\frac{1}{\ve^2\rho}\nabla p=\bm0,\label{2.5b}\\
	&p_t+\bm u\!\cdot\!\nabla p+\gamma p\bnabla\!\cdot\!\bm u=0.\label{2.5c}
\end{align}
We emphasize that this formulation is equivalent to the conservative system \eref{2.1a}--\eref{2.1d} only for smooth solutions, and
numerical approximations of \eref{2.5a}--\eref{2.5c} typically converge to nonphysical solutions when discontinuities are present; see 
\cite{abgrall2010comment,HouLeF} for a detailed discussion.

\subsection{Zero-Mach-Number Limit}\label{sec22}
It is well-known (see, e.g., a formal analysis in \cite{Ala,MetSch}) that in the zero-Mach-number limit the compressible Euler equations
reduce to the incompressible ones. To illustrate this, we examine the formal behavior of the primitive system \eref{2.5a}--\eref{2.5c} as
$\ve\to0$. We substitute the formal expansions
\begin{equation*}
	\rho=\rho^{(0)}+\ve\rho^{(1)}+\ve^2\rho^{(2)}+\dots,\quad\bm u=\bm u^{(0)}+\ve\bm u^{(1)}+\ve^2\bm u^{(2)}+\dots,\quad
	p=p^{(0)}+\ve p^{(1)}+\ve^2p^{(2)}+\dots
\end{equation*}
into \eref{2.5a}--\eref{2.5c} and collect terms by powers of $\ve$. This yields
\begin{align}
	{\cal O}(\ve^{-2}):\quad&\nabla p^{(0)}=\uvec{0},\label{2.8f}\\
	{\cal O}(\ve^{-1}):\quad&\nabla p^{(1)}=\uvec{0},\label{2.9f}\\
	{\cal O}(1):\quad&\rho^{(0)}_t+\bnabla\!\cdot\!(\rho^{(0)}\bm u^{(0)})=0,\label{2.10f}\\
	&\bm u^{(0)}_t+(\bm u^{(0)}\!\cdot\!\bnabla)\bm u^{(0)}+\frac{1}{\rho^{(0)}}\nabla p^{(2)}=\bm0,\label{2.11f}\\
	&p^{(0)}_t+\gamma p^{(0)}\bnabla\!\cdot\!\bm u^{(0)}=0,\label{2.12f}\\
	{\cal O}(\ve):\quad&p^{(1)}_t+\gamma\big(p^{(1)}\bnabla\!\cdot\!\bm u^{(0)}+p^{(0)}\bnabla\!\cdot\!\bm u^{(1)}\big)=0.\label{2.13f}
\end{align}

It follows from \eref{2.8f}--\eref{2.9f} that $p^{(0)}(x,y,t)=p^{(0)}(t)$ and $p^{(1)}(x,y,t)=p^{(1)}(t)$ are spatially uniform. One can
also show that both $p^{(0)}$ and $p^{(1)}$ are independent of time, provided the following Dirichlet boundary condition holds:
\begin{equation}
	p=\,\xbar p_0+\ve\,\xbar p_1+\ve^2 p_2+\dots,\quad\forall(x,y)\in\partial\Omega,
	\label{3.2}			
\end{equation}
where $\xbar p_0>0$ and $\xbar p_1$ are constants, $p_k(x,y,t)$, $k\ge2$ are bounded functions, and $\Omega$ is the spatial domain with
boundary $\partial\Omega$. In fact, such a boundary condition implies $p^{(0)}\equiv\,\xbar p_0$ and $p^{(1)}\equiv\,\xbar p_1$. Using this
in \eref{2.12f}, one concludes that $\bnabla\!\cdot\!\bm u^{(0)}=0$, which further implies $\bnabla\!\cdot\!\bm u^{(1)}=0$ thanks to
\eref{2.13f}. Hence, the zero-Mach-number limiting equations are
\begin{align}
	&\rho^{(0)}_t+\bnabla\!\cdot\!(\rho^{(0)}\bm u^{(0)})=0,\label{3.3}\\
	&\bm u^{(0)}_t+(\bm u^{(0)}\!\cdot\!\bnabla)\bm u^{(0)}+\frac{1}{\rho^{(0)}}\nabla p^{(2)}=\bm0,\label{3.4}\\
	&\bnabla\!\cdot\!\bm u^{(0)}=0,\quad p^{(0)}\equiv\,\xbar p_0,\label{3.5}
\end{align}
and the correct low-Mach-number scaling for the ${\cal O}(\ve)$ terms is
\begin{equation}
	\bnabla\!\cdot\!\bm u^{(1)}=0,\quad p^{(1)}\equiv\,\xbar p_1.
	\label{3.6}
\end{equation}

\subsection{Dual Formulation (DF) Framework}\label{sec4a}
As outlined above, a new AP scheme for the compressible Euler equations will be constructed within the DF framework, in which both
conservative and primitive formulations are numerically solved simultaneously. The conservative form ensures proper handling of
discontinuities, while the primitive form is used to achieve the AP property in the zero-Mach-number limit. While the DF methodology is not
specifically developed to handle multiscale features, it lays the groundwork for the AP scheme developed in subsequent sections by enabling
an efficient and accurate treatment of the studied Euler equations in both compressible and nearly incompressible flow regimes.

We consider a general 2-D hyperbolic system of conservation laws,
\begin{equation}
	\bm U_t+\bm F(\bm U)_x+\bm G(\bm U)_y=\bm0,
	\label{4.1f}
\end{equation}
where $\bm U$ is the vector of conservative variables and $\bm F$ and $\bm G$ are fluxes, and rewrite it in an equivalent nonconservative
form
\begin{equation}
	\bm V_t+\widetilde{\bm F}(\bm V)_x+\widetilde{\bm G}(\bm V)_y=B(\bm V)\bm V_x+C(\bm V)\bm V_y,
	\label{4.2f}
\end{equation}
where $\bm V$ is the vector of nonconservative variables, $\widetilde{\bm F}$ and $\widetilde{\bm G}$ are the corresponding fluxes, and
$B(\bm V)\bm V_x$ and $C(\bm V)\bm V_y$ are the nonconservative product terms.

The key idea of the DF approach is to evolve the solutions of \eref{4.1f} and \eref{4.2f} simultaneously. A crucial step in DF-based 
methods is a post-processing, in which the evolved values of the nonconservative solution $\bm V$ are replaced with a more reliable 
approximation after the update. This step is necessary because long-term evolutions of $\bm V$ by directly solving the nonconservative 
system \eref{4.2f} may lead to nonphysical solutions in the presence of discontinuities, which typically appear when the studied Euler 
system is considered in the compressible (large/medium-Mach-number) regime.

The post-processing can be described as follows. After advancing the solutions of \eref{4.1f} and \eref{4.2f} from a certain time level $t$
to the next time level $t+\dt$, the evolved values of $\bm V(t+\dt)$ are replaced with
\begin{equation}
	r\big(\bm V(\bm U(t+\dt)),\bm V(t+\dt)\big),
	\label{4.5f}
\end{equation}
where $r$ is a suitable replacement function and $\bm V(\bm U)$ is a conservative-primitive variable transformation. In the simplest
nonstiff case, one can set
\begin{equation}
	r\big(\bm V(\bm U(t+\dt)),\bm V(t+\dt)\big)=\bm V(\bm U(t+\dt)).
	\label{4.6f}
\end{equation}
However, in the development of the AP scheme below, we will modify the post-processing \eref{4.6f} by taking an appropriate function $r$ in
\eref{4.5f} to ensure that in the nearly incompressible (low-Mach-number) regime the AP $\bm V$-solution is not overwritten by the non-AP
conservative one.

\section{Novel AP Scheme for Compressible Euler Equations}\label{sec4}
Building on the DF framework described in \S\ref{sec4a}, we now present a novel AP scheme designed for the compressible Euler equations
across all Mach-number regimes, from fully compressible to nearly incompressible flows. The proposed method couples the
conservative \eref{2.1a}--\eref{2.1d} and primitive \eref{2.5a}--\eref{2.5c} formulations of the system, ensuring stability, accuracy, and
consistency with the analytical asymptotic behavior as $\ve\to0$.

In this section, we provide a complete description of the proposed space-time discretization, starting with the primitive system
\eref{2.5a}–\eref{2.5c}. In \S\ref{sec411}, we outline its temporal integration, which is based on a new hyperbolic splitting and an SI
approach. In \S\ref{sec32}, we present a fully discrete second-order AP scheme for the primitive system, and in \S\ref{sec42}, we describe
the semi-discrete CU scheme employed for the conservative system. \S\ref{sec422} is devoted to the clarification of important
implementation details. Finally, in \S\ref{sec43}, we present the Mach-number dependent post-processing strategy used to reconcile primitive
and conservative variables.

\subsection{Novel AP Time Discretization of the Primitive System}\label{sec411}
We begin by providing a precise definition of an asymptotic-preserving (AP) time discretization in the context of the zero-Mach-number
limit.
\begin{definition}[AP time discretization]\label{def41}
	Assume that the {\color{black}semi-discrete (not discretized in space) solution at time $t^n$, $(\rho^n(x,y),\bm u^n(x,y),p^n(x,y))$,} can be
	expanded as
	\begin{equation}
		\begin{aligned}
			&\rho^n=\rho^{(0),n}+\ve\rho^{(1),n}+\ve^2\rho^{(2),n}+\dots,\quad\bm u^n=\bm u^{(0),n}+\ve\bm u^{(1),n}+\ve^2\bm u^{(2),n}+\dots,\\
			&p^n=p^{(0),n}+\ve p^{(1),n}+\ve^2p^{(2),n}+\dots,
		\end{aligned}
		\label{4.1}
	\end{equation}
	which is compatible with the asymptotic limits \eref{3.5} and \eref{3.6}, that is,
	\begin{align}
		&p^{(0),n}\equiv\,\xbar p_0,\quad p^{(1),n}\equiv\,\xbar p_1,\label{3.2f}\\
		&\bnabla\!\cdot\!\bm u^{(0),n}=0,\quad\bnabla\!\cdot\!\bm u^{(1),n}=0,\label{4.2}
	\end{align}
	and assume that the Dirichlet boundary condition \eref{3.2} holds, namely,
	\begin{equation}
		p\equiv\,\xbar p_0+\ve\,\xbar p_1+\ve^2p_2+\dots,\quad\forall(x,y)\in\partial\Omega.
		\label{3.4f}
	\end{equation}
	
	Let us consider a one-step time discretization that produces at time $t^{n+1}=t^n+\dt$ an approximation $(\rho^{n+1},\bm u^{n+1},p^{n+1})$.
	We say that such time discretization is AP if it admits an asymptotic expansion of the type \eref{4.1} and yields a consistent
	discretization of \eref{3.3}--\eref{3.6} as $\ve\to0$.
\end{definition}

To construct an AP time discretization, we first perform a hyperbolic splitting of the primitive system \eref{2.5a}–\eref{2.5c}, separating
the stiff pressure-driven terms from the nonstiff convective terms. This splitting enables the use of an SI integration strategy in which
the stiff and nonstiff terms are treated semi-implicitly and explicitly, respectively, ensuring uniform stability and asymptotic consistency
as $\ve\to0$.

\subsubsection{A New Hyperbolic Splitting}
We follow the idea from \cite{haack2012all,LCK} and split the nonconservative system into two parts corresponding to the slow and fast
dynamics as follows. We first define the time-dependent variables 
\begin{equation}
	\rho_{\max}(t):=\max\limits_{(x,y)\in\Omega}\rho(x,y,t),\quad p_{\min}(t):=\min\limits_{(x,y)\in\Omega}p(x,y,t),
	\label{4.3}
\end{equation}
and then add and subtract $\frac{1}{\ve^2\rho_{\max}}\nabla p$ and $\gamma p_{\min}\bnabla\cdot\bm u$ from \eref{2.5b} and \eref{2.5c},
respectively, to rewrite system \eref{2.5a}--\eref{2.5c} as follows:
\begin{align}
	&\rho_t+\bnabla\!\cdot\!(\rho\bm u)=0,\label{2.1}\\
	&\bm u_t+(\bm u\!\cdot\!\bnabla)\bm u+\frac{\rho_{\max}-\rho}{\ve^2\rho\rho_{\max}}\,\nabla p=-\frac{1}{\ve^2\rho_{\max}}
	\nabla p,\label{2.2}\\[0.5ex]
	&p_t+\bm u\!\cdot\!\nabla p+\gamma(p-p_{\min})\bnabla\!\cdot\!\bm u=-\gamma p_{\min}\bnabla\!\cdot\!\bm u.\label{2.3}
\end{align}
This system can be put in the following vector form:
\begin{equation}
	\bm V_t+\widetilde{\mF}(\bm V)_x+\widetilde{\mG}(\bm V)_y=\widetilde B(\bm V)\bm V_x+\widetilde C(\bm V)\bm V_y+\widehat B(\bm V)\bm V_x+
	\widehat C(\bm V)\bm V_y,
	\label{4.47}
\end{equation}
where $\bm V:=(\rho,u,v,p)^\top$, the nonlinear nonstiff (slow dynamics) part consists of the fluxes
\begin{equation*}
	\widetilde{\mF}(\bm V)=\Big(\rho u,\frac{u^2}{2},0,0\Big)^\top\quad\mbox{and}\quad
	\widetilde{\mG}(\bm V)=\Big(\rho v,0,\frac{v^2}{2},0\Big)^\top,
\end{equation*}
and the nonstiff nonconservative terms $\widetilde B(\bm V)\bm V_x+\widetilde C(\bm V)\bm V_y$ with matrices
\begin{equation*}
	\widetilde B=-\begin{pmatrix}0&0&0&0\\0&0&0&\dfrac{\rho_{\max}-\rho}{\ve^2\rho\rho_{\max}}\\0&v&0&0\\0&\gamma(p-p_{\min})&0&u\end{pmatrix}
	\quad\mbox{and}\quad
	\widetilde C=-\begin{pmatrix}0&0&0&0\\0&0&u&0\\0&0&0&\dfrac{\rho_{\max}-\rho}{\ve^2\rho\rho_{\max}}\\0&0&\gamma(p-p_{\min})&v\end{pmatrix},
\end{equation*}
while, the linear stiff (fast dynamics) part consists of the stiff nonconservative terms 
$\widehat B(\bm V)\bm V_x+\widehat C(\bm V)\bm V_y$
with matrices
\begin{equation*}
	\widehat B=-\begin{pmatrix}0&0&0&0\\0&0&0&\dfrac{1}{\ve^2\rho_{\max}}\\0&0&0&0\\0&\gamma p_{\min}&0&0\end{pmatrix}
	\quad\mbox{and}\quad
	\widehat C=-\begin{pmatrix}0&0&0&0\\0&0&0&0\\0&0&0&\dfrac{1}{\ve^2\rho_{\max}}\\0&0&\gamma p_{\min}&0\end{pmatrix}.
\end{equation*}

We stress that the subsystem
$\bm V_t+\widetilde{\mF}(\bm V)_x+\widetilde{\mG}(\bm V)_y=\widetilde B(\bm V)\bm V_x+\widetilde C(\bm V)\bm V_y$ is indeed nonstiff as the
eigenvalues of the matrices
$\frac{\partial\widetilde{\mF}}{\partial\bm V}(\uvec{V})-\widetilde B(\bm V)$ and
$\frac{\partial\widetilde{\mG}}{\partial\bm V}(\uvec{V})-\widetilde C(\bm V)$, are $\left\{u\pm\tilde c,u,u\right\}$ and
$\left\{v\pm\tilde c,v,v\right\}$, respectively, with 
\begin{equation}
	\tilde c:=\frac{1}{\ve}\sqrt{\gamma\frac{(\rho_{\max}-\rho)(p-p_{\min})}{\rho\rho_{\max}}},
	\label{4.48f}
\end{equation}
which are real and of size ${\cal O}(1)$ thanks to the definitions of $\rho_{\max}$ and $p_{\min}$ in \eref{4.3} and to the asymptotic
analysis in \S\ref{sec22}, which ensure that
\begin{equation}
	0\le\rho_{\max}-\rho={\cal O}(1)\quad\mbox{and}\quad0\le p-p_{\min}={\cal O}(\ve^2).
	\label{4.7}
\end{equation}

In the next subsection, we will utilize this splitting and design an AP time discretization based on an explicit approximation of the
nonstiff subsystem and an SI discretization of the stiff terms on the right-hand sides (RHSs) of \eref{2.2} and \eref{2.3}.
{\color{black}
	\begin{remark}
		It should be pointed out that extensions of the proposed splitting to more general equations of state would require a re-examination of the
		primitive-variable formulation.
\end{remark}}

\subsubsection{First-Order AP SI Time Discretization}
The simplest first-order AP SI time discretization of the system \eref{4.47} reads as
\begin{equation}
	\begin{aligned}
		&\frac{\rho^{n+1}-\rho^n}{\dt}+\bnabla\!\cdot\!(\rho^n\bm u^n)=0,\\
		&\frac{\bm u^{n+1}-\bm u^n}{\dt}+(\bm u^n\!\cdot\!\bnabla)\bm u^n+\frac{\rho_{\max}^n-\rho^n}{\ve^2\rho^n\rho_{\max}^n}\,\nabla p^n+
		\frac{1}{\ve^2\rho_{\max}^n}\nabla p^{n+1}=\bm0,\\
		&\frac{p^{n+1}-p^n}{\dt}+\bm u^n\!\cdot\!\nabla p^n+\gamma(p^n-p_{\min}^n)\bnabla\!\cdot\!\bm u^n+\gamma p_{\min}^n
		\bnabla\!\cdot\!\bm u^{n+1}=0,
	\end{aligned}
	\label{4.8}
\end{equation}
which can also be written in the following vector form:
\begin{equation*}
	\frac{\bm V^{n+1}-\bm V^n}{\dt}+\bm{{\cal R}}^n+\bm{{\cal L}}^{n,n+1}=\bm0,
\end{equation*}
where $\bm V^n:=(\rho^n,\bm u^n,p^n)^\top\approx\bm V(t^n)$, $\rho_{\max}^n:=\rho_{\max}(t^n)$, $p_{\min}^n:=p_{\min}(t^n)$, and 
\begin{align}
	&\begin{aligned}
		\bm{{\cal R}}^n&:=\widetilde{\mF}(\bm V^n)_x+\widetilde{\mG}(\bm V^n)_y-\widetilde B(\bm V^n)\bm V_x^n-\widetilde C(\bm V^n)\bm V_y^n\\
		&\hspace*{0.15cm}=\begin{pmatrix}({\cal R}^\rho)^n\\(\bm{{\cal R}}^{\bm u})^n\\({\cal R}^p)^n\end{pmatrix}=
		\begin{pmatrix}\bnabla\!\cdot\!(\rho^n\bm u^n)\\(\bm u^n\!\cdot\!\bnabla)\bm u^n+\dfrac{\rho_{\max}^n-\rho^n}{\ve^2\rho^n\rho_{\max}^n}
			\nabla p^n\\[1.5ex]
			\bm u^n\!\cdot\!\nabla p^n+\gamma(p^n-p_{\min}^n)\bnabla\!\cdot\!\bm u^n\end{pmatrix},
	\end{aligned}\label{3.12f}\\
	&\bm{{\cal L}}^{n,n+1}:=-\widehat B(\bm V^n)\bm V_x^{n+1}-\widehat C(\bm V^n)\bm V_y^{n+1}=
	\begin{pmatrix}({\cal L}^\rho)^{n,n+1}\\(\bm{{\cal L}}^{\bm u})^{n,n+1}\\({\cal L}^p)^{n,n+1}\end{pmatrix}
	=\begin{pmatrix}0\\\dfrac{\nabla p^{n+1}}{\ve^2\rho_{\max}^n}\\[1.5ex]
		\gamma p_{\min}^n\bnabla\!\cdot\!\bm u^{n+1}\end{pmatrix}.\label{3.13f}
\end{align}

Notice that in \eref{3.13f}, $\bm{{\cal L}}^{n,n+1}$ is discretized in an SI (rather than fully implicit) manner, where both $\widehat B$
and $\widehat C$ are evaluated at $\bm V^n$ (and not at $\bm V^{n+1}$), which prevents from numerically solving complicated systems of
nonlinear algebraic equations.

We shall now prove that the time discretization \eref{4.8} is indeed AP, provided that the time step is computed based on the wave speeds 
of the nonstiff subsystem, that is, provided that
\begin{equation}
	\dt=K_{\rm CFL}\min\left\{\frac{\dx}{\max\limits_{(x,y)\in\Omega}\big(|u|+\tilde c\big)},\,
	\frac{\dy}{\max\limits_{(x,y)\in\Omega}\big(|v|+\tilde c\big)}\right\},
	\label{4.10}
\end{equation}
where $K_{\rm CFL}$ is a CFL number and $\dx$ and $\dy$ are mesh sizes used in the spatial discretization. Notice that selecting the time
step $\dt$ according to \eref{4.10} makes it asymptotically independent of $\ve$ as, according to \eref{4.48f}--\eref{4.7},
$\tilde c={\cal O}(1)$.

\begin{theorem}\label{thm31}
	The first-order SI time discretization \eref{4.8} is AP according to Definition \ref{def41}, provided that $\dt$ is computed as in
	\eref{4.10}.
\end{theorem}
\begin{proof}
	We begin by formally showing that the computed solution $\rho^{n+1}$, $\bm u^{n+1}$, $p^{n+1}$ admits an expansion of the type \eref{4.1}
	satisfying \eref{3.2f} and \eref{4.2} in the limit as $\ve\to0$. We substitute the corresponding expansion of the numerical solution at 
	time $t^n$ into the scheme \eref{4.8} and use \eref{3.2f}--\eref{4.2} to obtain
	\begin{align}
		&\begin{aligned}
			\rho^{n+1}&=\rho^{(0),n}+\ve\rho^{(1),n}+\ve^2\rho^{(2),n}-\dt\bnabla\!\cdot\!\big(\rho^{(0),n}\bm u^{(0),n}\big)\\
			&-\ve \dt \big[\bnabla\!\cdot\!\big(\rho^{(1),n}\bm u^{(0),n}\big)+\bnabla\!\cdot\!\big(\rho^{(0),n}\bm u^{(1),n}\big)\big]\\
			&-\ve^2 \dt \big[\bnabla\!\cdot\!\big(\rho^{(2),n}\bm u^{(0),n}\big)+\bnabla\!\cdot\!\big(\rho^{(1),n}\bm u^{(1),n}\big)+\bnabla\!\cdot\!\big(\rho^{(0),n}\bm u^{(2),n}\big)\big]+{\cal O}(\ve^3),
			\label{4.11}
		\end{aligned}\\
		&\begin{aligned}
			\bm u^{n+1}&=\bm u^{(0),n}+\ve\bm u^{(1),n}+\ve^2\bm u^{(2),n}-\dt\Big[(\bm u^{(0),n}\!\cdot\!\bnabla)\bm u^{(0),n}+
			\frac{\rho_{\max}^n-\rho^{(0),n}}{\rho^{(0),n}\rho_{\max}^n}\,\nabla p^{(2),n}\Big]\\
			&-\ve\dt\Big[(\bm u^{(1),n}\!\cdot\!\bnabla)\bm u^{(0),n}+(\bm u^{(0),n}\!\cdot\!\bnabla)\bm u^{(1),n}-
			\frac{\rho^{(1),n}}{\rho^{(0),n}\rho_{\max}^n}\,\nabla p^{(2),n}-\frac{\rho^{(0),n}}{\rho^{(0),n}\rho_{\max}^n}\,\nabla p^{(3),n}\Big]\\
			&-\ve^2\dt\Big[(\bm u^{(2),n}\!\cdot\!\bnabla)\bm u^{(0),n}+(\bm u^{(1),n}\!\cdot\!\bnabla)\bm u^{(1),n}+
			(\bm u^{(0),n}\!\cdot\!\bnabla)\bm u^{(2),n}\\
			&\hspace*{-0.3cm}
			-\frac{\rho^{(2),n}}{\rho^{(0),n}\rho_{\max}^n}\,\nabla p^{(2),n}-\frac{\rho^{(1),n}}{\rho^{(0),n}\rho_{\max}^n}\,\nabla p^{(3),n}
			-\frac{\rho^{(0),n}}{\rho^{(0),n}\rho_{\max}^n}\,\nabla p^{(4),n}\Big]
			-\frac{\dt}{\ve^2 \rho_{\max}^n}\,\nabla p^{n+1}+{\cal O}(\ve^3),\label{4.12}
		\end{aligned}\\[0.8ex]
		&p^{n+1}=\,\xbar p_0+\ve\,\xbar p_1+\ve^2 p^{(2),n}-\ve^2 \dt \bm u^{(0),n}\!\cdot\!\nabla p^{(2),n} -\dt\gamma p_{\min}^n\bnabla\!\cdot\!\bm u^{n+1}+{\cal O}(\ve^3).\label{4.13}
	\end{align}
	Thanks to the explicit nature of the density update in \eref{4.11}, we conclude that indeed $\rho^{n+1}$ admits the required asymptotic
	expansion $\rho^{n+1}=\rho^{(0),n+1}+\ve\rho^{(1),n+1}+\ve\rho^{(2),n+1}+\dots$, where the different terms of the expansion are obtained by
	collecting corresponding powers of $\ve$:
	\begin{equation}
		\begin{aligned}
			\rho^{(0),n+1}&=\rho^{(0),n}-\dt\,\bnabla\!\cdot\!(\rho^{(0),n}\bm u^{(0),n}),\\
			\rho^{(1),n+1}&=\rho^{(1),n}-\dt \big[\bnabla\!\cdot\!\big(\rho^{(1),n}\bm u^{(0),n}\big)+\bnabla\!\cdot\!\big(\rho^{(0),n}\bm u^{(1),n}\big)\big],\\
			\rho^{(2),n+1}&=\rho^{(2),n}-\dt \big[\bnabla\!\cdot\!\big(\rho^{(2),n}\bm u^{(0),n}\big)+\bnabla\!\cdot\!\big(\rho^{(1),n}\bm u^{(1),n}\big)+\bnabla\!\cdot\!\big(\rho^{(0),n}\bm u^{(2),n}\big)\big].
		\end{aligned}
		\label{3.18}
	\end{equation}
	In order to show that also $p^{n+1}$ admits an expansion of the same type, we take the divergence of the velocity equation \eref{4.12},
	substitute $\bnabla\!\cdot\!\muu^{n+1}$ into the pressure equation \eref{4.13}, and use the divergence-free assumption \eref{4.2} to obtain
	$$
	\begin{aligned}
		p^{n+1}&=\,\xbar p_0+\ve\,\xbar p_1+\ve^2p^{(2),n}-\ve^2\dt\bm u^{(0),n}\!\cdot\!\nabla p^{(2),n}\\
		&-\dt\gamma p_{\min}^n\bnabla\!\cdot\!\bm u^{(0),n}-\ve\dt\gamma p_{\min}^n\bnabla\!\cdot\!\bm u^{(1),n}-
		\ve^2\dt\gamma p_{\min}^n\bnabla\!\cdot\!\bm u^{(2),n}\\
		&+(\dt)^2\gamma p_{\min}^n\bnabla\!\cdot\!\Big[(\bm u^{(0),n}\!\cdot\!\bnabla)\bm u^{(0),n}+
		\frac{\rho_{\max}^n-\rho^{(0),n}}{\rho^{(0),n}\rho_{\max}^n}\,\nabla p^{(2),n}\Big]\\
		&+\ve(\dt)^2\gamma p_{\min}^n\bnabla\!\cdot\!\Big[(\bm u^{(1),n}\!\cdot\!\bnabla)\bm u^{(0),n}+(\bm u^{(0),n}\!\cdot\!\bnabla)\bm u^{(1),n}-
		\frac{\rho^{(1),n}}{\rho^{(0),n}\rho_{\max}^n}\,\nabla p^{(2),n}-\frac{\rho^{(0),n}}{\rho^{(0),n}\rho_{\max}^n}\,\nabla p^{(3),n}\Big]\\
		&+\ve^2 (\dt)^2\gamma p_{\min}^n\bnabla\!\cdot\!\Big[(\bm u^{(2),n}\!\cdot\!\bnabla)\bm u^{(0),n}+
		(\bm u^{(1),n}\!\cdot\!\bnabla)\bm u^{(1),n}+(\bm u^{(0),n}\!\cdot\!\bnabla)\bm u^{(2),n}\\
		&\hspace*{0.5cm}-\frac{\rho^{(2),n}}{\rho^{(0),n}\rho_{\max}^n}\,\nabla p^{(2),n}-
		\frac{\rho^{(1),n}}{\rho^{(0),n}\rho_{\max}^n}\,\nabla p^{(3),n}-\frac{\rho^{(0),n}}{\rho^{(0),n}\rho_{\max}^n}\,\nabla p^{(4),n}\Big]
		+\frac{(\dt)^2\gamma p_{\min}^n}{\ve^2 \rho_{\max}^n}\,\Delta p^{n+1}+{\cal O}(\ve^3),
	\end{aligned}
	$$
	which implies that $p^{n+1}$ is the solution of the elliptic equation
	$$
	-\Delta p^{n+1}+\frac{\ve^2\rho_{\max}^n}{(\dt)^2\gamma p_{\min}^n}\,p^{n+1}={\cal O}(\ve^2),
	$$
	subject to the boundary condition \eref{3.4f}. According to the theory of perturbed linear operators \cite{kato1966perturbation}, one can
	conclude that 
	\begin{align}
		&p^{n+1}=p^{(0),n+1}+\ve p^{(1),n+1}+\ve^2p^{(2),n+1}+\dots,\label{3.20f}\\
		&p^{(0),n+1}=p^{(0),n}\equiv\,\xbar p_0,\quad p^{(1),n+1}=p^{(1),n}\equiv\,\xbar p_1,\label{4.14}
	\end{align}
	and hence
	$$
	\nabla p^{n+1}=\ve^2\nabla p^{(2),n+1}+\ve^3\nabla p^{(3),n+1}+\ve^4\nabla p^{(4),n+1}+\dots,
	$$
	which we substitute into \eref{4.12} to obtain the velocity expansion
	$$
	\bm u^{n+1}=\bm u^{(0),n+1}+\ve\bm u^{(1),n+1}+\ve^2\bm u^{(2),n+1}+\dots
	$$
	with
	\begin{equation}
		\begin{aligned}
			\bm u^{(0),n+1}&=\bm u^{(0),n}-\dt\Big[(\bm u^{(0),n}\!\cdot\!\bnabla)\bm u^{(0),n}+
			\frac{\rho_{\max}^n-\rho^{(0),n}}{\rho^{(0),n}\rho_{\max}^n}\,\nabla p^{(2),n}\Big]-\frac{\dt}{\rho_{\max}^n}\,\nabla p^{(2),n+1},\\
			\bm u^{(1),n+1}&=\bm u^{(1),n}-\dt\Big[(\bm u^{(1),n}\!\cdot\!\bnabla)\bm u^{(0),n}+(\bm u^{(0),n}\!\cdot\!\bnabla)\bm u^{(1),n}\\
			&\hspace*{0.6cm}-\frac{\rho^{(1),n}}{\rho^{(0),n}\rho_{\max}^n}\,\nabla p^{(2),n}-
			\frac{\rho^{(0),n}}{\rho^{(0),n}\rho_{\max}^n}\,\nabla p^{(3),n}\Big]-\frac{\dt}{\rho_{\max}^n}\,\nabla p^{(3),n+1},\\
			\bm u^{(2),n+1}&=\bm u^{(2),n}-\dt\Big[(\bm u^{(2),n}\!\cdot\!\bnabla)\bm u^{(0),n}+(\bm u^{(1),n}\!\cdot\!\bnabla)\bm u^{(1),n}
			+(\bm u^{(0),n}\!\cdot\!\bnabla)\bm u^{(2),n}\\
			&\hspace*{0.6cm}-\frac{\rho^{(2),n}}{\rho^{(0),n}\rho_{\max}^n}\,\nabla p^{(2),n}
			-\frac{\rho^{(1),n}}{\rho^{(0),n}\rho_{\max}^n}\,\nabla p^{(3),n}-\frac{\rho^{(0),n}}{\rho^{(0),n}\rho_{\max}^n}\,\nabla p^{(4),n}\Big]-
			\frac{\dt}{\rho_{\max}^n}\,\nabla p^{(4),n+1}.
		\end{aligned}
		\label{3.22}
	\end{equation}
	
	We now need to show that \eref{3.2f} and \eref{4.2} hold for the updated solution, along with the consistency of the scheme \eref{4.8} with
	\eref{3.3} and \eref{3.4} as $\ve\to0$. We have already shown that \eref{3.2f} holds; see \eref{4.14}. The divergence-free conditions
	\eref{4.2} can be deduced from the pressure update \eref{4.13}, which in view of the obtained results yields
	$$
	\begin{aligned}
		\ve^2p^{(2),n+1}=\ve^2p^{(2),n}&-\ve^2\dt\bm u^{(0),n}\!\cdot\!\nabla p^{(2),n}-\dt\gamma p_{\min}^n\bnabla\!\cdot\!\bm u^{(0),n+1}\\
		&-\ve\dt\gamma p_{\min}^n\bnabla\!\cdot\!\bm u^{(1),n+1}-\ve^2\dt\gamma p_{\min}^n\bnabla\!\cdot\!\bm u^{(2),n+1}+{\cal O}(\ve^3).
	\end{aligned}
	$$
	Collecting the power-like terms of $\ve$, we deduce $\bnabla\!\cdot\!\bm u^{(0),n+1}=\bnabla\!\cdot\!\bm u^{(1),n+1}=0$.
	
	The consistency with \eref{3.3} immediately follows from the first equation in \eref{3.18}. To show the consistency with \eref{3.4}, we
	rewrite the first equation in \eref{3.22} as
	\begin{equation}
		\bm u^{(0),n+1}=\bm u^{(0),n}-\dt\bigg[(\bm u^{(0),n}\!\cdot\!\bnabla)\bm u^{(0),n}+\frac{\nabla p^{(2),n}}{\rho^{(0),n}}\bigg]
		-\frac{\dt}{\rho_{\max}^n}\left(\nabla p^{(2),n+1}-\nabla p^{(2),n}\right),
		\label{4.16}
	\end{equation}
	which is a consistent discretization of \eref{3.4}.
	
	We remark that since the zeroth and first modes of the pressure are constant, the evolution of the pressure in the zero-Mach-number limit
	essentially consists of the evolution of the second mode. Thus, $\nabla p^{(2),n+1}-\nabla p^{(2),n}\approx{\cal O}(\dt)$ and the last term
	in \eref{4.16}, in fact, represents a temporal diffusion term, which is proportional to ${\cal O}((\dt)^2)$. We also remark that according
	to \eref{4.10}, the time step $\dt$ is asymptotically independent of $\ve$.
\end{proof}

\subsubsection{Second-Order AP SI Time Discretization}\label{sec313}
We now introduce a second-order AP SI time discretization, which is based on the DeC approach, which was originally introduced in
\cite{fox1949some}. Our second-order AP SI-DeC time discretization is directly related to the IMEX-DeC methods presented in
\cite{offner2025analysis} and based on the DeC formulation introduced in \cite{Decremi}; see also
\cite{micalizzi2025efficient,micalizzi2024new}.

According to the second-order AP SI-DeC time discretization, the solution of \eref{4.47} is evolved from $t=t^n$ to $t=t^{n+1}$ through the
following two stages:
\begin{equation}
	\begin{aligned}
		&\bm V^*=\bm V^n-\dt\bm{{\cal R}}^n-\dt\bm{{\cal L}}^{n,*},\\
		&\bm V^{n+1}=\bm V^n-\frac{\dt}{2}\big[\bm{{\cal R}}^n+\bm{{\cal R}}^*\big]-\frac{\dt}{2}\big[\bm{{\cal L}}^{n,n}-\bm{{\cal L}}^{*,*}\big]
		-\dt\bm{{\cal L}}^{*,n+1},
	\end{aligned}
	\label{3.19}
\end{equation}
where the upper index $*$ is associated with the intermediate solution $\bm V^*$, and the definitions of the operators $\bm{{\cal R}}^*$ and
$\bm{{\cal L}}^{n,n}$, $\bm{{\cal L}}^{*,*}$, and $\bm{{\cal L}}^{*,n+1}$ are analogous to those given in \eref{3.12f} and \eref{3.13f},
respectively.

The scheme \eref{3.19} can be equivalently written as
\begin{equation}
	\begin{aligned}
		\rho^*=&~\rho^n-\dt\bnabla\!\cdot\!(\rho^n\bm u^n),\\
		\bm u^*=&~\bm u^n-\dt\Big[(\bm u^n\!\cdot\!\bnabla)\bm u^n+\frac{\rho_{\max}^n-\rho^n}{\ve^2\rho^n\rho_{\max}^n}\,\nabla p^n\Big]-
		\frac{\dt}{\ve^2\rho_{\max}^n}\nabla p^*,\\
		p^*=&~p^n-\dt\big[\bm u^n\!\cdot\!\nabla p^n+\gamma(p^n-p_{\min}^n)\bnabla\!\cdot\!\bm u^n\big]-\dt\gamma p_{\min}^n\bnabla\!\cdot\!\bm u^*,
	\end{aligned}
	\label{3.20}
\end{equation}
and
\begin{equation}
	\begin{aligned}
		\rho^{n+1}&=\rho^n-\frac{\dt}{2}\big[\bnabla\!\cdot\!(\rho^n\bm u^n)+\bnabla\!\cdot\!(\rho^*\bm u^*)\big],\\
		\bm u^{n+1}&=\bm u^n-\frac{\dt}{2}\Big[(\bm u^n\!\cdot\!\bnabla)\bm u^n+(\bm u^*\!\cdot\!\bnabla)\bm u^*
		+\frac{\rho_{\max}^n-\rho^n}{\ve^2\rho^n\rho_{\max}^n}\,\nabla p^n+\frac{\rho_{\max}^*-\rho^*}{\ve^2\rho^*\rho_{\max}^*}\,\nabla p^*\Big]\\
		&\hspace*{1.05cm}-\frac{\dt}{2\ve^2}\left(\frac{\nabla p^n}{\rho^n_{\max}}-\frac{\nabla p^*}{\rho^*_{\max}}\right)-
		\frac{\dt}{\ve^2\rho_{\max}^*}\nabla p^{n+1},\\[0.5ex]
		p^{n+1}&=p^n-\frac{\dt}{2}\big[\bm u^n\!\cdot\!\nabla p^n+\bm u^*\!\cdot\!\nabla p^*+\gamma(p^n-p_{\min}^n)\bnabla\!\cdot\!\bm u^n
		+\gamma(p^*-p_{\min}^*)\bnabla\!\cdot\!\bm u^*\big]\\[0.5ex]
		&\hspace*{1.0cm}-\frac{\dt}{2}\,\gamma\big(p_{\min}^n\bnabla\!\cdot\!\bm u^n-p_{\min}^*\bnabla\!\cdot\!\bm u^*\big)
		-\dt\gamma p_{\min}^*\bnabla\!\cdot\!\bm u^{n+1}.
	\end{aligned}
	\label{3.21}
\end{equation}

This second-order time discretization is indeed AP as shown in the next theorem.
\begin{theorem}
	The second-order SI-DeC discretization \eref{3.20}--\eref{3.21} is AP according to Definition \ref{def41}, provided that $\dt$ is computed
	according to \eref{4.10}.
\end{theorem}
\begin{proof}
	The proof proceeds along the same lines and uses the same arguments as in the proof of Theorem \ref{thm31}.
	
	We begin by observing that the first stage of the second-order SI-DeC discretization coincides with the first-order AP SI time 
	discretization studied before. Therefore, according to Theorem \ref{thm31}, the intermediate solution $\rho^*$, $\bm u^*$, $p^*$, obtained 
	by \eref{3.20} admits an expansion of the type \eref{4.1}, that is,
	\begin{equation}
		\begin{aligned}
			\rho^*&=\rho^{(0),*}+\ve\rho^{(1),*}+\ve^2\rho^{(2),*}+\dots,\quad\bm u^*=\bm u^{(0),*}+\ve\bm u^{(1),*}+\ve^2\bm u^{(2),*}+\dots,\\
			p^*&=p^{(0),*}+\ve p^{(1),*}+\ve^2p^{(2),*}+\dots,
		\end{aligned}
		\label{3.27}
	\end{equation}
	with
	\begin{equation}
		p^{(0),*}\equiv\,\xbar p_0,\quad p^{(1),*}\equiv\,\xbar p_1,\quad\bnabla\!\cdot\!\bm u^{(0),*}=0,\quad\bnabla\!\cdot\!\bm u^{(1),*}=0.
		\label{3.28}
	\end{equation}
	We then substitute the expansions \eref{4.1} and \eref{3.27} into \eref{3.21} and use the conditions \eref{3.2f}--\eref{4.2} and \eref{3.28}
	to obtain
	\begin{align}
		&\hspace*{-0.1cm}\begin{aligned}
			\rho^{n+1}&=\rho^{(0),n}+\ve\rho^{(1),n}+\ve^2\rho^{(2),n}-\frac{\dt}{2}\Big[\bnabla\!\cdot\!\big(\rho^{(0),n}\bm u^{(0),n}\big)+
			\bnabla\!\cdot\!\big(\rho^{(0),*}\bm u^{(0),*}\big)\Big]\\
			&\hspace*{-0.3cm}-\frac{\ve\dt}{2}\Big[\bnabla\!\cdot\!\big(\rho^{(1),n}\bm u^{(0),n}\big)+
			\bnabla\!\cdot\!\big(\rho^{(0),n}\bm u^{(1),n}\big)+\bnabla\!\cdot\!\big(\rho^{(1),*}\bm u^{(0),*}\big)+
			\bnabla\!\cdot\!\big(\rho^{(0),*}\bm u^{(1),*}\big)\Big]\\
			&\hspace*{-0.3cm}-\frac{\ve^2\dt}{2}\Big[\bnabla\!\cdot\!\big(\rho^{(2),n}\bm u^{(0),n}\big)+
			\bnabla\!\cdot\!\big(\rho^{(1),n}\bm u^{(1),n}\big)+\bnabla\!\cdot\!\big(\rho^{(0),n}\bm u^{(2),n}\big)\\
			&\hspace*{0.85cm}+\bnabla\!\cdot\!\big(\rho^{(2),*}\bm u^{(0),*}\big)+\bnabla\!\cdot\!\big(\rho^{(1),*}\bm u^{(1),*}\big)+
			\bnabla\!\cdot\!\big(\rho^{(0),*}\bm u^{(2),*}\big)\Big]+{\cal O}(\ve^3),\end{aligned}\label{3.24}\\
		&\hspace*{-0.1cm}\begin{aligned}
			\bm u^{n+1}&=\bm u^{(0),n}+\ve\bm u^{(1),n}+\ve^2\bm u^{(2),n}\\
			&\hspace*{-0.3cm}-\frac{\dt}{2}\Big[(\bm u^{(0),n}\!\cdot\!\bnabla)\bm u^{(0),n}+(\bm u^{(0),*}\!\cdot\!\bnabla)\bm u^{(0),*}+
			\frac{\rho_{\max}^n-\rho^{(0),n}}{\rho^{(0),n}\rho_{\max}^n}\,\nabla p^{(2),n}+
			\frac{\rho_{\max}^*-\rho^{(0),*}}{\rho^{(0),*}\rho_{\max}^*}\,\nabla p^{(2),*}\Big]\\
			&\hspace*{-0.3cm}-\frac{\ve\dt}{2}\Big[(\bm u^{(1),n}\!\cdot\!\bnabla)\bm u^{(0),n}+(\bm u^{(0),n}\!\cdot\!\bnabla)\bm u^{(1),n}+
			(\bm u^{(1),*}\!\cdot\!\bnabla)\bm u^{(0),*}+(\bm u^{(0),*}\!\cdot\!\bnabla)\bm u^{(1),*}\\
			&\hspace*{0.95cm}-\frac{\rho^{(1),n}}{\rho^{(0),n}\rho_{\max}^n}\,\nabla p^{(2),n}-
			\frac{\rho^{(0),n}}{\rho^{(0),n}\rho_{\max}^n}\,\nabla p^{(3),n}-\frac{\rho^{(1),*}}{\rho^{(0),*}\rho_{\max}^n}\,\nabla p^{(2),*}-
			\frac{\rho^{(0),*}}{\rho^{(0),*}\rho_{\max}^*}\,\nabla p^{(3),*}\Big]\\
			&\hspace*{-0.3cm}-\frac{\ve^2\dt}{2}\Big[(\bm u^{(2),n}\!\cdot\!\bnabla)\bm u^{(0),n}+(\bm u^{(1),n}\!\cdot\!\bnabla)\bm u^{(1),n}+
			(\bm u^{(0),n}\!\cdot\!\bnabla)\bm u^{(2),n}\\
			&\hspace*{0.85cm}+(\bm u^{(2),*}\!\cdot\!\bnabla)\bm u^{(0),*}+(\bm u^{(1),*}\!\cdot\!\bnabla)\bm u^{(1),*}+
			(\bm u^{(0),*}\!\cdot\!\bnabla)\bm u^{(2),*}\\
			&\hspace*{0.95cm}-\frac{\rho^{(2),n}}{\rho^{(0),n}\rho_{\max}^n}\,\nabla p^{(2),n}-
			\frac{\rho^{(1),n}}{\rho^{(0),n}\rho_{\max}^n}\,\nabla p^{(3),n}-\frac{\rho^{(0),n}}{\rho^{(0),n}\rho_{\max}^n}\,\nabla p^{(4),n}\\
			&\hspace*{0.95cm}-\frac{\rho^{(2),*}}{\rho^{(0),*}\rho_{\max}^*}\,\nabla p^{(2),*}-
			\frac{\rho^{(1),*}}{\rho^{(0),*}\rho_{\max}^*}\,\nabla p^{(3),*}-\frac{\rho^{(0),*}}{\rho^{(0),*}\rho_{\max}^*}\,\nabla p^{(4),*}\Big]\\
			&\hspace*{-0.3cm}-\frac{\dt}{2}\bigg[\frac{\nabla p^{(2),n}}{\rho_{\max}^n}-\frac{\nabla p^{(2),*}}{\rho_{\max}^*}\bigg]
			-\frac{\ve\dt}{2}\bigg[\frac{\nabla p^{(3),n}}{\rho_{\max}^n}-\frac{\nabla p^{(3),*}}{\rho_{\max}^*}\bigg] 
			-\frac{\ve^2\dt}{2}\bigg[\frac{\nabla p^{(4),n}}{\rho_{\max}^n}-\frac{\nabla p^{(4),*}}{\rho_{\max}^*}\bigg]\\
			&\hspace*{-0.3cm}-\frac{\dt}{\ve^2\rho_{\max}^*}\,\nabla p^{n+1}+{\cal O}(\ve^3),\end{aligned}\label{3.25}\\
		&\hspace*{-0.1cm}\begin{aligned}
			p^{n+1}&=\,\xbar p_0+\ve\,\xbar p_1+\ve^2 p^{(2),n}-\frac{\ve^2\dt}{2}\Big[\bm u^{(0),n}\!\cdot\!\nabla p^{(2),n}+
			\bm u^{(0),*}\!\cdot\!\nabla p^{(2),*}\Big]\\
			&-\frac{\ve^2\dt\gamma}{2}\Big[p_{\min}^n\bnabla\!\cdot\!\bm u^{(2),n}-p_{\min}^*\bnabla\!\cdot\!\bm u^{(2),*}\Big]-
			\dt\gamma p_{\min}^*\bnabla\!\cdot\!\bm u^{n+1}+{\cal O}(\ve^3).\end{aligned}\label{3.26}
	\end{align}
	
	The explicit nature of the density update \eref{3.24} implies that $\rho^{n+1}$ admits the required expansion
	$\rho^{n+1}=\rho^{(0),n+1}+\ve\rho^{(1),n+1}+\ve^2\rho^{(2),n+1}+\dots$ with $\rho^{(0),n+1}$ satisfying
	\begin{equation}
		\rho^{(0),n+1}=\rho^{(0),n}-\frac{\dt}{2}\Big[\bnabla\!\cdot\!\big(\rho^{(0),n}\bm u^{(0),n}\big)+
		\bnabla\!\cdot\!\big(\rho^{(0),*}\bm u^{(0),*}\big)\Big],
		\label{3.32ff}
	\end{equation}
	and other coefficients satisfying the equations, which can be easily obtained by grouping the corresponding powers of $\ve$.
	
	As in the proof of Theorem \ref{thm31}, we show that $p^{n+1}$ admits the asymptotic expansion by proving that it satisfies a well-posed
	elliptic problem with suitable boundary conditions. Taking the divergence of the velocity equation \eref{3.25} and substituting
	$\bnabla\!\cdot\!\muu^{n+1}$ into the pressure equation \eref{3.26} yields
	$$
	-\Delta p^{n+1}+\frac{\ve^2\rho_{\max}^*}{(\dt)^2\gamma p_{\min}^*}\,p^{n+1}={\cal O}(\ve^2).
	$$
	This together with the boundary conditions \eref{3.2}, results in the same expansion for $p^{n+1}$, which we have established in
	\eref{3.20f}--\eref{4.14} for the first-order SI method, leading to
	\begin{equation}
		\nabla p^{n+1}=\ve^2\left[\nabla p^{(2),n+1}+\ve\nabla p^{(3),n+1}+\ve^2\nabla p^{(4),n+1}+{\cal O}(\ve^3)\right].
		\label{3.32f}
	\end{equation}
	Next, we substitute \eref{3.32f} into the the velocity equation \eref{3.25} and a straightforward grouping of the power-like terms of $\ve$
	gives the equations for the coefficients of the velocity expansion
	$\bm u^{n+1}=\bm u^{(0),n+1}+\ve\bm u^{(1),n+1}+\ve^2\bm u^{(2),n+1}+\dots$. The equation for $\bm u^{(0),n+1}$ is
	\begin{equation}
		\begin{aligned}
			\bm u^{(0),n+1}&=\bm u^{(0),n}-\frac{\dt}{2}\bigg[(\bm u^{(0),n}\!\cdot\!\bnabla)\bm u^{(0),n}+(\bm u^{(0),*}\!\cdot\!\bnabla)\bm u^{(0),*}+
			\frac{\nabla p^{(2),n}}{\rho^{(0),n}}+\frac{\nabla p^{(2),*}}{\rho^{(0),*}}\bigg]\\
			&-\frac{\dt}{\rho_{\max}^*}\left(\nabla p^{(2),n+1}-\nabla p^{(2),*}\right),
		\end{aligned}
		\label{3.34}
	\end{equation}
	and the other equations can be obtained similarly.
	
	Let us now show the consistency with the asymptotic limit. The required conditions \eref{4.14} on the pressure modes have been already
	shown. The divergence-free conditions for the velocity modes are then established from the pressure update \eref{3.26}, which becomes
	$$
	\begin{aligned}
		\ve^2p^{(2),n+1}&=\ve^2 p^{(2),n}-\frac{\ve^2\dt}{2}\Big[\bm u^{(0),n}\!\cdot\!\nabla p^{(2),n}+
		\bm u^{(0),*}\!\cdot\!\nabla p^{(2),*}\Big]-
		\frac{\ve^2\dt\gamma}{2}\Big[p_{\min}^n\bnabla\!\cdot\!\bm u^{(2),n}-p_{\min}^*\bnabla\!\cdot\!\bm u^{(2),*}\Big]\\
		&-\dt\gamma p_{\min}^*\bnabla\!\cdot\!\bm u^{(0),n+1}-\ve\dt\gamma p_{\min}^*\bnabla\!\cdot\!\bm u^{(1),n+1}-
		\ve^2\dt\gamma p_{\min}^*\bnabla\!\cdot\!\bm u^{(2),n+1}+{\cal O}(\ve^3).
	\end{aligned}
	$$
	It is clear that the ${\cal O}(1)$ and ${\cal O}(\ve)$ terms here vanish, that is,
	$\bnabla\!\cdot\!\bm u^{(0),n+1}=\bnabla\!\cdot\!\bm u^{(1),n+1}=0$.
	
	Finally, we notice that \eref{3.32ff} and \eref{3.34} are consistent discretizations of \eref{3.3} and \eref{3.4} with the last term in
	\eref{3.34} representing a temporal diffusion, which is consistent with the order of accuracy of the scheme and thus proportional to
	${\cal O}((\dt)^3)$.
\end{proof}
\begin{remark}
	The described AP SI-DeC time discretization can be extended to arbitrarily high order in a straightforward way within the DeC framework. 
	For the sake of brevity, we restrict our consideration to the second order of accuracy, which matches the accuracy that will be used in the
	spatial discretization discussed in \S\ref{sec32}.
\end{remark}

\subsection{Fully Discrete Second-Order AP Scheme for the Primitive System}\label{sec32}
In this section, we construct a fully discrete scheme based on the second-order AP SI time discretization presented in \S\ref{sec313}. To
this end, we first introduce uniform Cartesian cells $I_{j,k}:=[x_\jmh,x_\jph]\times[y_\kmh,y_\kph]$ with $x_\jph-x_\jmh\equiv\dx$ and
$y_\kph-y_\kmh\equiv\dy$, centered at $(x_j,y_k)$ with $x_j=\big(x_\jmh+x_\jph\big)/2$ and $y_k=\big(y_\kmh+y_\kph\big)/2$, and assume that
the cell averages $\,\xbar{\bm V}_{j,k}^{\,n}:\approx\frac{1}{\dx\dy}\iint_{I_{j,k}}\bm V(x,y,t^n)\,{\rm d}x{\rm d}y\,$ are available at
time $t^n$.

The fully discrete FV version of the second-order AP scheme \eref{3.19} reads as
\begin{align}
	&\xbar{\bm V}^{\,*}_{j.k}=\,\xbar{\bm V}^{\,n}_{j.k}-\dt\bm{{\cal R}}^n_{j.k}-\dt\bm{{\cal L}}^{n,*}_{j.k},\label{3.31a}\\
	&\xbar{\bm V}^{\,n+1}_{j.k}=\,\xbar{\bm V}^{\,n}_{j.k}-\frac{\dt}{2}\big[\bm{{\cal R}}^n_{j.k}+\bm{{\cal R}}^*_{j.k}\big]-
	\frac{\dt}{2}\big[\bm{{\cal L}}^{n,n}_{j.k}-\bm{{\cal L}}^{*,*}_{j.k}\big]-\dt\bm{{\cal L}}^{*,n+1}_{j.k},\label{3.31b}
\end{align}
where $\bm{{\cal R}}^n_{j.k}$ and $\bm{{\cal R}}^*_{j.k}$ are obtained using the PCCU discretization from \cite{ACKM,CKN22}, which is a
low-dissipation generalization of the PCCU discretization from \cite{diaz2019path}, while $\bm{{\cal L}}^{n,*}_{j.k}$,
$\bm{{\cal L}}^{n,n}_{j.k}$, $\bm{{\cal L}}^{*,*}_{j.k}$, and $\bm{{\cal L}}^{*,n+1}_{j.k}$ are obtained using central differences. In what
follows, for the sake of brevity, we provide details on $\bm{{\cal R}}^n_{j.k}$ and $\bm{{\cal L}}^{n,*}_{j.k}$ only, whereas the remaining
discretizations are obtained in a similar manner.

We begin with
\begin{equation}
	\begin{aligned}
		\bm{{\cal R}}_{j,k}^n:=&\,\frac{1}{\dx}\Bigg[\widetilde{\bm{{\cal F}}}_{\jph,k}^n-\widetilde{\bm{{\cal F}}}_{\jmh,k}^n-
		\widetilde{\bm B}_{j,k}^n-\frac{\tilde a^{+,n}_{\jmh,k}\widetilde{\bm B}_{\bm\Psi,\jmh,k}^n}
		{\tilde a^{+,n}_{\jmh,k}-\tilde a^{-,n}_{\jmh,k}}+
		\frac{\tilde a^{-,n}_{\jph,k}\widetilde{\bm B}_{\bm\Psi,\jph,k}^n}{\tilde a^{+,n}_{\jph,k}-\tilde a^{-,n}_{\jph,k}}\Bigg]\\
		+&\,\frac{1}{\dy}\Bigg[\widetilde{\bm{{\cal G}}}_{j,\kph}^n-\widetilde{\bm{{\cal G}}}_{j,\kmh}^n-\widetilde{\bm C}_{j,k}^n
		-\frac{\tilde b^{+,n}_{j,\kmh}\widetilde{\bm C}_{\bm\Psi,j,\kmh}^n}{\tilde b^{+,n}_{j,\kmh}-\tilde b^{-,n}_{j,\kmh}}+
		\frac{\tilde b^{-,n}_{j,\kph}\widetilde{\bm C}_{\bm\Psi,j,\kph}^n}{\tilde b^{+,n}_{j,\kph}-\tilde b^{-,n}_{j,\kph}}\Bigg],
	\end{aligned}
	\label{3.32}
\end{equation}
where $\widetilde{\bm{{\cal F}}}_{\jph,k}^n$ and $\widetilde{\bm{{\cal G}}}_{j,\kph}^n$ are the CU numerical fluxes
$$
\begin{aligned}
	&\widetilde{\bm{{\cal F}}}_{\jph,k}^n:=\frac{\tilde a^{+,n}_{\jph,k}\widetilde{\bm F}\big(\bm V^{-,n}_{\jph,k}\big)-
		\tilde a^{-,n}_{\jph,k}\widetilde{\bm F}\big(\bm V^{+,n}_{\jph,k}\big)}{\tilde a^{+,n}_{\jph,k}-\tilde a^{+,n}_{\jph,k}}+
	\frac{\tilde a^{+,n}_{\jph,k}\tilde a^{-,n}_{\jph,k}}{\tilde a^{+,n}_{\jph,k}-\tilde a^{-,n}_{\jph,k}\,}
	\Big(\bm V^{+,n}_{\jph,k}-\bm V^{-,n}_{\jph,k}-\delta\bm V_{\jph,k}^n\Big),\\
	&\widetilde{\bm{{\cal G}}}_{j,\kph}^n:=\frac{\tilde b^{+,n}_{j,\kph}\widetilde{\bm G}\big(\bm V^{-,n}_{j,\kph}\big)-
		\tilde b^{-,n}_{j,\kph}\widetilde{\bm G}\big(\bm V^{+,n}_{j,\kph}\big)}{\tilde b^{+,n}_{j,\kph}-\tilde b^{-,n}_{j,\kph}}+
	\frac{\tilde b^{+,n}_{j,\kph}\tilde b^{-,n}_{j,\kph}}{\tilde b^{+,n}_{j,\kph}-\tilde b^{-,n}_{j,\kph}}\,
	\Big(\bm V^{+,n}_{j,\kph}-\bm V^{-,n}_{j,\kph}-\delta\bm V_{j,\kph}^n\Big),
\end{aligned}
$$
and $\bm V^{\pm,n}_{\jph,k}$ and $\bm V^{\pm,n}_{j,\kph}$ are reconstructed values of $\bm V$ at the midpoints of the cell interfaces,
$\delta\bm V_{\jph,k}^n$ and $\delta\bm V_{j,\kph}^n$ are ``built-in'' anti-diffusion terms, and $\tilde a^{\pm,n}_{\jph,k}$ and
$\tilde a^{\pm,n}_{j,\kph}$ denote the one-sided local propagation speeds of the nonstiff subsystem in the $x$- and $y$-direction,
respectively.

The point values
\begin{equation}
	\begin{aligned}
		&\bm V^{-,n}_{\jph,k}:=\,\xbar{\bm V}_{j,k}^{\,n}+\frac{\dx}{2}(\bm V_x)_{j,k}^n,&&
		\bm V^{+,n}_{\jph,k}:=\,\xbar{\bm V}_{j+1,k}^{\,n}-\frac{\dx}{2}(\bm V_x)_{j+1,k}^n,\\
		&\bm V^{-,n}_{j,\kph}:=\,\xbar{\bm V}_{j,k}^{\,n}+\frac{\dy}{2}(\bm V_y)_{j,k}^n,&&
		\bm V^{+,n}_{j,\kph}:=\,\xbar{\bm V}_{j,k+1}^{\,n}-\frac{\dy}{2}(\bm V_y)_{j,k+1}^n,
	\end{aligned}
	\label{3.33f}
\end{equation}
are computed using the piecewise linear reconstruction
$$
\xbar{\bm V}_{j,k}^{\,n}+(\bm V_x)_{j,k}^n(x-x_j)+(\bm V_y)_{j,k}^n(y-y_k),~~(x,y)\in I_{j,k},
$$
in which the slopes $(\bm V_x)_{j,k}^n$ and $(\bm V_y)_{j,k}^n$ are approximated using the generalized minmod limiter (see, e.g.,
\cite{LieNoe,nessyahu1990non,sweby1984high}):
\begin{equation}
	\begin{aligned}
		(\bV_x)_{j,k}^n&:={\rm minmod}\left(\theta\,\frac{\xbar{\bV}_{j,k}^{\,n}-\,\xbar{\bV}_{j-1,k}^{\,n}}{\dx},\,
		\frac{\xbar{\bV}_{j+1,k}^{\,n}-\,\xbar{\bV}_{j-1,k}^{\,n}}{2\dx},\,
		\theta\,\frac{\xbar{\bV}_{j+1,k}^{\,n}-\,\xbar{\bV}_{j,k}^{\,n}}{\dx}\right),\\
		(\bV_y)_{j,k}^n&:={\rm minmod}\left(\theta\,\frac{\xbar{\bV}_{j,k}^{\,n}-\,\xbar{\bV}_{j,k-1}^{\,n}}{\dy},\,
		\frac{\xbar{\bV}_{j,k+1}^{\,n}-\,\xbar{\bV}_{j,k-1}^{\,n}}{2\dy},\,
		\theta\,\frac{\xbar{\bV}_{j,k+1}^{\,n}-\,\xbar{\bV}_{j,k}^{\,n}}{\dy}\right),
	\end{aligned}
	\label{3.33}
\end{equation}
where the minmod function, defined by
\begin{equation*}
	{\rm minmod}(z_1,z_2,\ldots):=\left\{\begin{aligned}
		&\min(z_1,z_2,\ldots)&&\mbox{if}~z_i>0,~\forall i,\\
		&\max(z_1,z_2,\ldots)&&\mbox{if}~z_i<0,~\forall i,\\
		&\,0&&\mbox{otherwise},
	\end{aligned}\right.
\end{equation*}
is applied in a componentwise manner. The parameter $\theta\in[1,2]$ in \eref{3.33} is to be chosen to adjust the amount of numerical
dissipation present in the resulting scheme, with larger values of $\theta$ leading to sharper but, in general, more oscillatory solutions.

The one-sided local speeds of propagation are estimated using the smallest and largest eigenvalues of the matrices
$\frac{\partial\widetilde{\mF}}{\partial\bm V}(\bm V)-\widetilde B(\bm V)$ and
$\frac{\partial\widetilde{\mG}}{\partial\bm V}(\bm V)-\widetilde C(\bm V)$ as follows:
\begin{equation}
	\begin{aligned}
		&\tilde a^{-,n}_{\jph,k}:=\min\left\{u^{-,n}_{\jph,k}-\tilde c^{\,-,n}_{\jph,k},\,u^{+,n}_{\jph,k}-\tilde c^{\,+,n}_{\jph,k},\,-\delta
		\right\},\\
		&\tilde a^{+,n}_{\jph,k}:=\max\left\{u^{-,n}_{\jph,k}+\tilde c^{\,-,n}_{\jph,k},\,u^{+,n}_{\jph,k}+\tilde c^{\,+,n}_{\jph,k},\,
		\delta\right\},\\
		&\tilde b^{-,n}_{j,\kph}:=\min\left\{v^{-,n}_{j,\kph}-\tilde c^{\,-,n}_{j,\kph},\,v^{+,n}_{j,\kph}-\tilde c^{\,+,n}_{j,\kph},\,
		-\delta\right\},\\
		&\tilde b^{+,n}_{j,\kph}:=\max\left\{v^{-,n}_{j,\kph}+\tilde c^{\,-,n}_{j,\kph},\,v^{+,n}_{j,\kph}+\tilde c^{\,+,n}_{j,\kph},\,
		\delta\right\},
	\end{aligned}
	\label{4.49}
\end{equation}
where the sound speeds
$$
\tilde c^{\,\pm,n}_{\jph,k}:=\frac{1}{\ve}\sqrt{\gamma\frac{\big(\rho_{\max}^n-\rho^{\,\pm,n}_{\jph,k}\big)
		\big(p^{\,\pm,n}_{\jph,k}-p_{\min}^n)}{\rho^{\,\pm,n}_{\jph,k}\,\rho_{\max}^n}},\quad
\tilde c^{\,\pm,n}_{j,\kph}:=\frac{1}{\ve}\sqrt{\gamma\frac{\big(\rho_{\max}^n-\rho^{\,\pm,n}_{j,\kph}\big)
		\big(p^{\,\pm,n}_{j,\kph}-p_{\min}^n)}{\rho^{\,\pm,n}_{j,\kph}\,\rho_{\max}^n}}
$$
are computed using the following discrete versions of $\rho_{\max}^n$ and $p_{\min}^n$
\begin{equation}
	\rho_{\max}^n:=\max\limits_{j,k}\,\xbar{\rho}_{j,k}^{\,n},\quad p_{\min}^n:=\min\limits_{j,k}\,\xbar p_{j,k}^{\,n},
	\label{3.35}
\end{equation}
and $\delta$ is a small positive parameter introduced to prevent divisions by $0$ (we have taken $\delta:=10^{-15}$ in the numerical
experiments reported in \S\ref{sec6}).

The ``built-in'' anti-diffusion terms are
$$
\begin{aligned}
	&\delta\bm V_{\jph,k}^n:={\rm minmod}\Big(\bm V_{\jph,k}^{{\rm int},n}-\bm V^{-,n}_{\jph,k},\,
	\bm V^{+,n}_{\jph,k}-\bm V_{\jph,k}^{{\rm int},n}\Big),\\
	&\delta\bm V_{j,\kph}^n:={\rm minmod}\Big(\bm V_{j,\kph}^{{\rm int},n}-\bm V^{-,n}_{j,\kph},\,
	\bm V^{+,n}_{j,\kph}-\bm V_{j,\kph}^{{\rm int},n}\Big),
\end{aligned}
$$
where 
$$
\begin{aligned}
	&\bm V_{\jph,k}^{{\rm int},n}:=\frac{\tilde a^{+,n}_{\jph,k}\bm V^{+,n}_{\jph,k}-\tilde a^{-,n}_{\jph,k}\bm V^{-,n}_{\jph,k}
		-\widetilde{\bm F}\big(\bm V^{+,n}_{\jph,k}\big)+\widetilde{\bm F}\big(\bm V^{-,n}_{\jph,k}\big)}
	{\tilde a^{+,n}_{\jph,k}-\tilde a^{-,n}_{\jph,k}},\\[0.3ex]
	&\bm V_{j,\kph}^{{\rm int},n}:=\frac{\tilde b^{+,n}_{j,\kph}\bm V^{+,n}_{j,\kph}-\tilde b^{-,n}_{j,\kph}\bm V^{-,n}_{j,\kph}
		-\widetilde{\bm G}\big(\bm V^{+,n}_{j,\kph}\big)+\widetilde{\bm G}\big(\bm V^{-,n}_{j,\kph}\big)}
	{\tilde b^{+,n}_{j,\kph}-\tilde b^{-,n}_{j,\kph}}.
\end{aligned}
$$

Finally,
\begin{equation}
	\bm{{\cal L}}^{n,*}_{j.k}:=\Big(0,\,\frac{\,\xbar p_{j+1,k}^{\,*}-\,\xbar p_{j-1,k}^{\,*}}{2\dx\,\ve^2\rho_{\max}^n},\,
	\frac{\,\xbar p_{j,k+1}^{\,*}-\,\xbar p_{j,k-1}^{\,*}}{2\dy\,\ve^2\rho_{\max}^n},\,
	\gamma p_{\min}^n\bnabla\!\cdot\xbar{\bm u}_{j,k}^{\,*}\Big)^\top,
	\label{4.48b}
\end{equation}
where $\bnabla\!\cdot\xbar{\bm u}_{j,k}$ denotes the discrete divergence operator computed using second-order central differences:
\begin{equation}
	\bnabla\!\cdot\xbar{\bm u}_{j,k}:=\frac{\,\xbar u_{j+1,k}-\,\xbar u_{j-1,k}}{2\dx}+\frac{\,\xbar v_{j,k+1}-\,\xbar v_{j,k-1}}{2\dy}.
	\label{3.38f}
\end{equation}

\subsection{Semi-Discrete CU Scheme for the Conservative System}\label{sec42}
We now consider the conservative formulation \eref{2.1a}--\eref{2.1c}, which can be put into the following vector form:
\begin{equation}
	\begin{aligned}
		&\bm U_t+\mF(\bm U)_x+\mG(\bm U)_y=\bm0,\quad\bm U:=(\rho,\rho u,\rho v,E)^\top,\\
		&\mF(\bm U):=\Big(\rho u,\rho u^2+\frac{p}{\ve^2},\rho uv,u(E+p)\Big)^\top,\quad
		\mG(\bm U):=\Big(\rho v,\rho uv,\rho v^2+\frac{p}{\ve^2},v(E+p)\Big)^\top.
	\end{aligned}
	\label{3.39g}
\end{equation}

In the semi-discrete CU scheme, the cell averages
$\,\xbar{\bm U}_{j,k}(t):\approx\frac{1}{\dx\dy}\iint_{I_{j,k}}\bm U(x,y,t)\,{\rm d}x{\rm d}y$ are evolved in time by numerically solving
the following system of ODEs:
\begin{equation}
	\frac{{\rm d}}{{\rm d}t}\,\xbar{\bm U}_{j,k}=-\frac{\bm{{\cal F}}_{\jph,k}-\bm{{\cal F}}_{\jmh,k}}{\dx}-
	\frac{\bm{{\cal G}}_{j,\kph}-\bm{{\cal G}}_{j,\kmh}}{\dy},
	\label{2.15a}
\end{equation}
where $\bm{{\cal F}}_{\jph,k}$ and $\bm{{\cal G}}_{j,\kph}$ are the CU numerical fluxes from \cite{KLin} defined as
\begin{equation}
	\begin{aligned}
		\bm{{\cal F}}_{\jph,k}&:=\frac{a^+_{\jph,k}\bm F\big(\bm U^-_{\jph,k}\big)-a^-_{\jph,k}\bm F\big(\bm U^+_{\jph,k}\big)}
		{a^+_{\jph,k}-a^-_{\jph,k}}\\
		&+\frac{a^+_{\jph,k}a^-_{\jph,k}}{a^+_{\jph,k}-a^-_{\jph,k}}\,\Big(\bm U^+_{\jph,k}-\bm U^-_{\jph,k}-\delta\bm U_{\jph,k}\Big),\\[0.8ex]
		\bm{{\cal G}}_{j,\kph}&:=\frac{b^+_{j,\kph}\bm G\big(\bm U^-_{j,\kph}\big)-b^-_{j,\kph}\bm G\big(\bm U^+_{j,\kph}\big)}
		{b^+_{j,\kph}-b^-_{j,\kph}}\\
		&+\frac{b^+_{j,\kph}b^-_{j,\kph}}{b^+_{j,\kph}-b^-_{j,\kph}}\,\Big(\bm U^+_{j,\kph}-\bm U^-_{j,\kph}-\delta\bm U_{j,\kph}\Big).
	\end{aligned}
	\label{3.41}
\end{equation}
Here, the interface values $\bm U_{\jph,k}^\pm:=\bm U\big(\bm V_{\jph,k}^\pm\big)$ and
$\bm U_{j,\kph}^\pm:=\bm U\big(\bm V_{j,\kph}^\pm\big)$ are computed from the reconstructed primitive variables $\bm V_{\jph,k}^\pm$ and 
$\bm V_{j,\kph}^\pm$ (see \S\ref{sec32}) at the corresponding time level via a straightforward transformation $\bU(\bV)$ from $\bm V$ to
$\bm U$. The quantities $a_{\jph,k}^\pm$ and $b_{j,\kph}^\pm$ are the one-sided local speeds of propagation for the conservative system
\eref{3.39g} in the $x$- and $y$-direction, respectively. They are estimated using the largest and smallest eigenvalues of the corresponding
flux Jacobians as follows:
\begin{equation}
	\begin{aligned}
		\begin{aligned}
			&a^-_{\jph,k}:=\min\left\{u^-_{\jph,k}-c^-_{\jph,k},\,u^+_{\jph,k}-c^+_{\jph,k},\,-\delta\right\},\\
			&a^+_{\jph,k}:=\max\left\{u^-_{\jph,k}+c^-_{\jph,k},\,u^+_{\jph,k}+c^+_{\jph,k},\,\delta\right\},
		\end{aligned}\quad c^\pm_{\jph,k}:=\frac{1}{\ve}\sqrt{\frac{\gamma p^\pm_{\jph,k}}{\rho^\pm_{\jph,k}}},\\[1.5ex]
		\begin{aligned}
			&b^-_{j,\kph}:=\min\left\{v^-_{j,\kph}-c^-_{j,\kph},\,v^+_{j,\kph}-c^+_{j,\kph},\,-\delta\right\},\\
			&b^+_{j,\kph}:=\max\left\{v^-_{j,\kph}+c^-_{j,\kph},\,v^+_{j,\kph}+c^+_{j,\kph},\,\delta\right\},
		\end{aligned}\quad c^\pm_{j,\kph}:=\frac{1}{\ve}\sqrt{\frac{\gamma p^\pm_{j,\kph}}{\rho^\pm_{j,\kph}}},
	\end{aligned}
	\label{3.42}
\end{equation}
where $\delta:=10^{-15}$ is used to avoid divisions by $0$.

The ``built-in'' anti-diffusion terms are
\begin{equation}
	\begin{aligned}
		&\delta\bm U_{\jph,k}:={\rm minmod}\Big(\bm U_{\jph,k}^{\rm int}-\bm U^-_{\jph,k},\bm U^+_{\jph,k}-\bm U_{\jph,k}^{\rm int}\Big),\\[0.5ex]
		&\delta\bm U_{j,\kph}:={\rm minmod}\Big(\bm U_{j,\kph}^{\rm int}-\bm U^-_{j,\kph},\bm U^+_{j,\kph}-\bm U_{j,\kph}^{\rm int}\Big),
	\end{aligned}
	\label{3.43}
\end{equation}
with 
\begin{equation}
	\begin{aligned}
		&\bm U_{\jph,k}^{\rm int}:=\frac{a^+_{\jph,k}\bm U^+_{\jph,k}-a^-_{\jph,k}\bm U^-_{\jph,k}
			-\bm F(\bm U^+_{\jph,k})+\bm F(\bm U^-_{\jph,k})}{a^+_{\jph,k}-a^-_{\jph,k}},\\[0.5ex]
		&\bm U_{j,\kph}^{\rm int}:=\frac{b^+_{j,\kph}\bm U^+_{j,\kph}-b^-_{j,\kph}\bm U^-_{j,\kph}
			-\bm G\big(\bm U^+_{j,\kph}\big)+\bm G\big(\bm U^-_{j,\kph}\big)}{b^+_{j,\kph}-b^-_{j,\kph}}.
	\end{aligned}
	\label{3.44}
\end{equation}

Note that most of the indexed quantities in the semi-discrete setting above are time-dependent, but we have omitted this dependence to ease
the notation.

Finally, the system of ODEs \eref{2.15a} has to be integrated in time using an appropriate ODE solver. Its solution is performed
simultaneously with one of the primitive systems using the explicit counterpart of the SI-DeC scheme, and a post-processing is performed
at each stage, as explained in \S\ref{sec422}.

\subsection{Implementation Details}\label{sec422}
In our DF-FV approach, the solutions of the primitive and conservative systems are evolved simultaneously according to the following
algorithm.

\smallskip
\noindent
$\bullet\,$ \textbf{Step 1 (Compute $\xbar\rho^{\,*}_{j,k}$).} We use the $\rho$-equation in \eref{3.31a} to obtain
\begin{equation*}
	\xbar\rho^{\,*}_{j,k}=\,\xbar\rho^{\,n}_{j,k}-\dt({\cal R}^\rho)^n_{j,k}.
\end{equation*}

\smallskip
\noindent
$\bullet\,$ \textbf{Step 2 (Solve the linear elliptic equation for $\,\xbar p^{\,*}_{j,k}$).} We apply the discrete divergence operator
\eref{3.38f} to the $\bm u$-equations  in \eref{3.31a} and substitute them into the $p$-equation in \eref{3.31a} to obtain the following
linear system of algebraic equations for $\,\xbar p^{\,*}_{j,k}$, which is a discretization of the linear elliptic equation for $p^*$:
\begin{equation*}
	\xbar p^{\,*}_{j,k}-\frac{(\dt)^2\gamma p_{\min}^n}{\ve^2\rho_{\max}^n}\,\Delta\xbar p_{j,k}^{\,*}=\,\xbar p^{\,n}_{j,k}-
	\dt({\cal R}^p)^n_{j,k}-\dt\gamma p_{\min}^n\bnabla\!\cdot\xbar{\bm u}_{j,k}^{\,n}
	+(\dt)^2\gamma p_{\min}^n\bnabla\!\cdot\!(\bm{{\cal R}}^{\bm u})^n_{j,k},
	\label{3.39f}
\end{equation*}
where the discrete Laplacian $\Delta\xbar p_{j,k}$ is defined as
$$
\Delta\xbar p_{j,k}:=\frac{\,\xbar p_{j-1,k}-2\,\xbar p_{j,k}+\,\xbar p_{j+1,k}}{(\dx)^2}+
\frac{\,\xbar p_{j,k-1}-2\,\xbar p_{j,k}+\,\xbar p_{j,k+1}}{(\dy)^2}.
$$

\smallskip
\noindent
$\bullet\,$ \textbf{Step 3 (Compute $\,\xbar{\bm u}^{\,*}_{j,k}$).} Once $\,\xbar p^{\,*}_{j,k}$ is available, we use the $\bm u$-equations
in \eref{3.31a} to obtain
\begin{equation*}
	\xbar{\bm u}^{\,*}_{j,k}=\,\xbar{\bm u}^{\,n}_{j,k}-\dt(\bm{{\cal R}}^{\bm u})^n_{j,k}-\dt(\bm{{\cal L}}^{\bm u})^{n,*}_{j,k}.
\end{equation*}

\smallskip
\noindent
$\bullet\,$ \textbf{Step 4 (Compute $\,\xbar{\bm U}^{\,*}_{j,k}$).} 
We perform the conservative update with the explicit counterpart of the SI-DeC scheme to obtain the solution $\,\xbar{\bm U}^{\,*}_{j,k}$ 
at the intermediate stage
\begin{equation}
	\xbar{\bm U}^{\,*}_{j,k}=\,\xbar{\bm U}^{\,n}_{j,k}-\dt\left[\frac{\bm{{\cal F}}_{\jph,k}^n-\bm{{\cal F}}_{\jmh,k}^n}{\dx}+
	\frac{\bm{{\cal G}}_{j,\kph}^n-\bm{{\cal G}}_{j,\kmh}^n}{\dy}\right],
	\label{3.50}
\end{equation}
and then post-process the primitive solution by replacing $\xbar{\bm V}^{\,*}_{j,k}$ with
$r\big(\bm V\big(\,\xbar{\bm U}^{\,*}_{j,k}\big),\,\xbar{\bm V}^{\,*}_{j,k}\big)$; see \S\ref{sec43}.

\smallskip
\noindent
$\bullet\,$ \textbf{Step 5 (Compute $\xbar\rho^{\,n+1}_{j,k}$).} We solve the $\rho$-equation in \eref{3.31b} to obtain
\begin{equation*}
	\xbar\rho^{\,n+1}_{j,k}=\,\xbar\rho^{\,n}_{j,k}-\frac{\dt}{2}\big[({\cal R}^\rho)^n_{j,k}+({\cal R}^\rho)^*_{j,k}\big].
\end{equation*}

\smallskip
\noindent
$\bullet\,$ \textbf{Step 6 (Solve the linear elliptic equation for $\,\xbar p^{\,n+1}_{j,k}$).} We apply the discrete divergence operator
\eref{3.38f} to the $\bm u$-equations in \eref{3.31b} and substitute them into the $p$-equation in \eref{3.31b} to obtain the following
linear system of algebraic equations for $\,\xbar p^{\,n+1}_{j,k}$, which is a discretization of the linear elliptic equation for $p^{n+1}$:
\begin{equation*}
	\begin{aligned}
		\xbar p^{\,n+1}_{j,k}&-\frac{(\dt)^2\gamma p_{\min}^*}{\ve^2\rho_{\max}^*}\,\Delta\xbar p_{j,k}^{\,n+1}=\,\xbar p^{\,n}_{j,k}-
		\frac{\dt}{2}\left[({\cal R}^p)_{j,k}^n+({\cal R}^p)_{j,k}^*\right]-
		\frac{\dt}{2}\left[({\cal L}^p)^{n,n}_{j,k}-({\cal L}^p)^{*,*}_{j,k}\right]\\
		&-\dt\gamma p_{\min}^*\bnabla\!\cdot\xbar{\bm u}_{j,k}^{\,n}
		+\frac{(\dt)^2\gamma p_{\min}^*}{2}\bnabla\!\cdot\!\left[(\bm{{\cal R}}^{\bm u})^n_{j,k}+(\bm{{\cal R}}^{\bm u})^*_{j,k}\right]
		-\frac{(\dt)^2\gamma p_{\min}^*}{2}\bnabla\!\cdot\!\left[(\bm{{\cal L}}^{\uvec{u}})^{n,n}_{j,k}-(\bm{{\cal L}}^{\uvec{u}})^{*,*}_{j,k}\right].
	\end{aligned}
	\label{3.40f}
\end{equation*}

\smallskip
\noindent
$\bullet\,$ \textbf{Step 7 (Compute $\,\xbar{\bm u}^{\,n+1}_{j,k}$).} Once $\,\xbar p^{\,n+1}_{j,k}$ is available, we compute
\begin{equation*}
	\begin{aligned}
		\xbar{\bm u}^{\,n+1}_{j,k}=
		\,\xbar{\bm u}^{\,n}_{j,k}-\frac{\dt}{2}\left[(\bm{{\cal R}}^{\bm u})^n_{j,k}+(\bm{{\cal R}}^{\bm u})^*_{j,k}\right]
		-\frac{\dt}{2}\left[(\bm{{\cal L}}^{\uvec{u}})^{n,n}_{j,k}
		-(\bm{{\cal L}}^{\uvec{u}})^{*,*}_{j,k}\right]-\dt(\bm{{\cal L}}^{\bm u})^{*,n+1}_{j,k}.
	\end{aligned}
\end{equation*}

\smallskip
\noindent
$\bullet\,$ \textbf{Step 8 (Compute $\,\xbar{\bm U}^{\,n+1}_{j,k}$).} Finally, we use the explicit part of the SI-DeC scheme to evaluate  
\begin{equation}
	\begin{aligned}
		\xbar{\bm U}^{\,n+1}_{j,k}=\,\xbar{\bm U}^{\,n}_{j,k}&-\frac{\dt}{2}\left[\frac{\bm{{\cal F}}_{\jph,k}^n-\bm{{\cal F}}_{\jmh,k}^n}{\dx}+
		\frac{\bm{{\cal F}}_{\jph,k}^{*}-\bm{{\cal F}}_{\jmh,k}^{*}}{\dx}\right]\\
		&-\frac{\dt}{2}\left[\frac{\bm{{\cal G}}_{j,\kph}^n-\bm{{\cal G}}_{j,\kmh}^n}{\dy}+
		\frac{\bm{{\cal G}}_{j,\kph}^{*}-\bm{{\cal G}}_{j,\kmh}^{*}}{\dy}\right].
	\end{aligned}
	\label{3.51}
\end{equation}
and then post-process of the primitive solution by replacing $\xbar{\bm V}^{\,n+1}_{j,k}$ with
$r\big(\bm V\big(\,\xbar{\bm U}^{\,n+1}_{j,k}\big),\,\xbar{\bm V}^{\,n+1}_{j,k}\big)$; see \S\ref{sec43}.

We recall that the interface values of $\bm U$ needed for the computation of the numerical fluxes in \eref{3.50} and \eref{3.51} are
obtained from the reconstructed primitive variables $\bm V$ at the corresponding time levels. Note that the conservative updates are, as a
matter of fact, explicit, since they are performed using the explicit part of the SI-DeC scheme.

{\color{black}
	\begin{remark}
		It should be observed that some parts of the reported algorithm can be parallelized, for example, Steps 1 and 2 or Steps 5 and 6.
		Furthermore, Steps 1--3 can be performed in parallel with Step 4 before applying the post-processing, as well as Steps 5--7 with Step 8.
\end{remark}}

{\color{black}
	\begin{remark}
		It should be observed that working with two sets of variables induces a computational overhead compared to approaches based on a single
		formulation. However, the computational cost is not doubled. For example, the same reconstructed cell interface values are shared by both
		the $\bm V$- and $\bm U$-solutions, and therefore no additional reconstruction procedure is required. Moreover, the evolution of the
		$\bm U$-solution is performed explicitly and is therefore computationally less demanding than the SI evolution of the $\bm V$-solution.
	\end{remark}
}

\subsection{Post-Processing}\label{sec43}
{\color{black}As mentioned in \S\ref{sec422},} upon completion of Steps 4 and 8, we {\color{black}replace $\xbar{\bm V}^{\,*}_{j,k}$ with
	$r\big(\bm V\big(\,\xbar{\bm U}^{\,*}_{j,k}\big),\,\xbar{\bm V}^{\,*}_{j,k}\big)$ and $\xbar{\bm V}^{\,n+1}_{j,k}$ with
	$r\big(\bm V\big(\,\xbar{\bm U}^{\,n+1}_{j,k}\big),\,\xbar{\bm V}^{\,n+1}_{j,k}\big)$, respectively. The function $r$ is selected based on
	the following considerations. Since the $\bm V$-solution is AP but nonconservative, and the $\bm U$-solution is conservative but non-AP, we
	use} their convex combination with coefficients dependent on $\ve$, leveraging the AP SI method in the low-Mach-number regime and the sharp
conservative CU scheme in the moderate- and high-Mach-number regimes---thus ensuring accuracy, stability, and physical consistency across
all flow regimes. Specifically, we select the following replacement function $r$:
\begin{equation}
	r\big(\bm V\big(\,\xbar{\bm U}_{j,k}\big),\,\xbar{\bm V}_{j,k}\big)=
	(1-s(\ve))\,\bm V\big(\,\xbar{\bm U}_{j,k}\big)+s(\ve)\,\xbar{\bm V}_{j,k},\quad\forall j,k,
	\label{star}
\end{equation}
where $s$ is a suitable switching function, which is supposed to be increasing, continuous, and satisfy $s(1)=0$ and $s(0)=1$. Moreover,
in the high-Mach-number regime, $s$ should be $\sim0$ so that the primitive variables $\bm V$ are almost completely overwritten by
$\bm V(\bm U)$, while, in the low-Mach-number regime, $s$ should be $\sim1$ so that the primitive variables $\bm V$ stay almost unchanged.
For intermediate values of $\ve$, a smooth transition between $1$ and $0$ is expected.

\section{Numerical Examples}\label{sec6}
In this section, we verify the accuracy and robustness of the proposed AP scheme on a variety of numerical examples across different values 
of $\ve$. In all of the numerical examples, we:

\smallskip
\noindent
$\bullet$ Take the minmod parameter $\theta=1.3$;

\smallskip
\noindent
$\bullet$ Adaptively select time steps based on the time-step restriction \eref{4.10} for the nonstiff part of the primitive system;

\smallskip
\noindent
$\bullet$ Set $\gamma=1.4$ (except for Example 1, in which $\gamma=2$);

\smallskip
\noindent
$\bullet$ Modify \eref{4.3} to
\begin{equation}
	\rho_{\max}=\max\limits_{(x,y)\in\Omega}\rho+\ve^4,\quad p_{\min}=\min\limits_{(x,y)\in\Omega}p-\ve^4.
	\label{4.9}
\end{equation}
Notice that this modification has almost no impact in the low-Mach-number regime, but it aims at adding more upwinding and thus improving
the stability property of the resulting AP scheme when $\ve$ is large;

\smallskip
\noindent
$\bullet$ Choose the following switching function:
\begin{equation*}
	s(\ve)=\begin{cases}1-\ve^{\alpha},&0<\ve\le\ve_0,\\\exp\bigg(1-\frac{1}{1-\left(\frac{\ve-\ve_0}{\ve_1-\ve_0}\right)^2}\bigg)
		\left[(1-\ve_0^\alpha)-(1-\ve_1)^\alpha\right]+(1-\ve_1)^\alpha,&\ve_0<\ve<\ve_1,\\(1-\ve)^\alpha,&\ve_1\le\ve\le1,\end{cases}
\end{equation*}
where $\ve_0$, $\ve_1$, and $\alpha$ are positive constants taken to be $\ve_0=0.15$, $\ve_1=0.4$, and $\alpha=14$ in all of the numerical
examples below. This switching function is plotted in Figure \ref{fig31}.
\begin{figure}[ht!]
	\centerline{\includegraphics[width=0.32\textwidth]{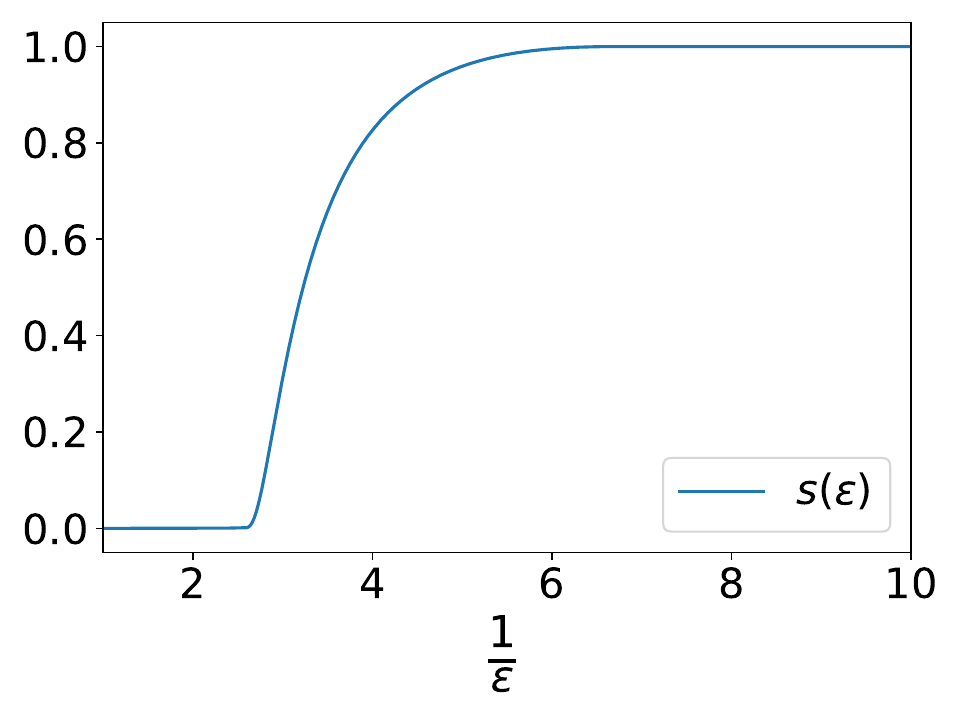}}
	\caption{\sf Switching function $s(\varepsilon)$ plotted with respect to $1/\ve$.\label{fig31}}
\end{figure}
{\color{black}
	\begin{remark}
		Modification \eref{4.9} is used to improve the handling of discontinuous solutions for large (intermediate)-Mach-number
		regimes. However, we would like to emphasize that it vanishes as $\ve\to0$ and therefore does not affect either the consistency of the
		method or the AP analysis presented in the paper.
\end{remark}}
{\color{black}
	\begin{remark}
		We stress that the same switching function, $s(\ve)$, together with the same values of the associated parameters, has been used in all of 
		the numerical examples reported below. The obtained numerical results indicate that this choice is sufficiently robust across the considered
		Mach-number regimes. {\color{black}Further studies on the optimal design of the switching function and on the selection of its parameters
			may constitute interesting directions for future research.}
\end{remark}}

\subsubsection*{Example 1---Accuracy Test for Low-Mach-Number Smooth Vortex}
In this example taken from \cite{zeifang2020novel}, we consider a smooth, unsteady Mach dependent vortex over the computational domain
$[-10,10]\times[-10,10]$ subject to the periodic boundary conditions. The analytical solution is given, modulo the periodicity, by
\begin{equation*}
	\begin{aligned}
		\rho(\bm x_r)&=1-\frac{\ve^2}{16\pi^2}\,{\rm e}^{1-\|\bm x_r\|_2^2},\quad
		u(\bm x_r)=1-\frac{\ve y_r}{2\pi}\,{\rm e}^{\frac{1-\|\bm x_r\|_2^2}{2}},\quad
		v(\bm x_r)=1+\frac{\ve x_r}{2\pi}\,{\rm e}^{\frac{1-\|\bm x_r\|_2^2}{2}},\\
		E(\bm x_r)&=1+\ve^2\Big[\rho^2(\bm x_r)+\frac{\rho(\bm x_r)}{2}\big(u^2(\bm x_r)+v^2(\bm x_r)\big)\Big],
	\end{aligned}
\end{equation*}
where $\bm x_r(x,y,t)=(x_r,y_r)^\top:=(x-t,y-t)^\top$.

We take the CFL number $K_{\rm CFL}=0.475$ and compute the numerical solution until the final time $t=0.1$ on a series of uniform
$N\times N$ meshes with $N=64$, $128$, $256$, and $512$ for $\ve=1$, $0.1$, $0.01$, and $0.001$. {\color{black}We study the convergence in
	terms of the $L^1$-errors, defined for each scalar component $V$ of the $\bm V$-solution at the final time as
	$$
	\dx\dy\sum_{j,k}\abs{V_{j,k}(t=0.1)-V(x_j,y_k,t=0.1)},
	$$ 
	where $V(x,y,t)$ denotes the corresponding component of the exact solution. The obtained errors are reported in Figure \ref{fig42}, showing}
that the expected second-order convergence rate has been achieved in all variables for all considered $\ve$. One can also observe that, for
fixed mesh refinement, the error decreases for decreasing $\ve$ as a result of the convergence of the analytical solution to the
incompressible limit $(\rho,u,v,p)^\top=(1,1,1,1)^\top$ and of the AP character of the proposed AP DF-FV scheme.
\begin{figure}[ht!]
	\centerline{\includegraphics[width=0.28\textwidth]{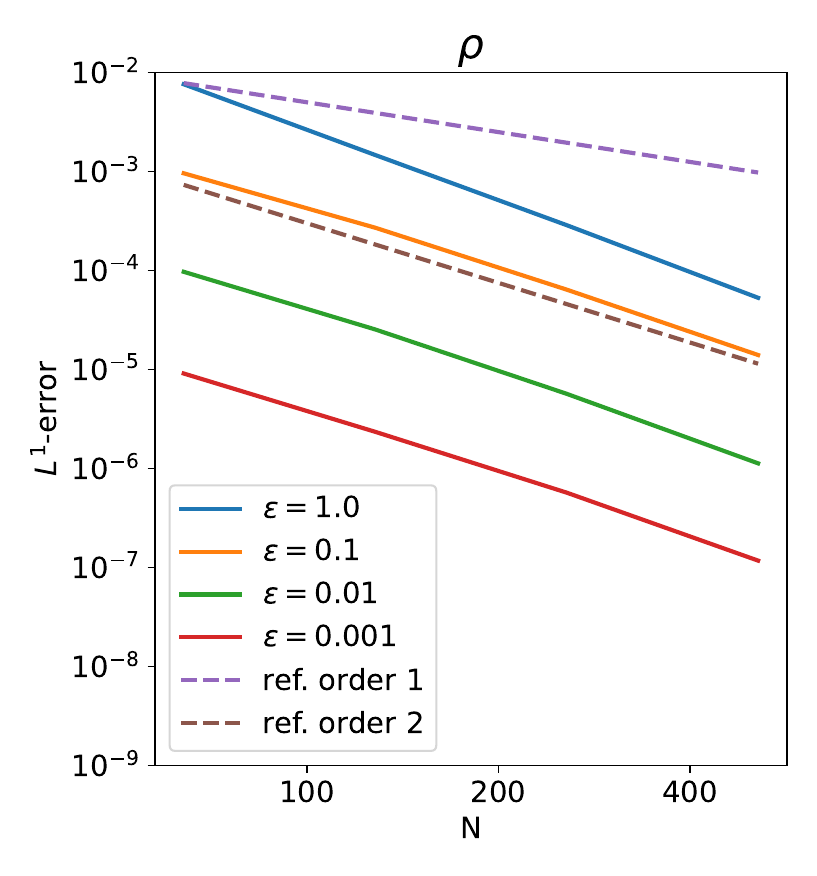}\hspace*{0.8cm}
		\includegraphics[width=0.28\textwidth]{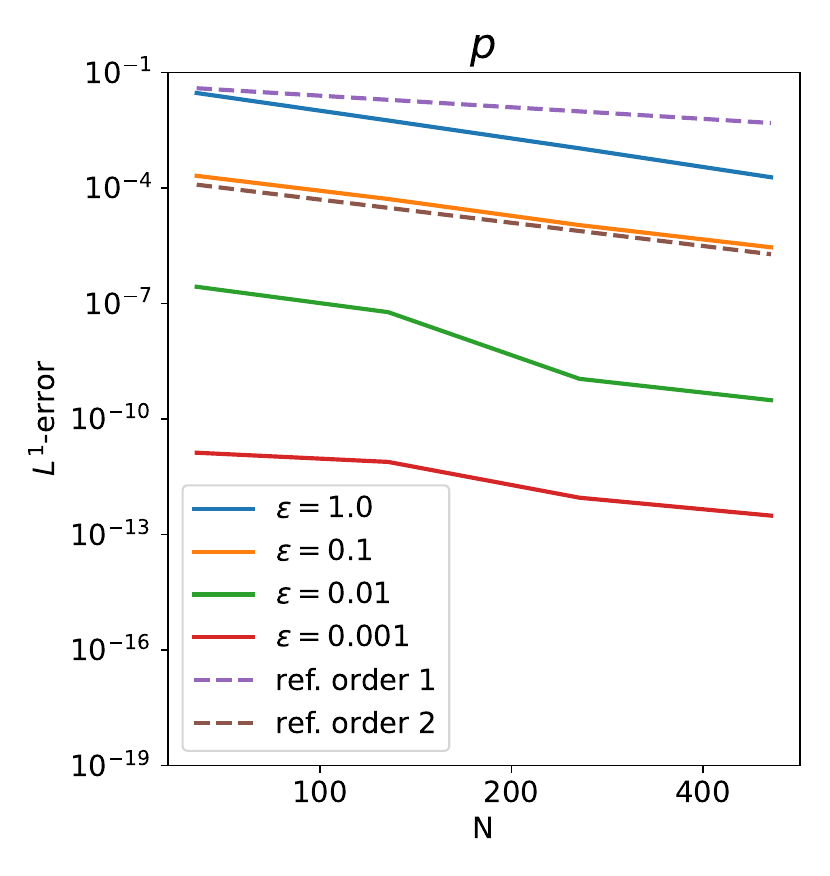}}
	\vskip10pt
	\centerline{\includegraphics[width=0.28\textwidth]{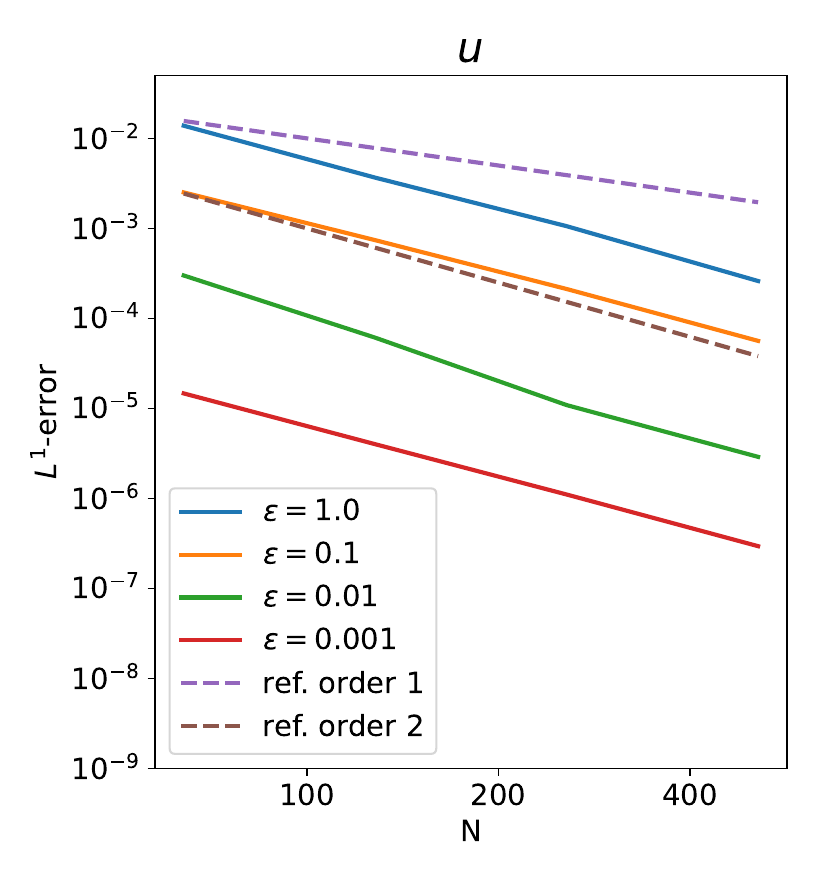}\hspace*{0.8cm}
		\includegraphics[width=0.28\textwidth]{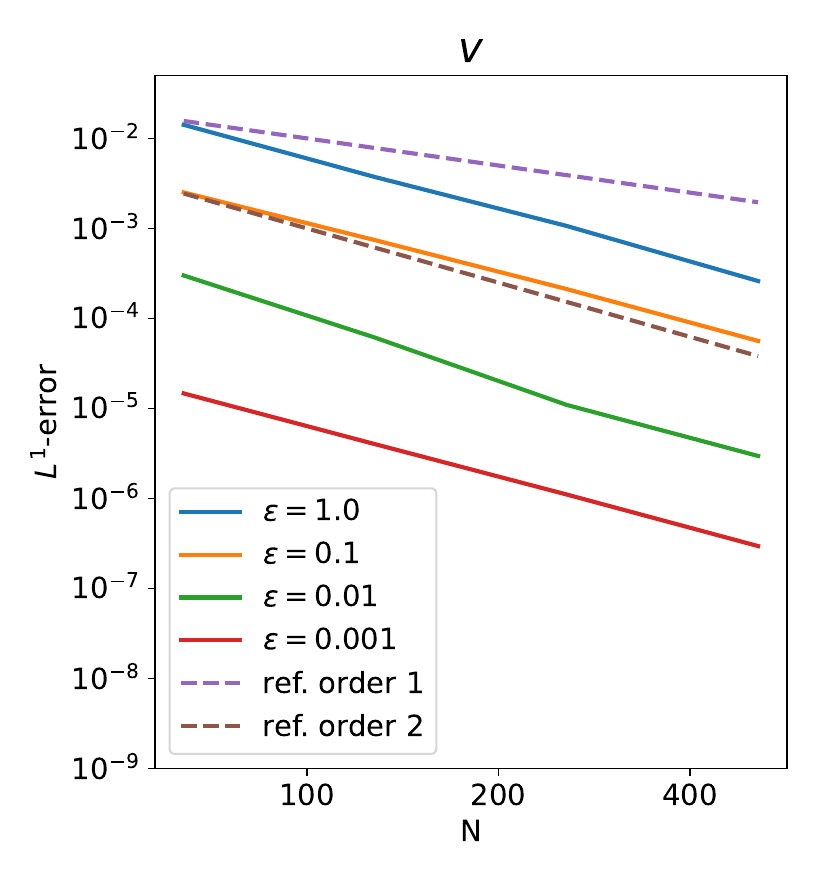}}
	\caption{\sf Example 1: Convergence analysis.\label{fig42}}
\end{figure}

\subsubsection*{Example 2---Gresho Vortex}
This example was introduced in \cite{gresho1990theory} and, since then, it has been widely used as a common benchmark to numerically validate
the AP property. We consider a steady vortex over the computational domain $[0,1]\times[0,1]$ subject to the periodic boundary conditions.
At any time $t$, the shape of the vortex is given by
\begin{equation*}
	\begin{aligned}
		\rho(r)&\equiv1,\quad u(r)=-\frac{y_r}{r}\psi(r),\quad v(r)=\frac{x_r}{r}\psi(r),\\
		p(r)&=\begin{cases}
			1+12.5\ve^2r^2,&r<0.2,\\
			1+\ve^2\left(4\ln(5r)+4-20r+12.5r^2\right),&0.2\le r<0.4,\\
			1+\ve^2(4\ln2-2),&r\ge0.4,
		\end{cases}
	\end{aligned}
\end{equation*}
where
\begin{equation*}
	x_r:=x-0.5,\quad y_r:=y-0.5,\quad r:=\sqrt{x_r^2+y_r^2},\quad
	\psi(r):=\begin{cases}
		5r,&r<0.2,\\
		2-5r,&0.2\le r<0.4,\\
		0,&r\ge0.4.
	\end{cases}
\end{equation*}

We take the CFL number $K_{\rm CFL}=0.475$ and compute the numerical solution until the final time $t=1$ on a uniform $128\times128$ mesh
for $\ve=10^{-\alpha}$ with $\alpha=1,\dots,6$, and report the obtained local Mach number, defined as $\|\bm u\|_2/\sqrt{\gamma}$, in 
Figure \ref{fig43} along with its initial distribution. According to what is expected due to the AP feature of the scheme, the shape of the 
vortex is preserved and no evident dependency on $\ve$ can be observed.
\begin{figure}[ht!]
	\centerline{\includegraphics[width=0.33\textwidth]{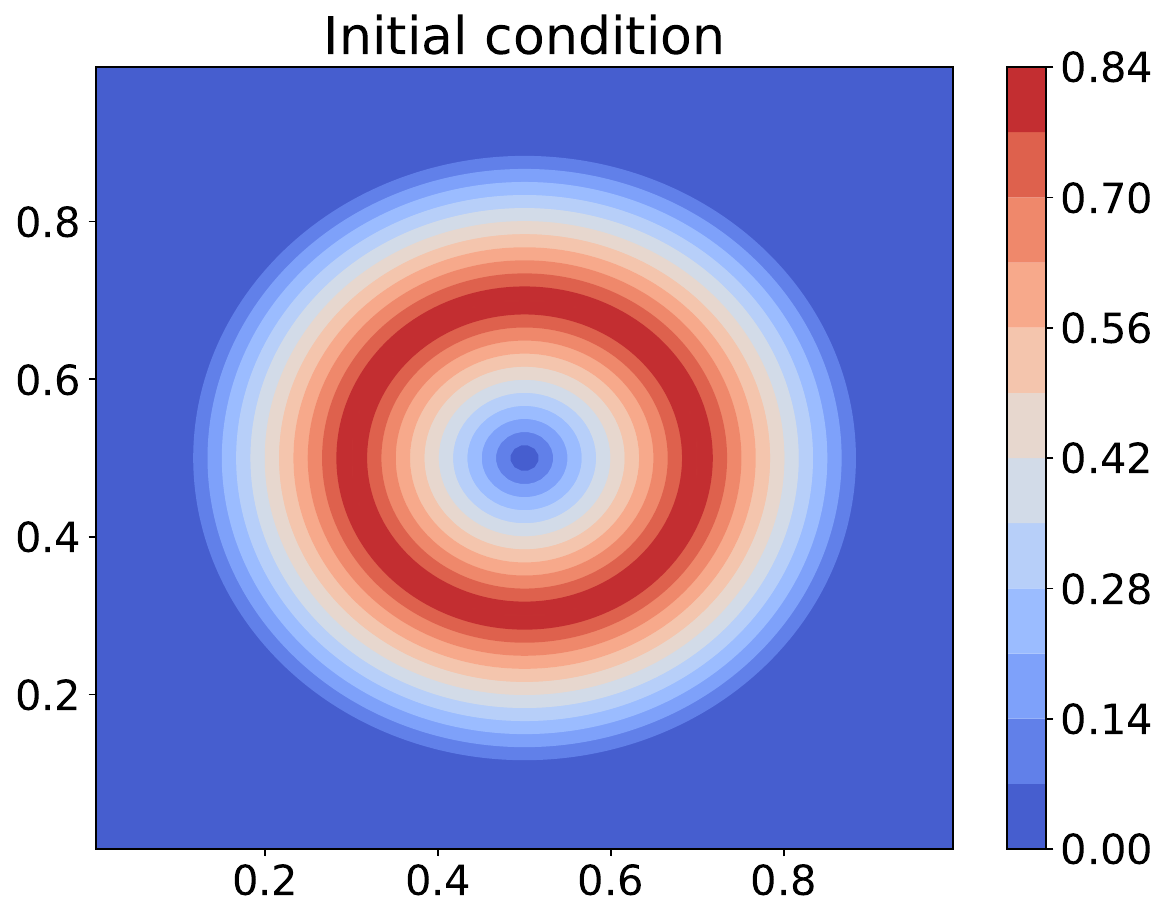}}
	\vskip8pt
	\centerline{\includegraphics[width=0.33\textwidth]{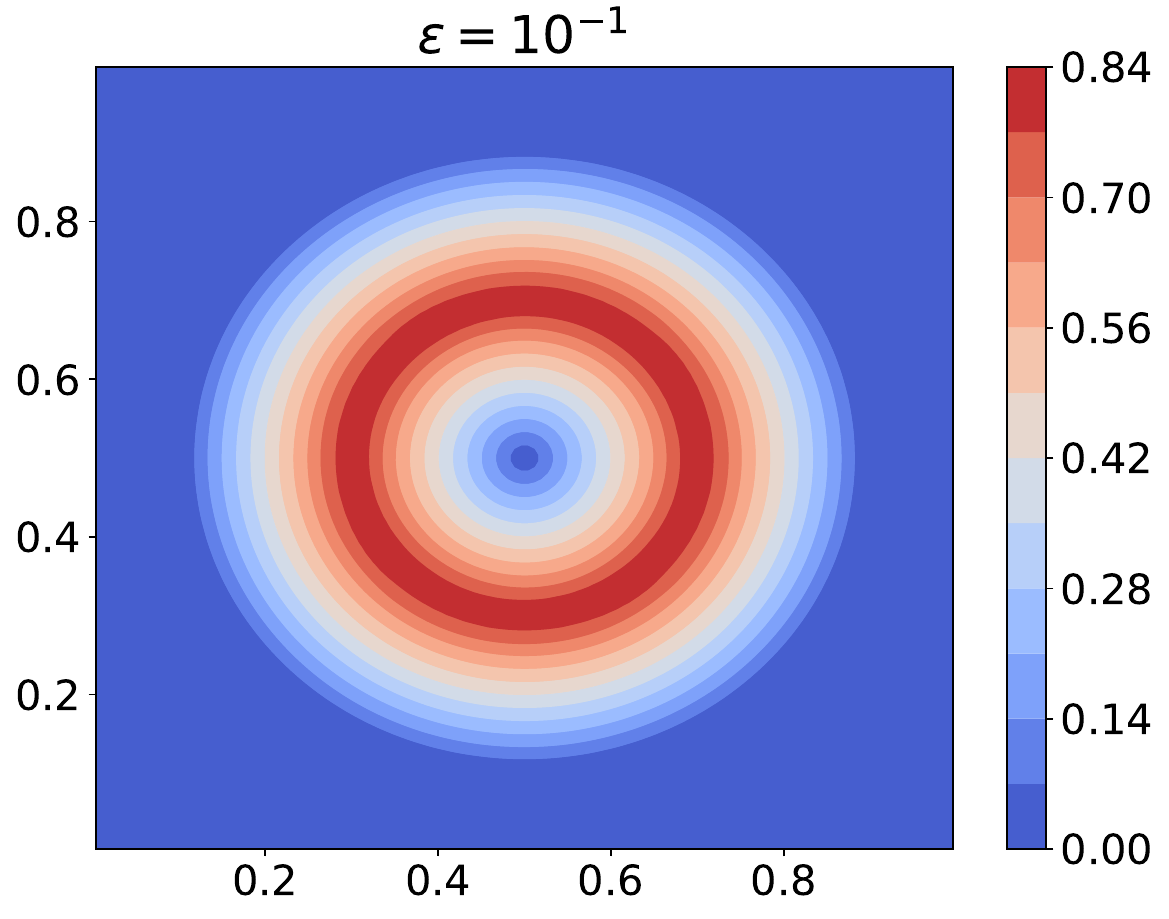}\hspace*{0.1cm}
		\includegraphics[width=0.33\textwidth]{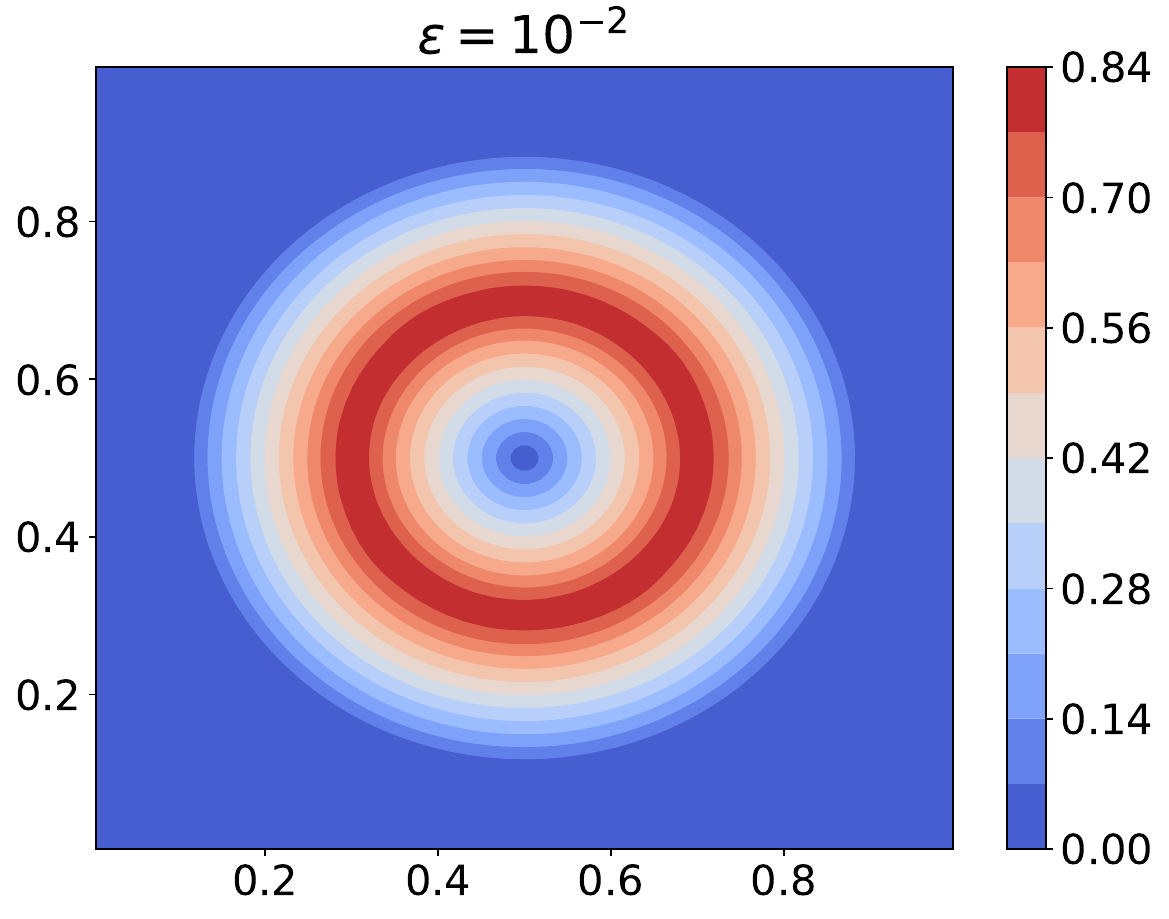}\hspace*{0.1cm}
		\includegraphics[width=0.33\textwidth]{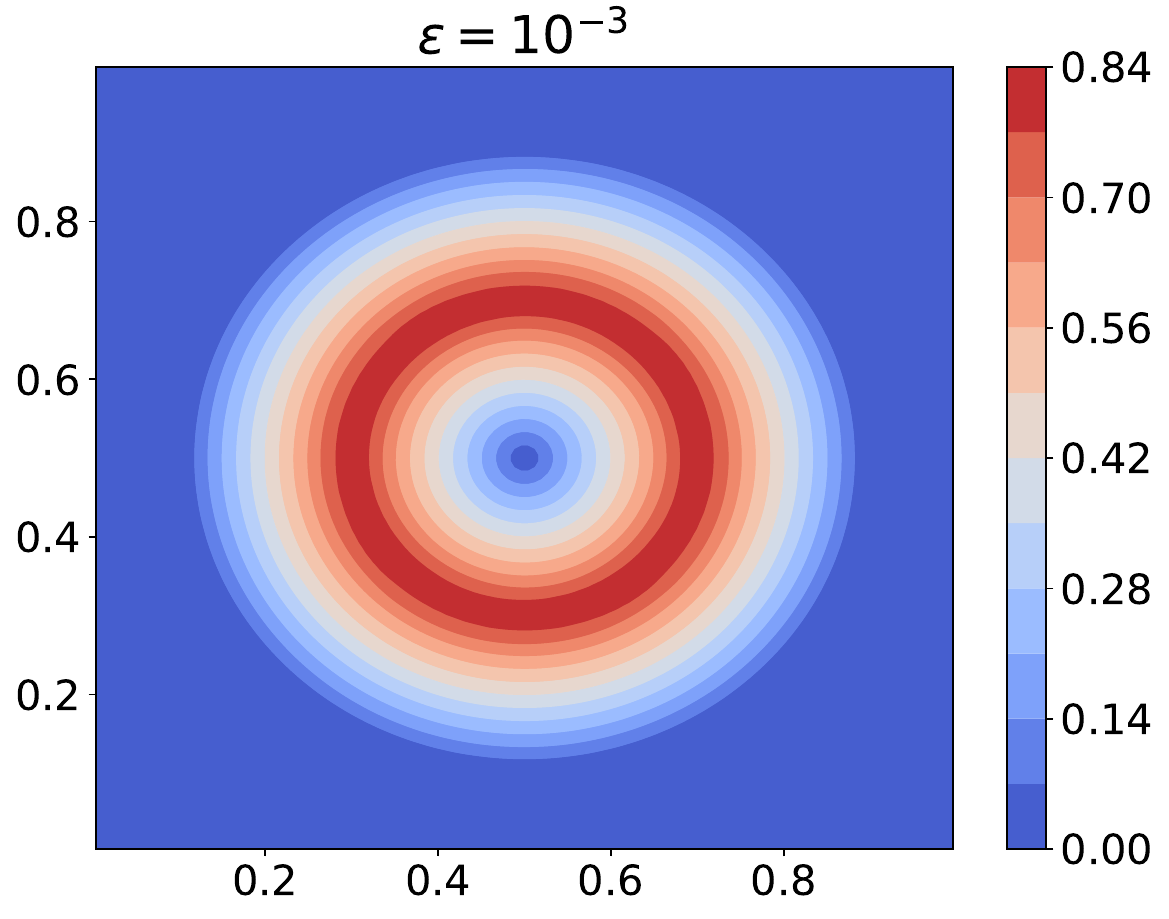}}
	\vskip8pt
	\centerline{\includegraphics[width=0.33\textwidth]{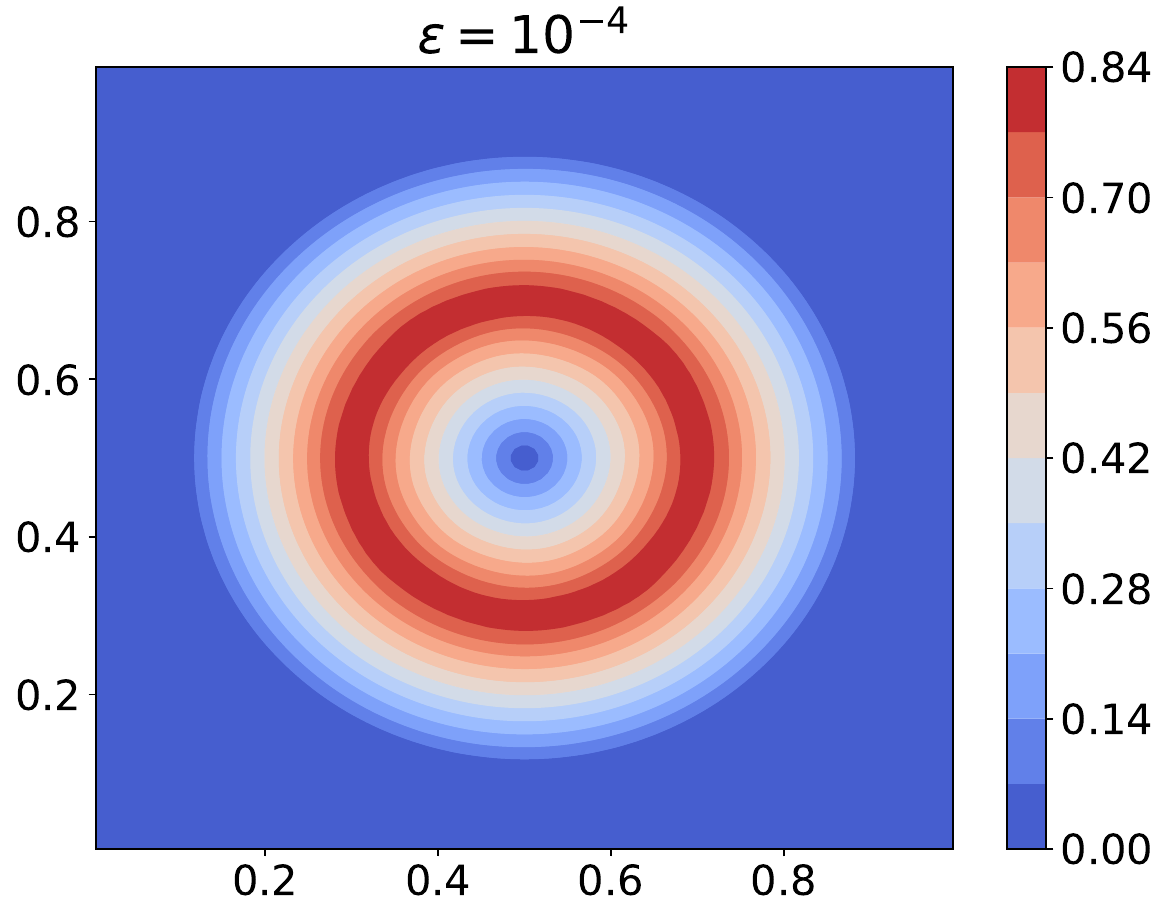}\hspace*{0.1cm}
		\includegraphics[width=0.33\textwidth]{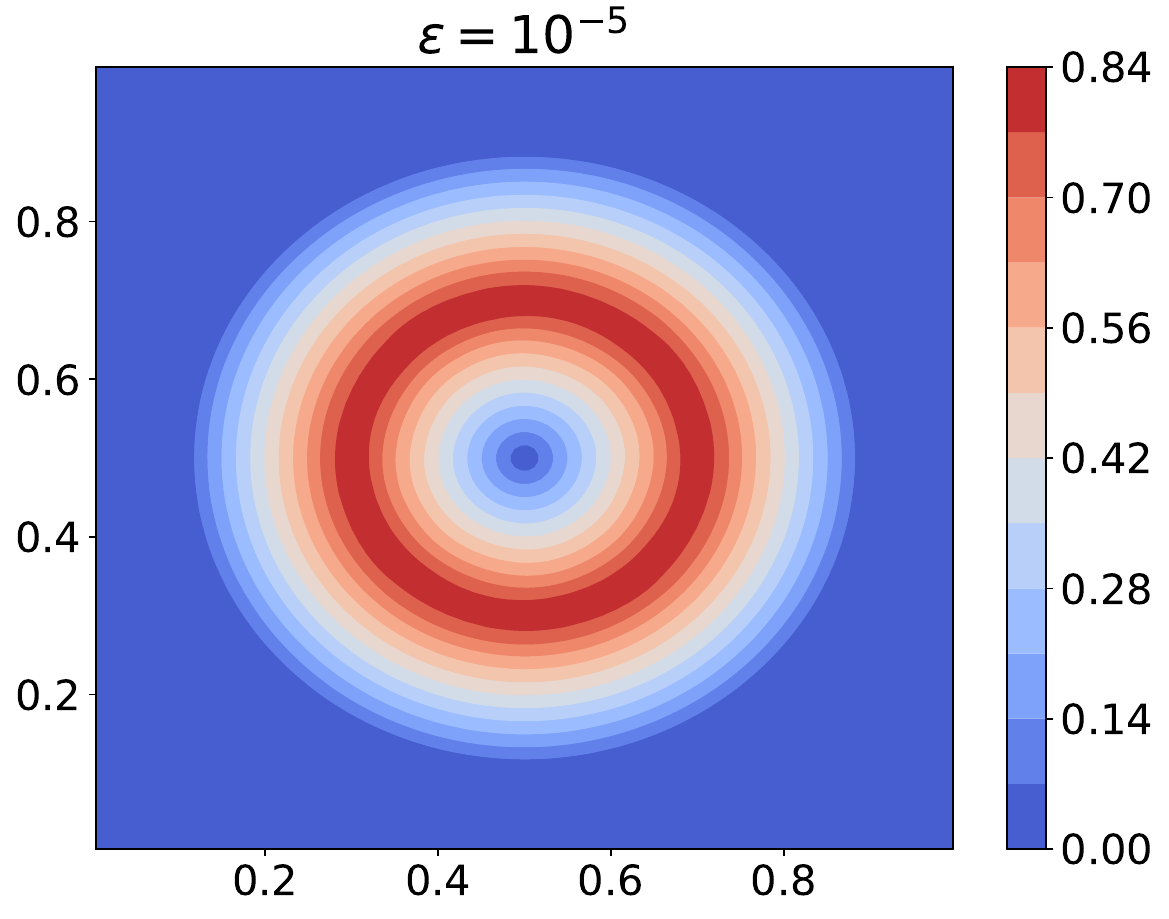}\hspace*{0.1cm}
		\includegraphics[width=0.33\textwidth]{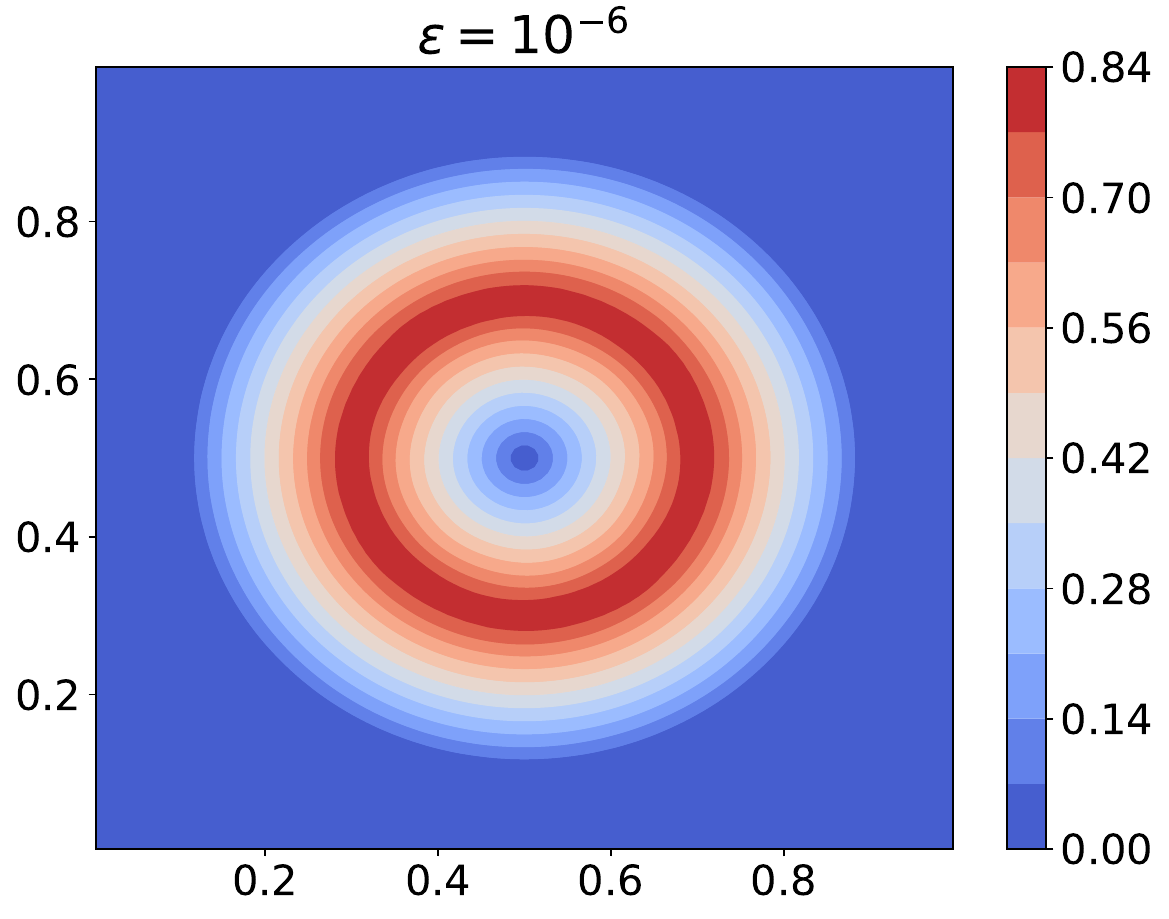}}
	\caption{\sf Example 2: Initial local Mach number independently of $\ve$ (top) and local Mach number at $t=1$ for different values of $\ve$.
		\label{fig43}}
\end{figure}

\subsubsection*{Example 3---Baroclinic Vorticity Generation}
In this example taken from \cite{noelle2014weakly}, we consider a low-Mach-number flow with $\ve=0.05$ involving an acoustic wave, which
moves within two density layers in the computational domain $[-\frac{1}{\ve},\frac{1}{\ve}]\times[0,\frac{2}{5\ve}]$ subject to the 
periodic boundary conditions. The initial conditions are
\begin{equation*}
	\begin{aligned}
		&\rho(x,y,0)=1+\frac{\ve}{2000}[1+\cos(\ve\pi x)]+4.5\ve y-\begin{cases}0,&0\le y\le\frac{1}{5\ve},\\1.8,&\mbox{otherwise,}\end{cases}\\
		&u(x,y,0)=\frac{\sqrt{\gamma}}{2}[1+\cos(\ve\pi x)],\quad v(x,y,0)\equiv0,\quad p(x,y,0)=1+\frac{\ve\gamma}{2}[1+\cos(\ve\pi x)].
	\end{aligned}
\end{equation*}
{\color{black}It should be observed that the initial density discontinuity is not accompanied by a corresponding pressure discontinuity.
	Therefore, the initial data are not expected to generate any strong propagating compressible waves. On the other hand,} the acoustic wave
induces different accelerations in the two density layers, which results in rotational excitation and in the formation of a long-wavelength
sinusoidal shear layer. Due to the interaction with the acoustic wave, such a shear layer becomes unstable, and several
Kelvin-Helmholtz-type unstable structures originate from it.

The numerical solution is computed with the CFL number $K_{\rm CFL}=0.475$ until the final time $t=20$ on a $800\times160$ uniform mesh. 
The density at times $t=0$, $10$, and $20$ is plotted in Figure \ref{fig44}. Since the solution develops instabilities, no strong 
convergence is expected in this example; see \cite{zeifang2020novel}. One can, however, observe that the underlying physics is correctly 
captured.
\begin{figure}[ht!]
	\centerline{\includegraphics[width=0.9\textwidth]{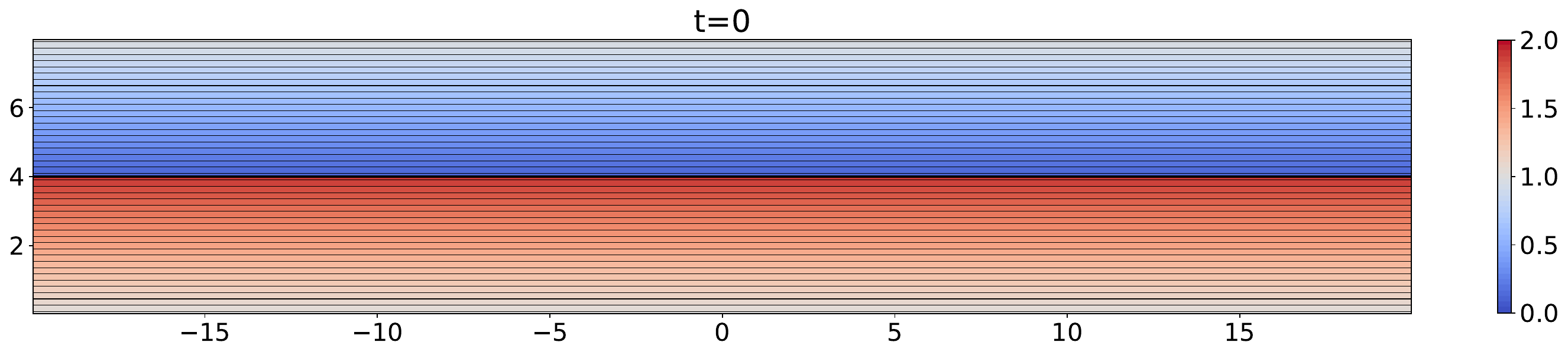}}
	\centerline{\includegraphics[width=0.9\textwidth]{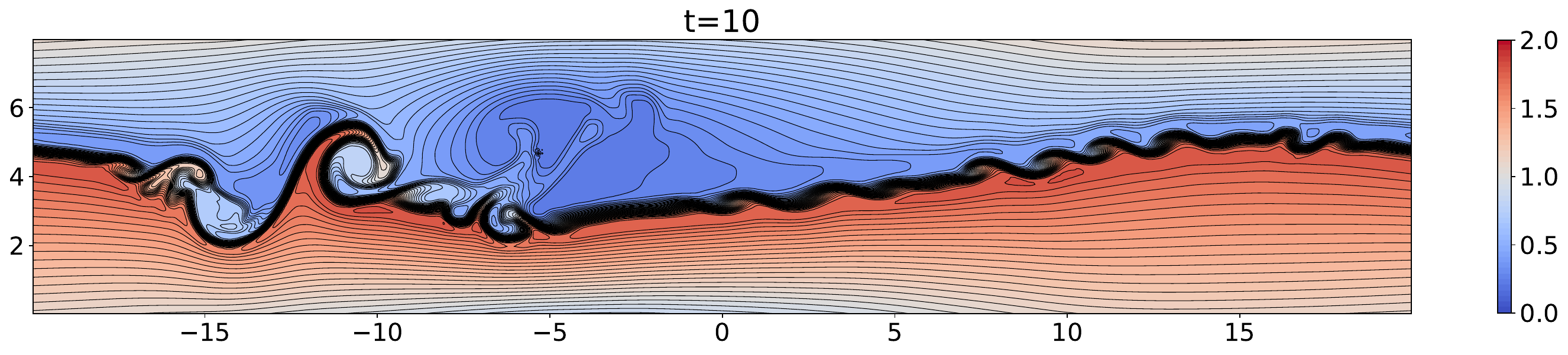}}
	\centerline{\includegraphics[width=0.9\textwidth]{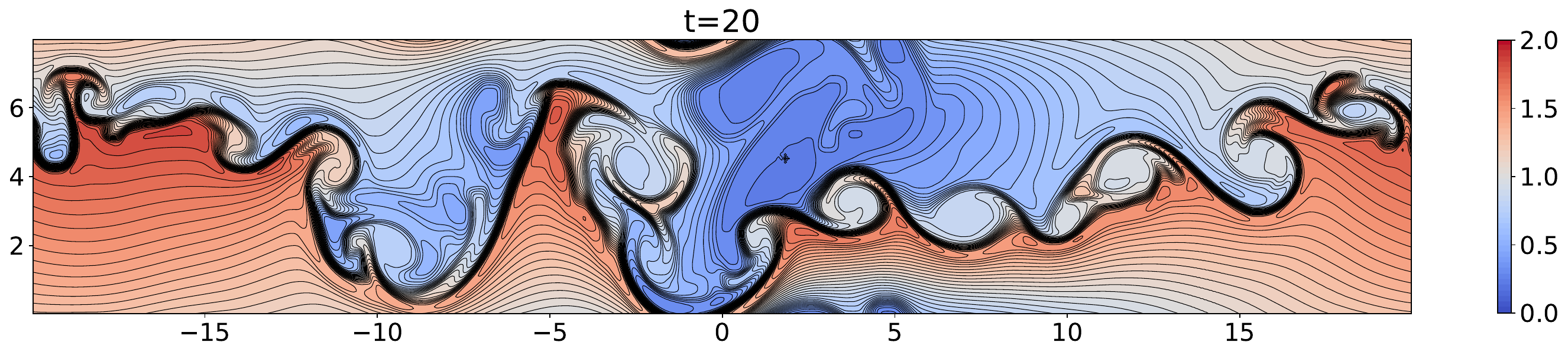}}
	\caption{\sf Example 3: Density at different times.\label{fig44}}
\end{figure}

\subsubsection*{Example 4---Double Shear Layer Problem}
In the following test case, originally introduced in \cite{bell1989second} for the incompressible Navier–Stokes equations and subsequently 
adopted in, e.g., \cite{zeifang2020novel,boscarino2018all,weinan1994numerical} in the context of compressible Euler equations in the
low-Mach-number regime, a shear layer develops, and the AP property of the proposed scheme can be assessed. In particular, we would like to
check whether the scheme maintains its consistency for small values of $\ve$, that is, in the almost incompressible regime.

The initial conditions,
\begin{align*}
	&\rho(x,y,0)\equiv\frac{\pi}{15},\quad
	u(x,y,0)=\left\{\begin{aligned}
		&\tanh\!\left[15\Big(\frac{y}{\pi}-\frac{1}{2}\Big)\right],&&y\le\pi,\\[0.5ex]
		&\tanh\!\left[15\Big(\frac{3}{2}-\frac{y}{\pi}\Big)\right],&&\mbox{otherwise},
	\end{aligned}\right.\quad 
	v(x,y,0)=0.05\sin x,\\
	&p(x,y,0)\equiv\frac{1}{\gamma},
\end{align*}
are prescribed in the computational domain $[0,2\pi]\times[0,2\pi]$ subject to the periodic boundary conditions. The initial vorticity
$\omega:=v_x-u_y$, where the derivatives are approximated using second-order central differences, is plotted in Figure \ref{fig45}.
\begin{figure}[ht!]
	\centerline{\includegraphics[width=0.33\textwidth]{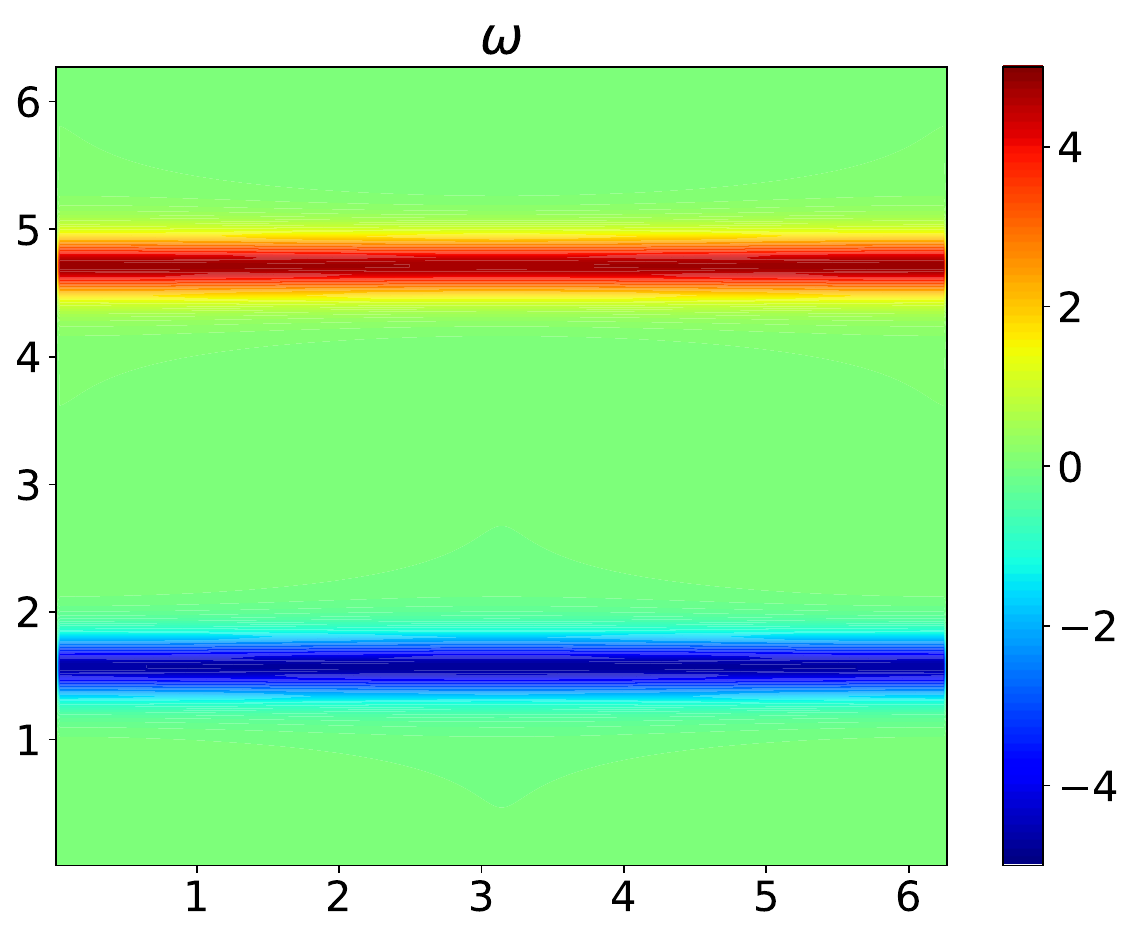}}
	\caption{\sf Example 4: Initial vorticity.\label{fig45}}
\end{figure}

We compute the numerical solutions for $\ve=10^{-\alpha}$ with $\alpha=1,\dots,6$ until the final time $t=10$ on a $256\times256$ uniform
mesh using $K_{\rm CFL}=0.1$. Figures \ref{fig46} and \ref{fig47} display the vorticity at times $t=6$ and $t=10$, respectively, for
different $\ve$. The obtained results are consistent with those reported in \cite{zeifang2020novel,boscarino2018all}. Moreover, no
macroscopic dependence on $\ve$ is observed, providing further evidence of the AP property of the proposed DF-FV scheme.
\begin{figure}[ht!]
	\centerline{\includegraphics[width=0.33\textwidth]{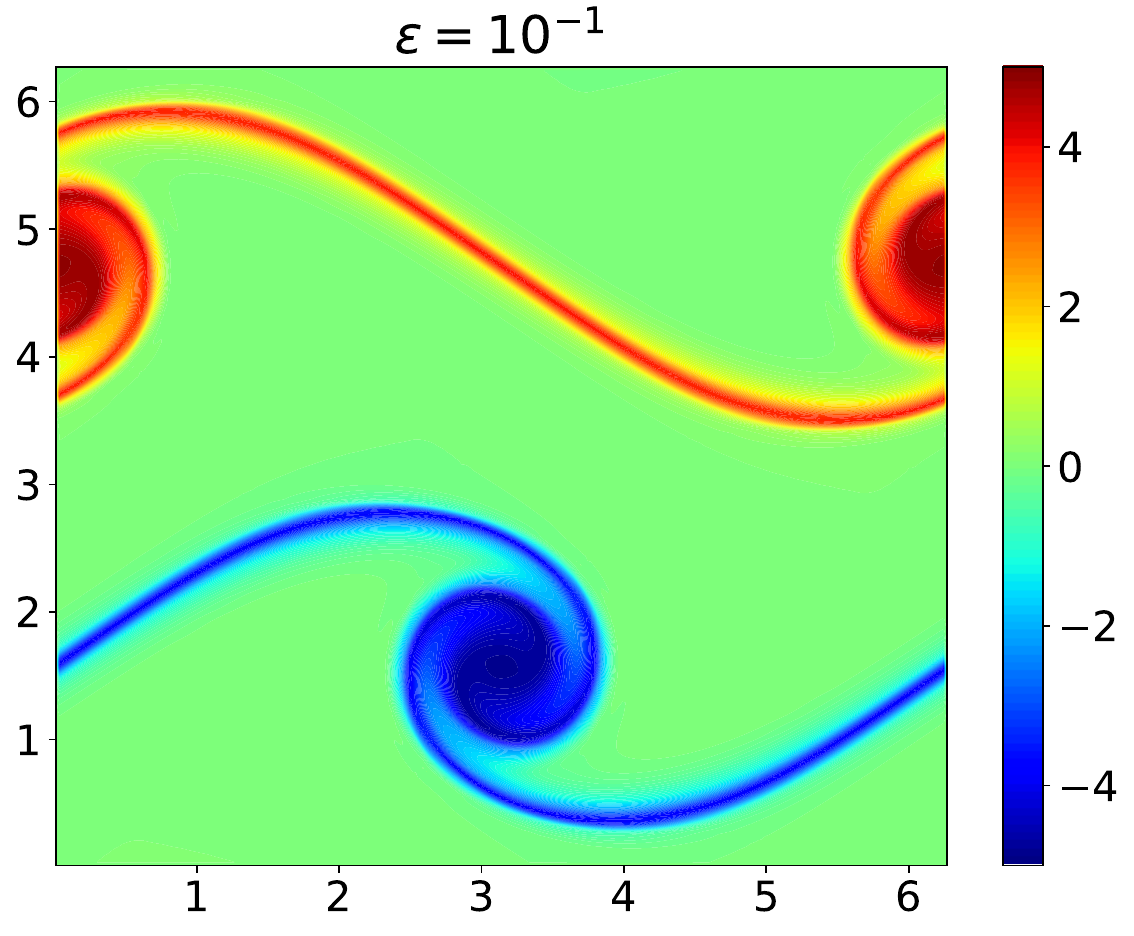}\hspace*{0.1cm}
		\includegraphics[width=0.33\textwidth]{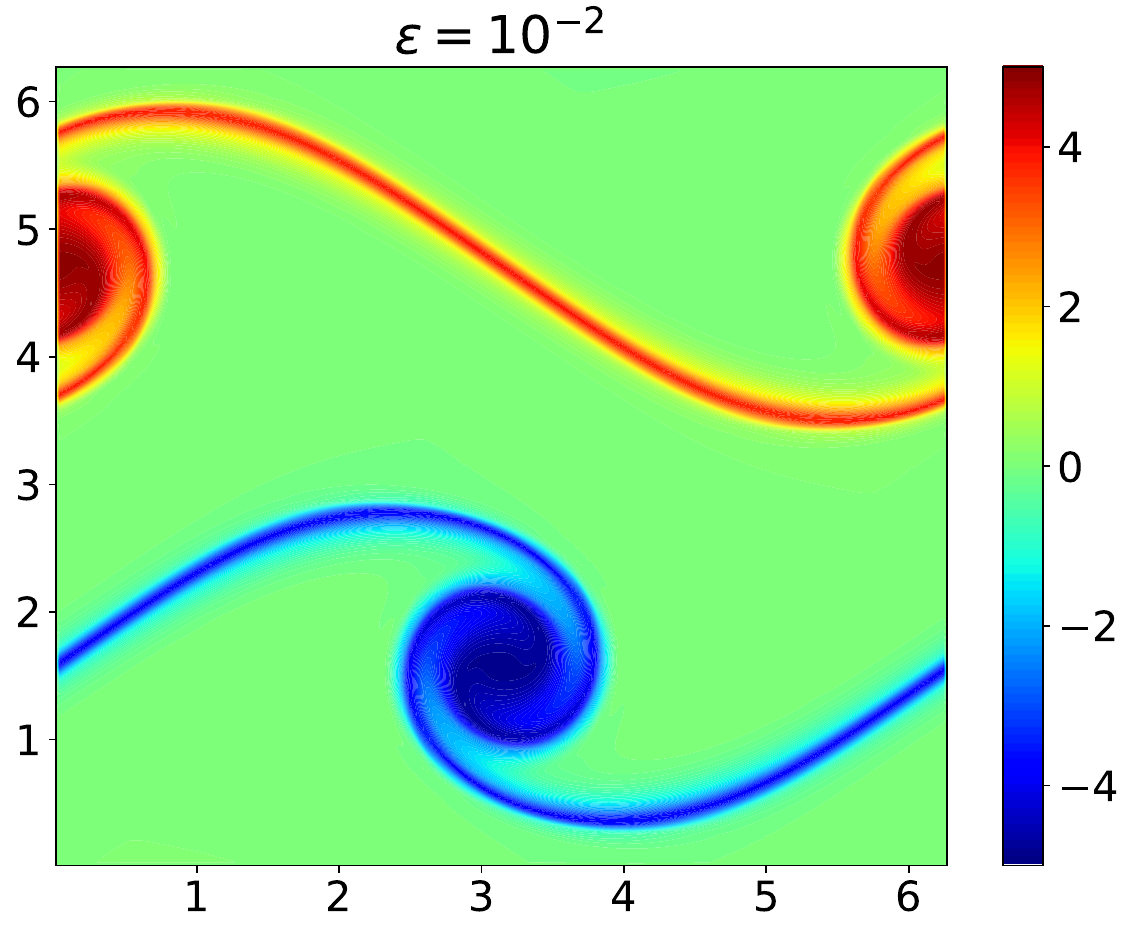}\hspace*{0.1cm}
		\includegraphics[width=0.33\textwidth]{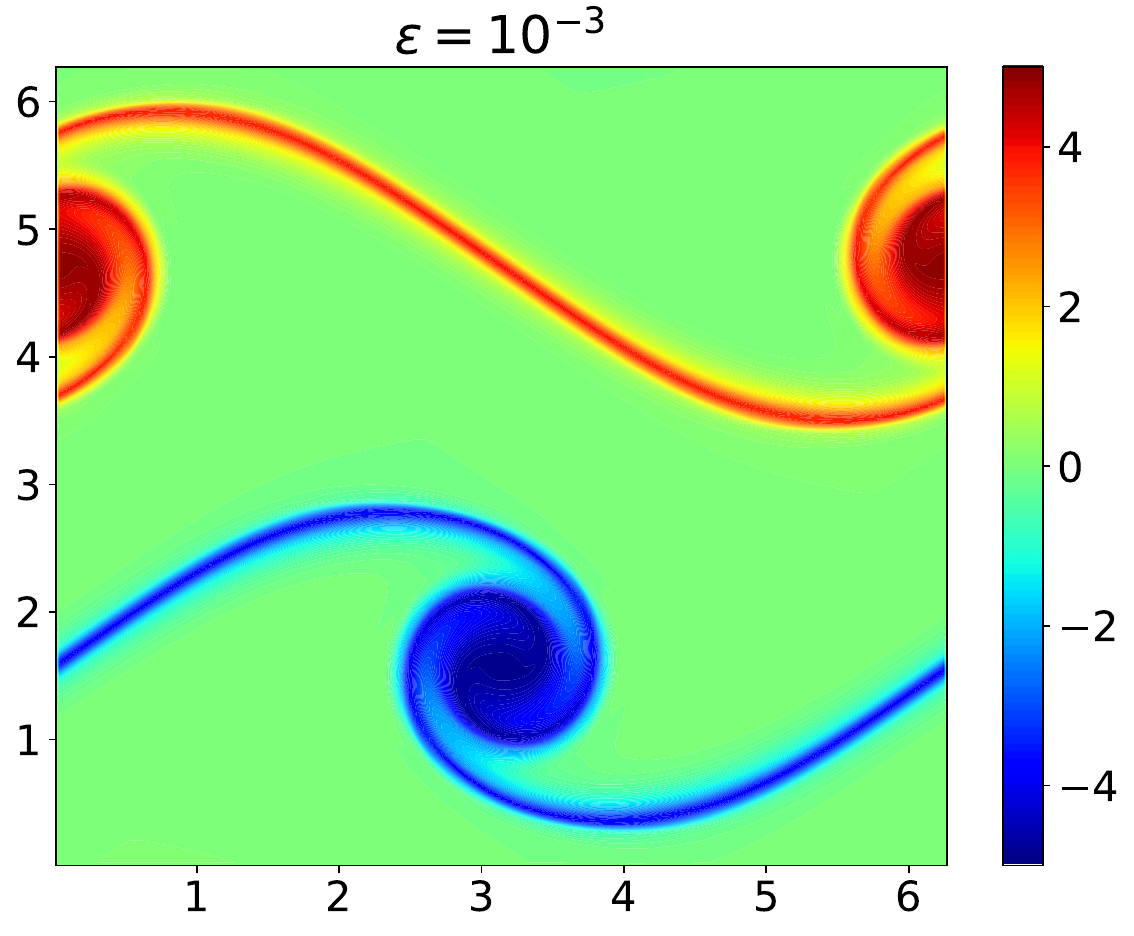}}
	\vskip8pt
	\centerline{\includegraphics[width=0.33\textwidth]{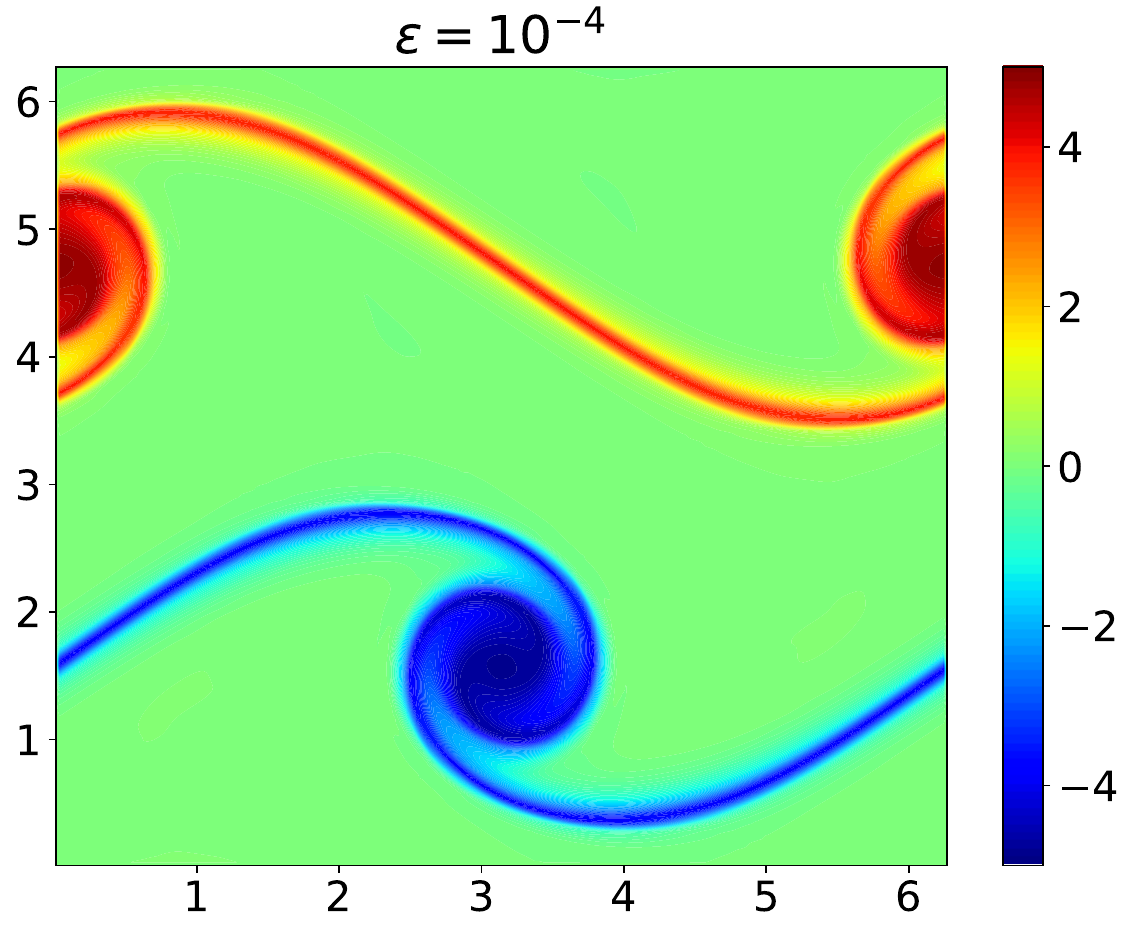}\hspace*{0.1cm}
		\includegraphics[width=0.33\textwidth]{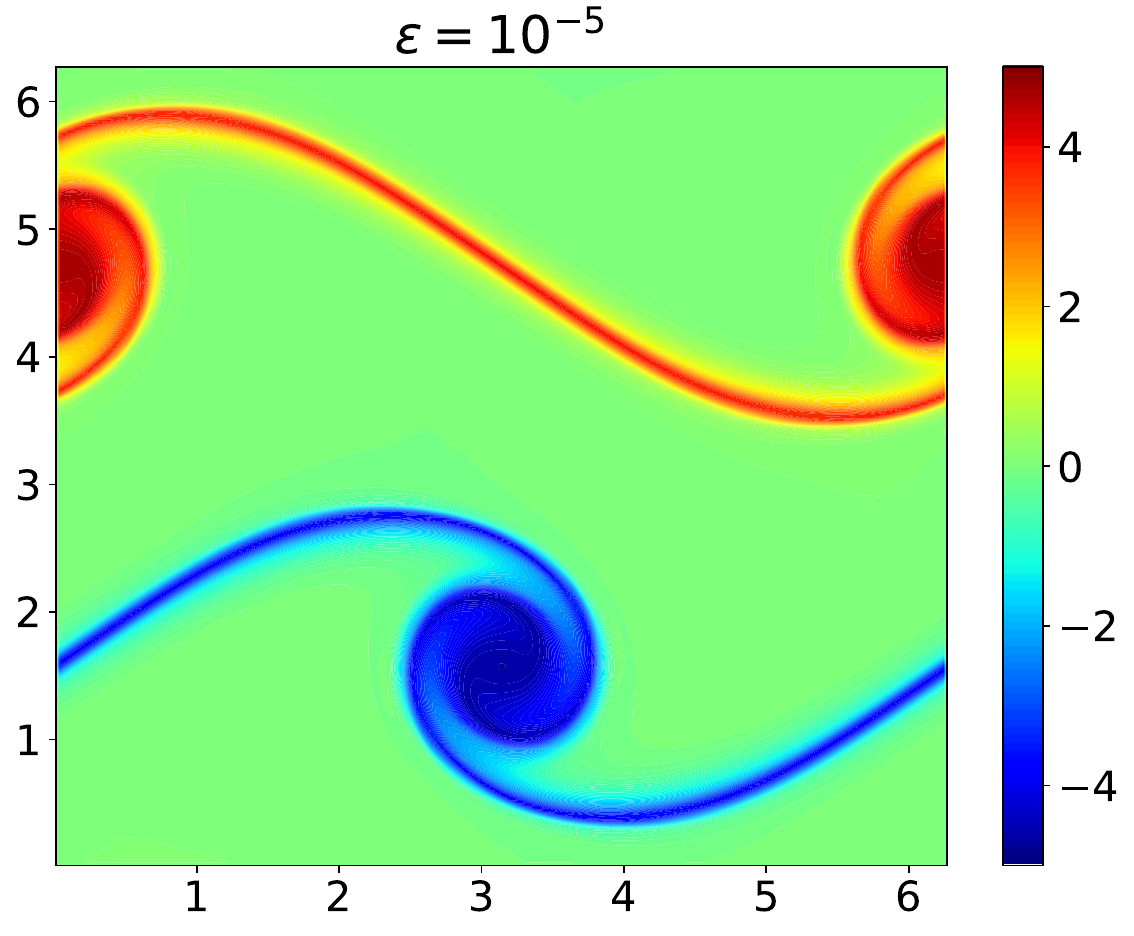}\hspace*{0.1cm}
		\includegraphics[width=0.33\textwidth]{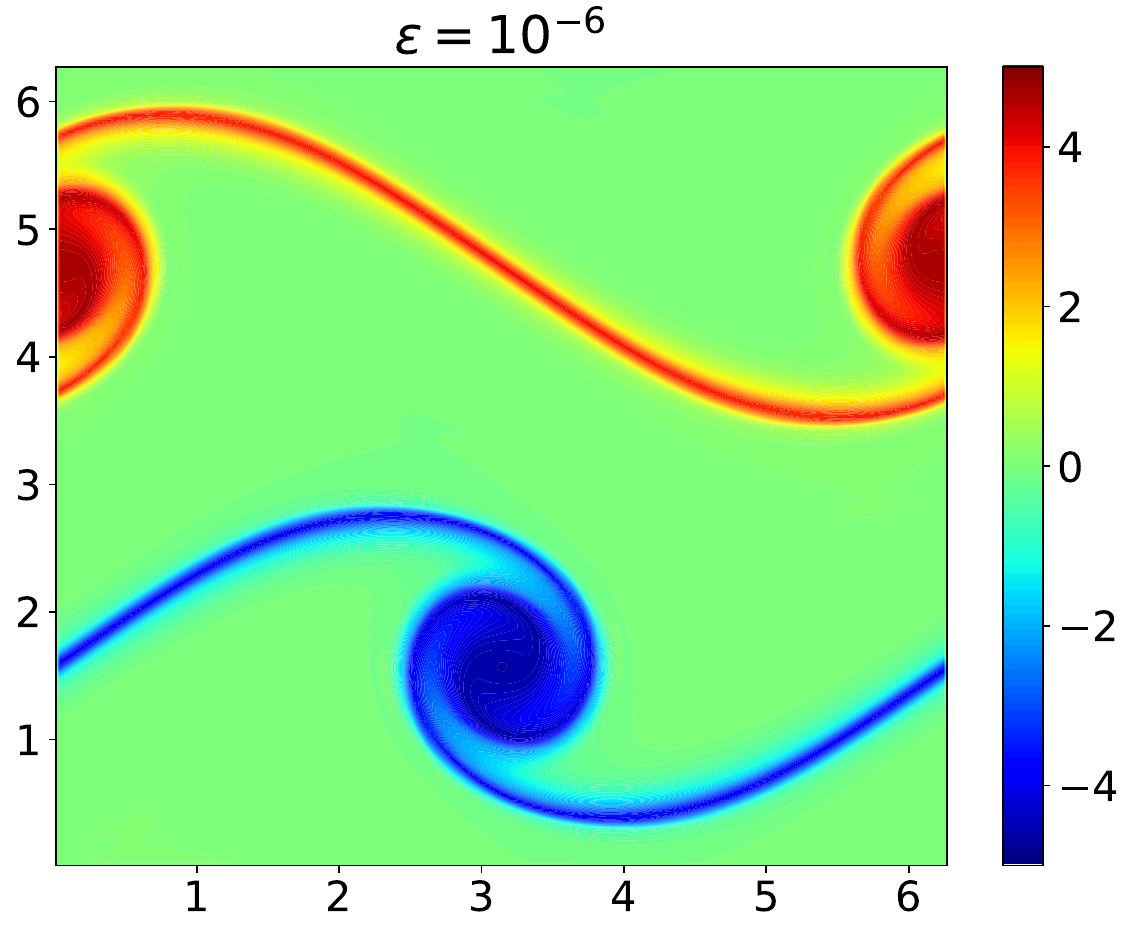}}
	\caption{\sf Example 4: Vorticity at $t=6$ for different values of $\ve$. $K_{\rm CFL}=0.1$.\label{fig46}}
\end{figure}
\begin{figure}[ht!]
	\centerline{\includegraphics[width=0.33\textwidth]{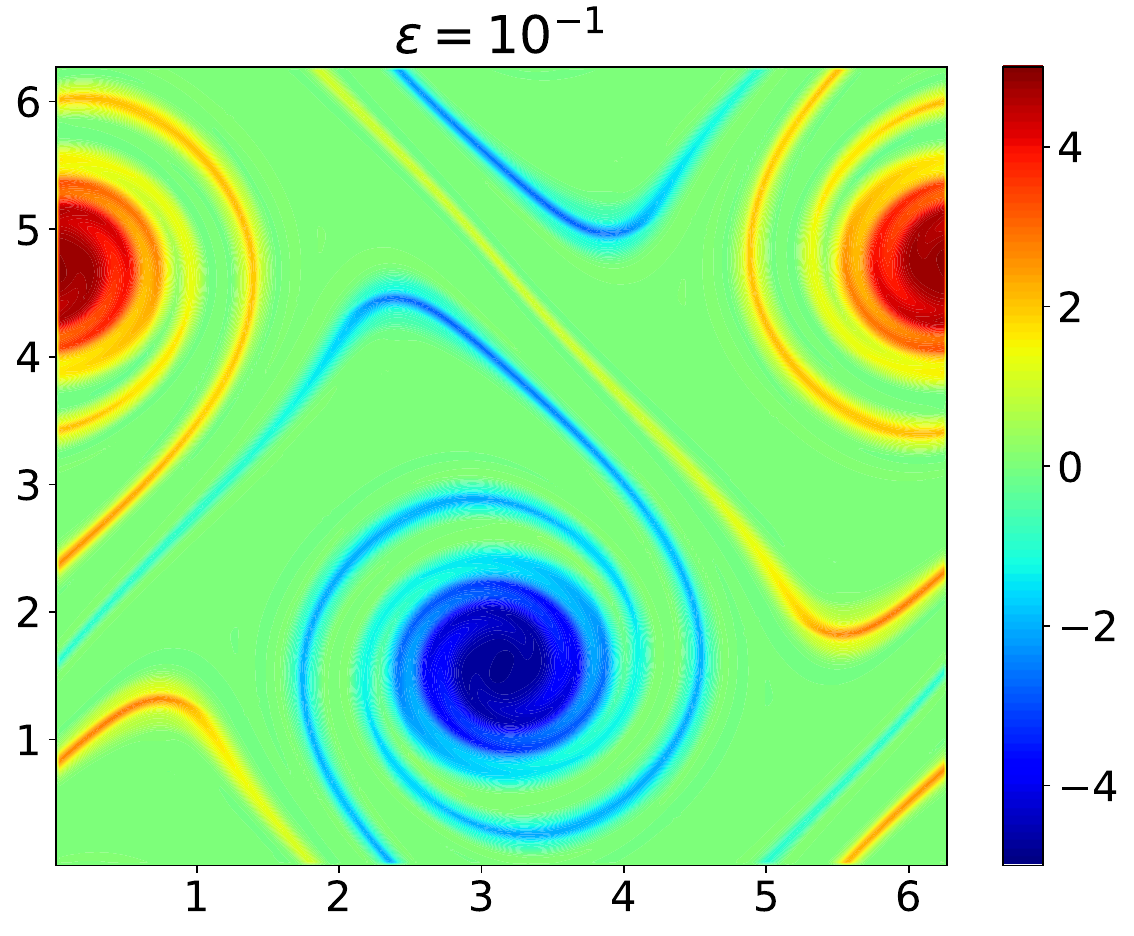}\hspace*{0.1cm}
		\includegraphics[width=0.33\textwidth]{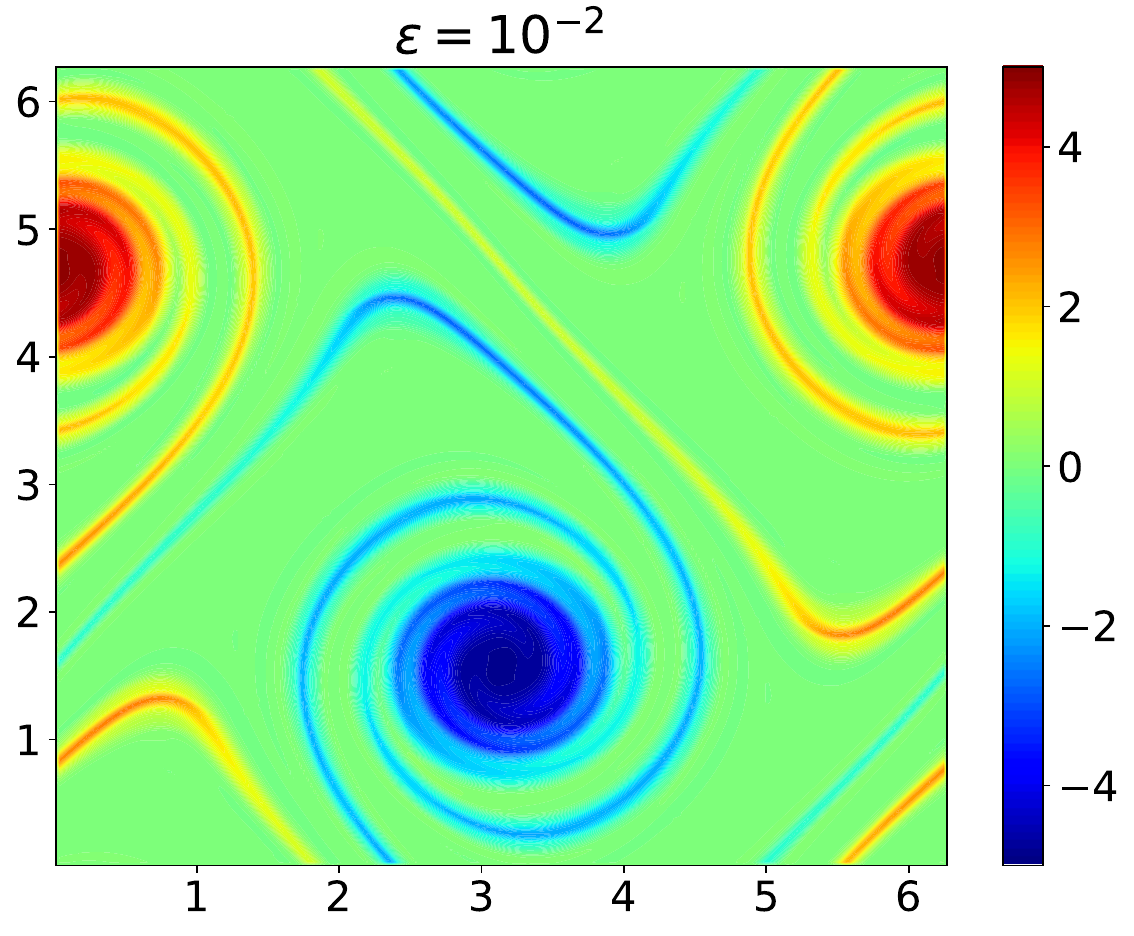}\hspace*{0.1cm}
		\includegraphics[width=0.33\textwidth]{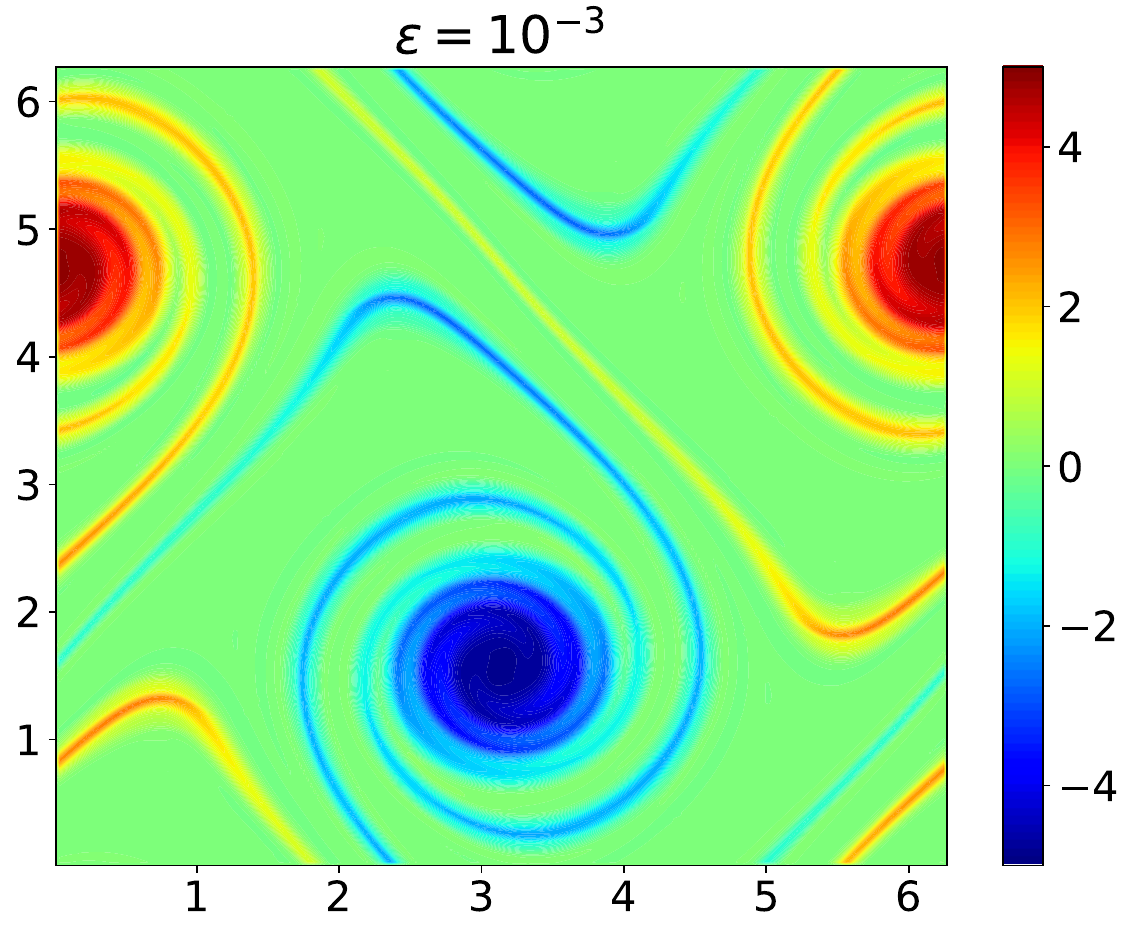}}
	\vskip8pt
	\centerline{\includegraphics[width=0.33\textwidth]{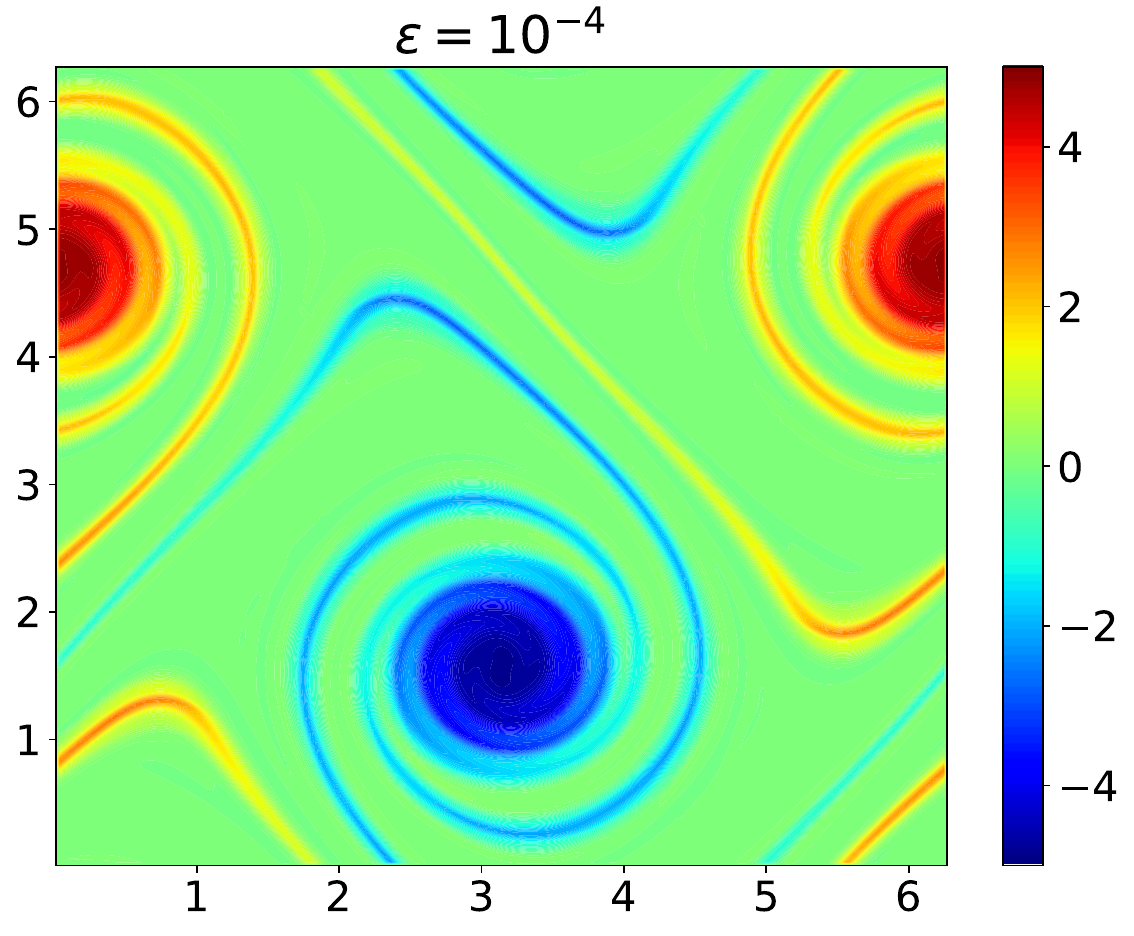}\hspace*{0.1cm}
		\includegraphics[width=0.33\textwidth]{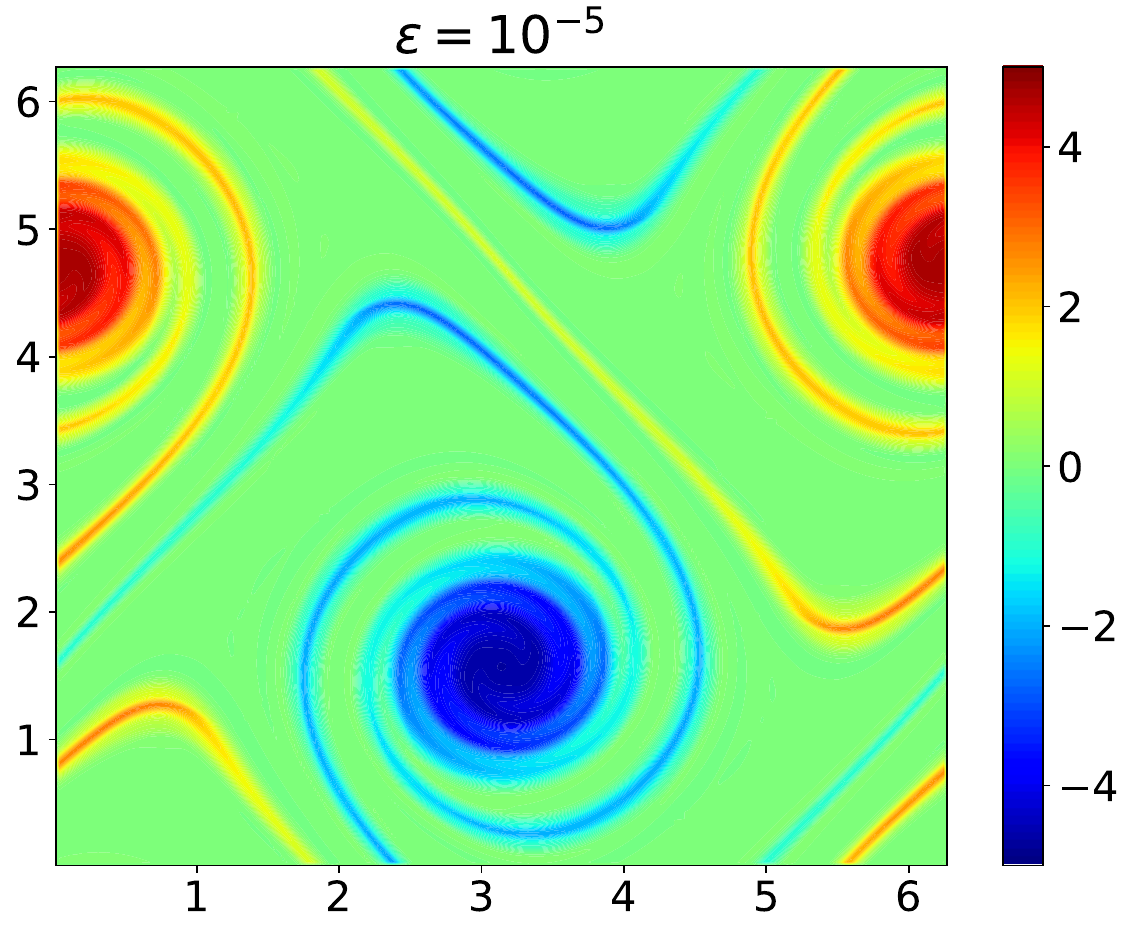}\hspace*{0.1cm}
		\includegraphics[width=0.33\textwidth]{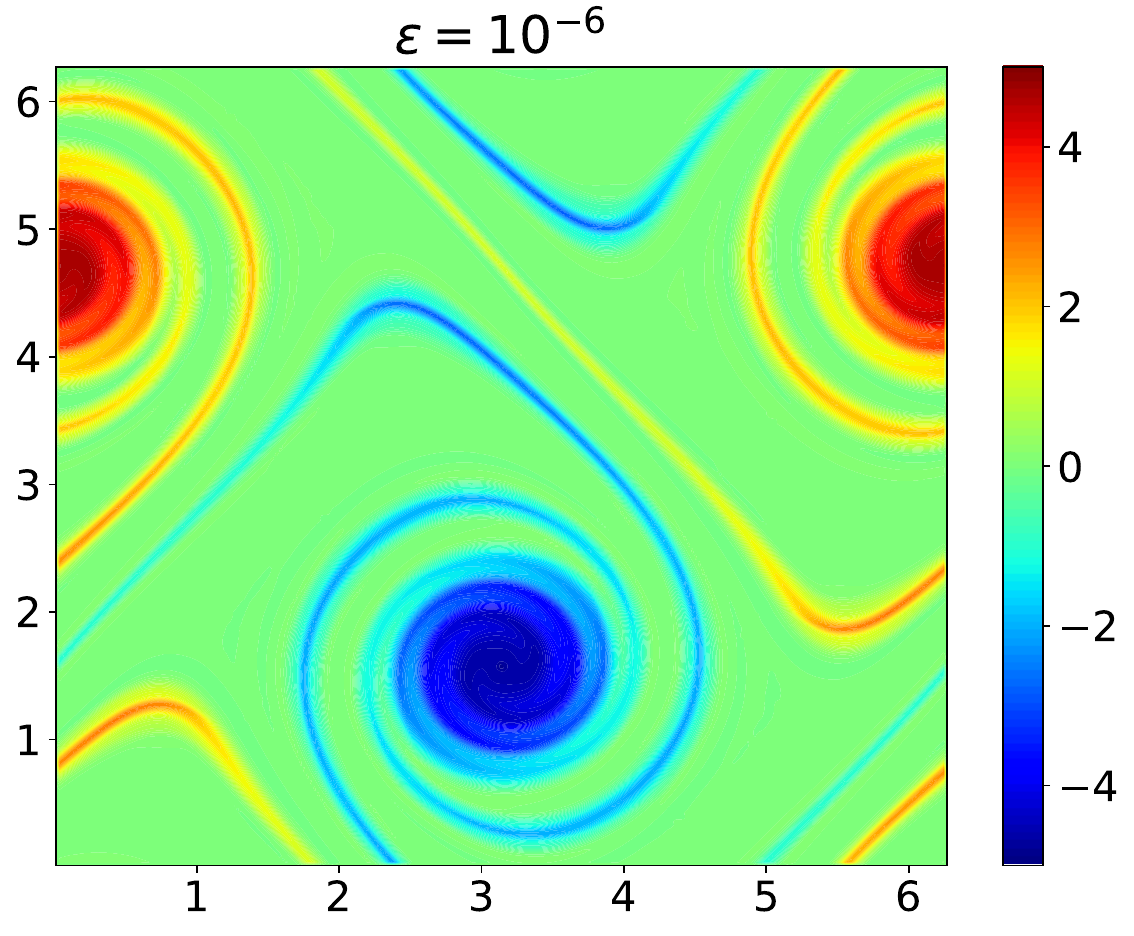}}
	\caption{\sf Example 4: The same as in Figure \ref{fig46}, but at $t=10$.\label{fig47}}
\end{figure}

We remark that the simulations remain stable for larger CFL numbers. However, the use of larger $K_{\rm CFL}$ may lead to a noticeable 
increase in the amount of the numerical diffusion for very small values of $\ve\lesssim10^{-3}$. To illustrate this, we recompute the 
solution with $K_{\rm CFL}=0.475$ for $\ve=10^{-4}$, $10^{-5}$, and $10^{-6}$ and plot the obtained results (for $t=6$) in Figure 
\ref{fig48}. As one can clearly see, the numerical solution is now substantially more diffusive compared with those reported in the bottom 'row of Figure \ref{fig46}.
\begin{figure}[ht!]
	\centerline{\includegraphics[width=0.33\textwidth]{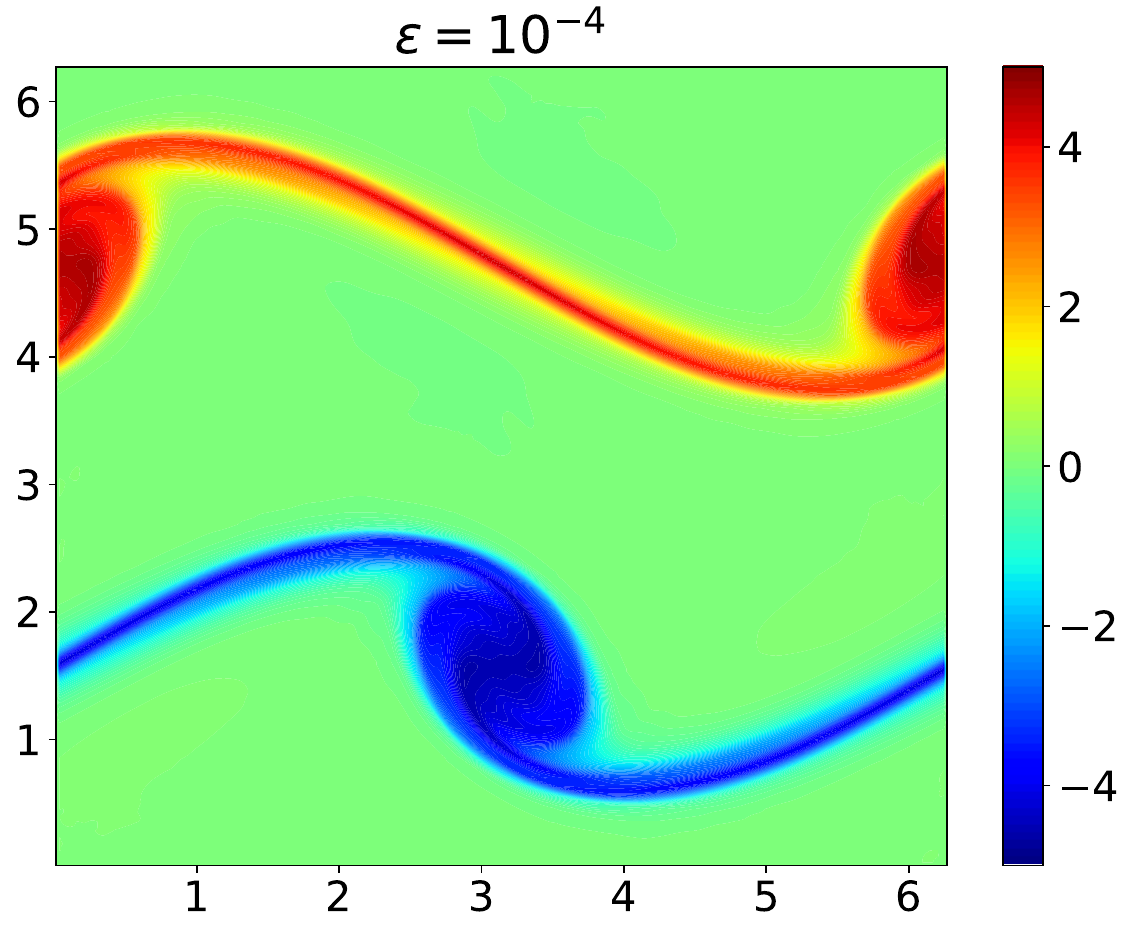}\hspace*{0.1cm}
		\includegraphics[width=0.33\textwidth]{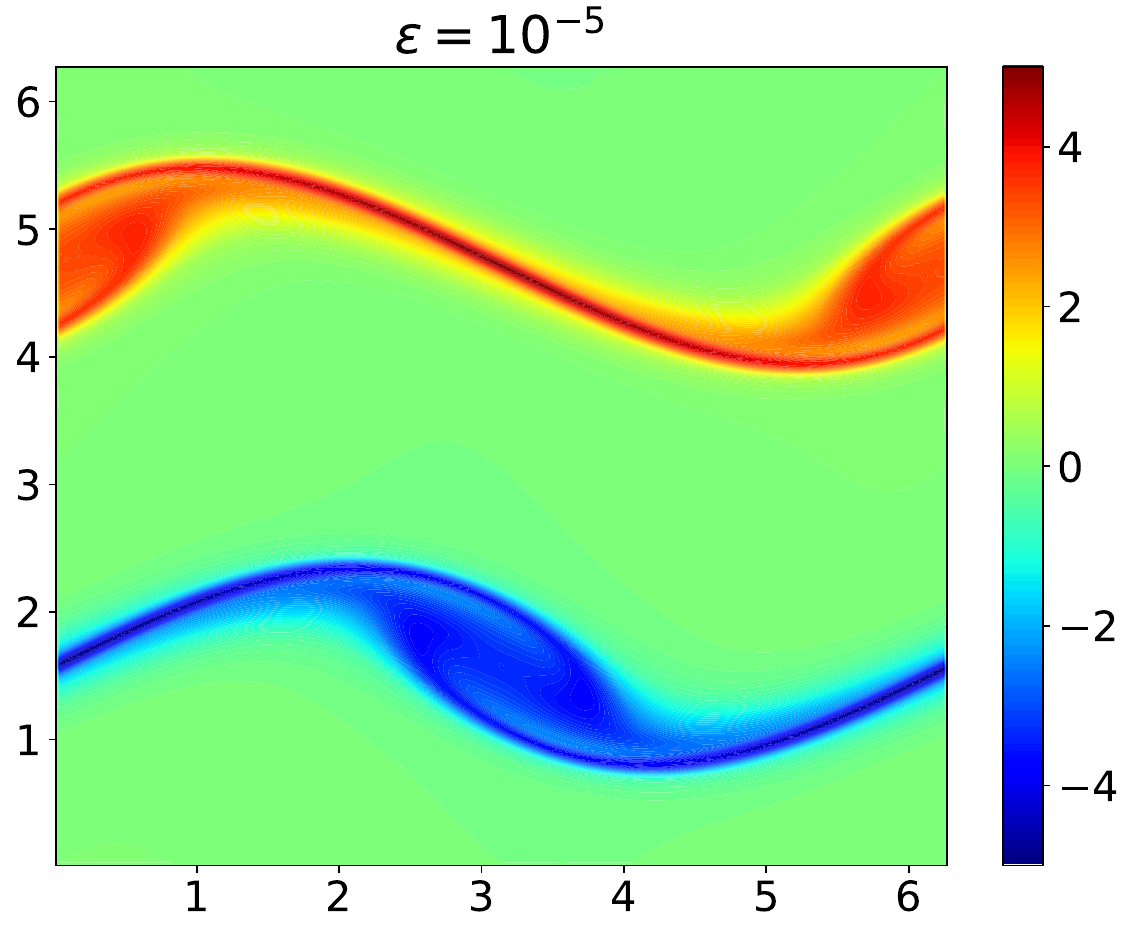}\hspace*{0.1cm}
		\includegraphics[width=0.33\textwidth]{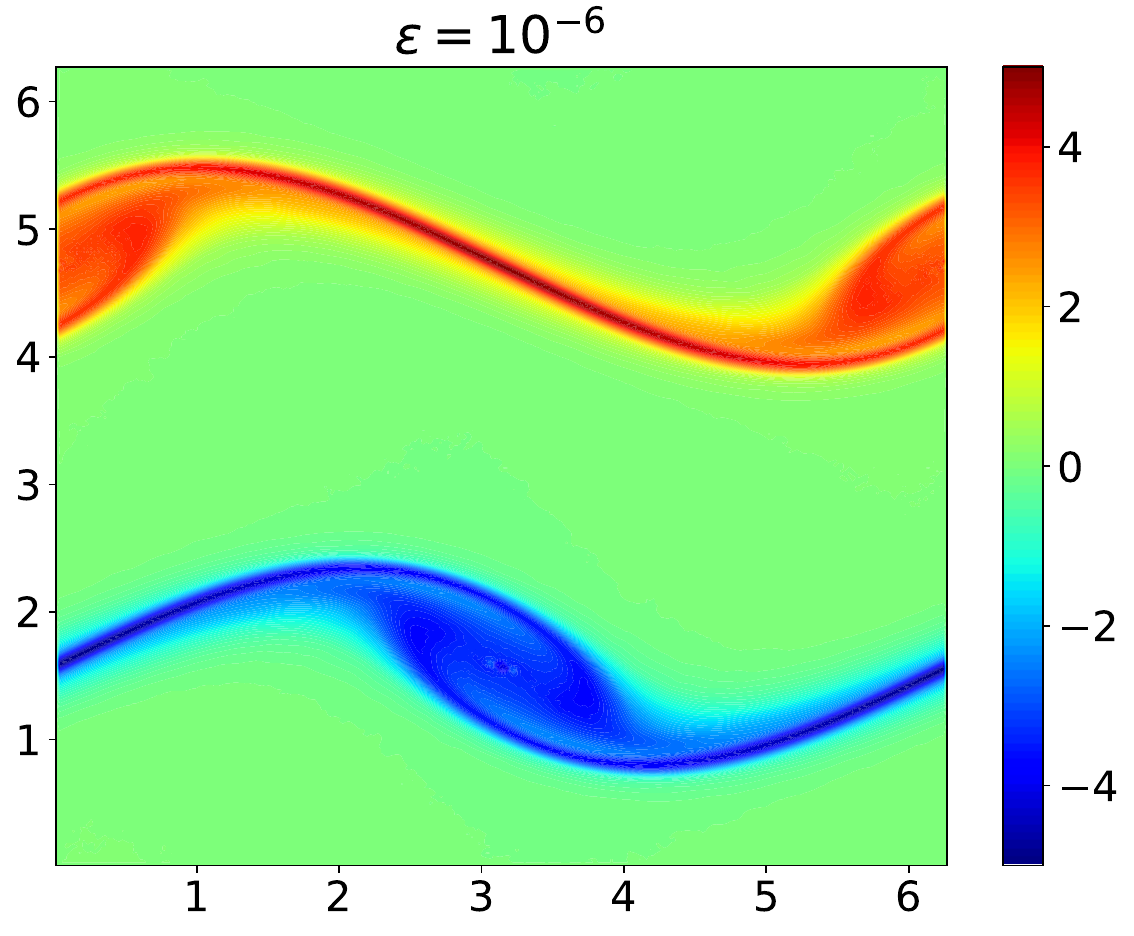}}
	\caption{\sf Example 4: The same as in Figure \ref{fig46}, but for $K_{\rm CFL}=0.475$ (left).\label{fig48}}
\end{figure}

\subsubsection*{Example 5---Explosion Problem}
In the last numerical example, we consider an explosion problem taken from \cite{ToroBook}. The initial data,
$$
(\rho,u,v,p)(x,y,0)=\begin{cases}(1,0,0,1),&\sqrt{x^2+y^2}<0.4,\\(0.125,0,0,0.1),&\mbox{otherwise},\end{cases}
$$
are prescribed in the computational domain $[-1,1]\times[-1,1]$ subject to the free boundary conditions.

The main objective of this test is to verify that the proposed AP DF-FV scheme remains accurate and stable, also in the high-Mach-number
regime, in which strong shocks and contact discontinuities may be present. To this end, we perform simulations for several values of $\ve$.
For $\ve=0.9$, $0.6$, and $0.3$ the final times are $t=0.2$, $0.15$, and $0.08$, respectively. The surface plots of the density $\rho$
computed on a uniform mesh with $400\times400$ cells using $K_{\rm CFL}=0.475$ are reported in Figure \ref{fig49}, where one can see that
the obtained solutions are oscillation-free and their nonsmooth features are accurately resolved for all values of $\ve$.
\newcommand{\trimfig}[1]{\includegraphics[width=0.35\textwidth, trim=140 260 190 250, clip]{#1}}
\begin{figure}[ht!]
	\centerline{\hspace*{-0.7cm}\trimfig{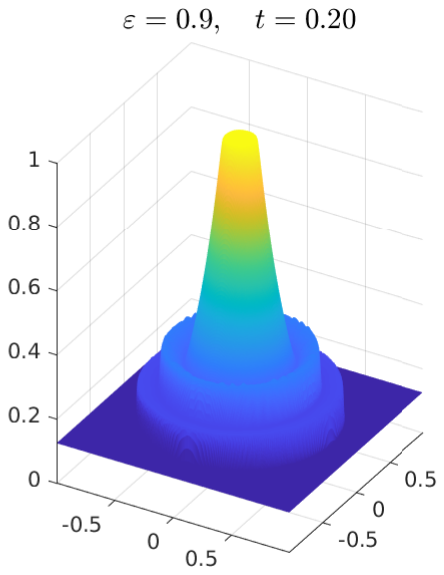}\hspace*{-0.2cm}
		\trimfig{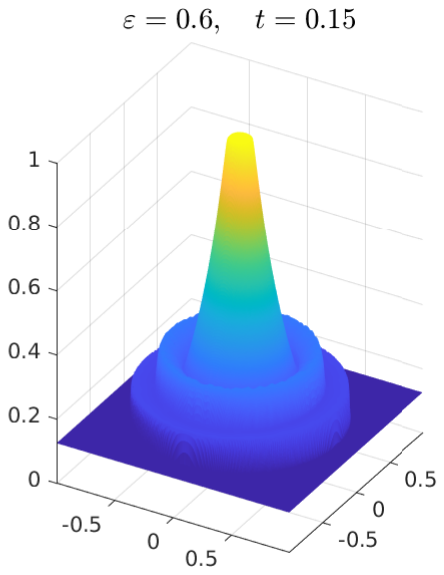}\hspace*{-0.2cm}
		\trimfig{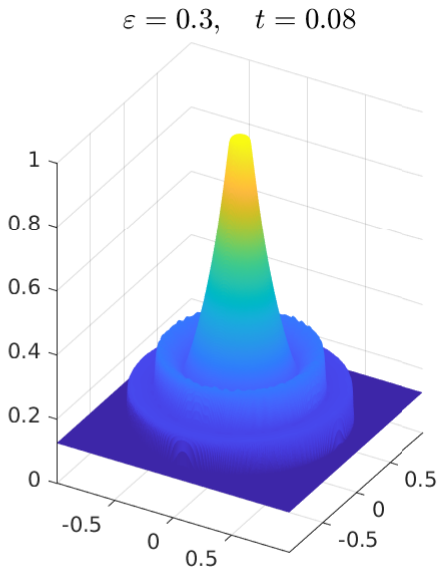}}
	\caption{\sf Example 5: Surface plot of density for different values of $\ve$ and corresponding times.\label{fig49}}
\end{figure}

To further assess the correctness of the computed solutions, we plot their one-dimensional (1-D) slices along the diagonal $y=x$ in Figure
\ref{fig410} together with the corresponding slices of the reference solution, which was obtained using the second-order semi-discrete CU
scheme from \cite{KLin} on a much finer mesh with $2000\times2000$ cells using the CFL number $0.2$ and the three-stage third-order strong
stability preserving (SSP) Runge-Kutta method \cite{GKS,GST}. As one can clearly see, the computed solutions show a perfect agreement with
the reference ones, and the discontinuities locations are correctly captured.
\begin{figure}[ht!]
	\centerline{\includegraphics[width=0.8\textwidth]{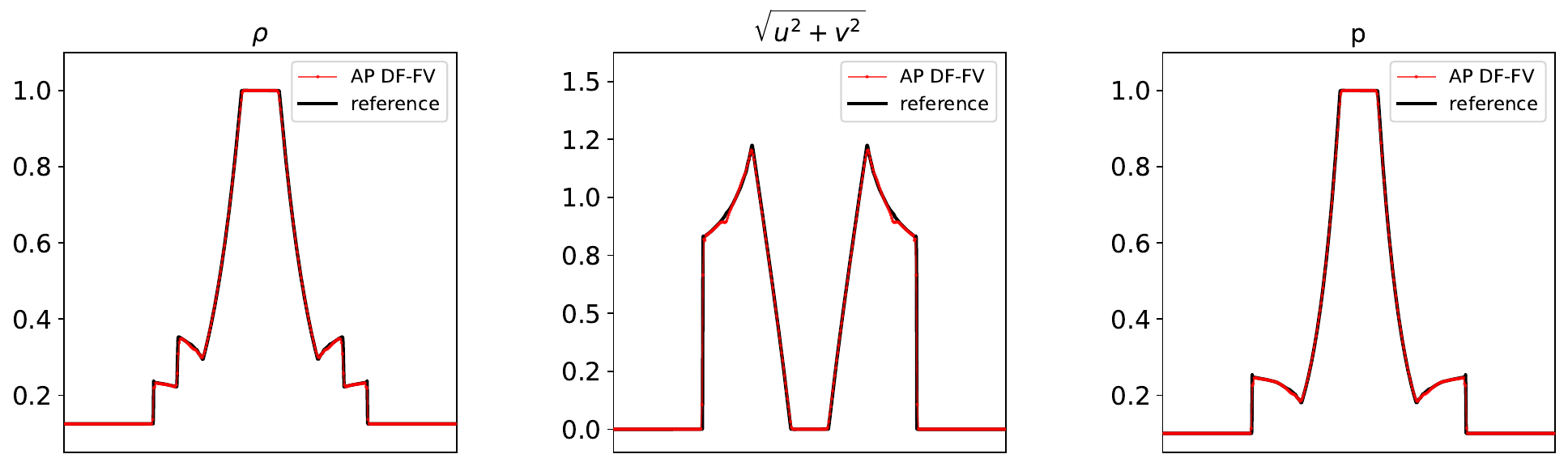}}
	\vskip3pt
	\centerline{\includegraphics[width=0.8\textwidth]{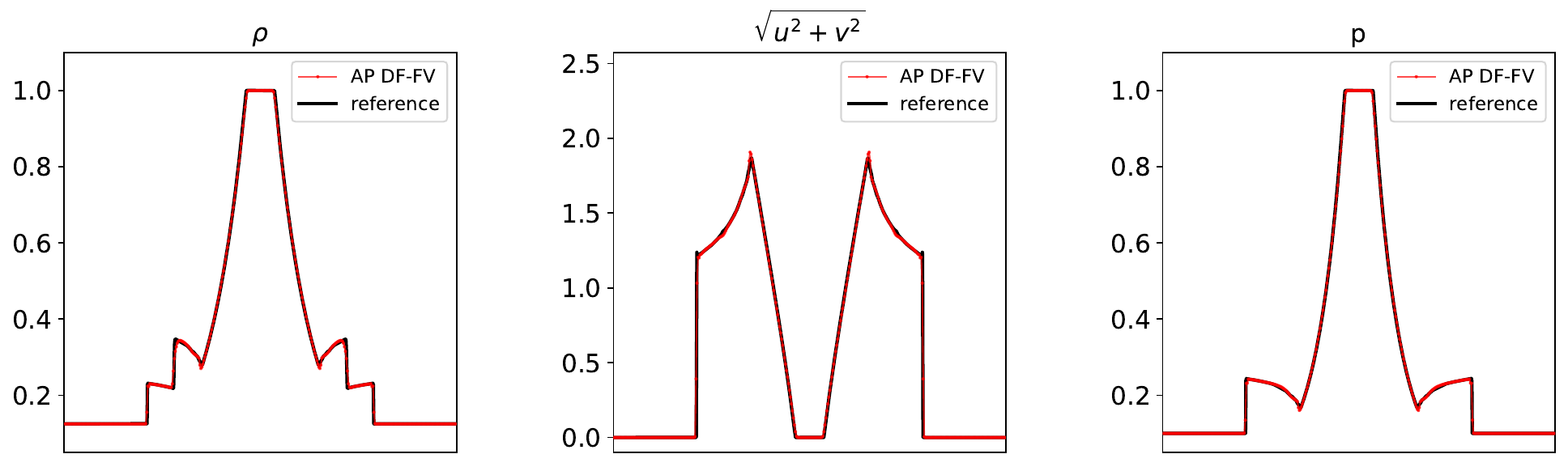}}
	\vskip3pt
	\centerline{\includegraphics[width=0.8\textwidth]{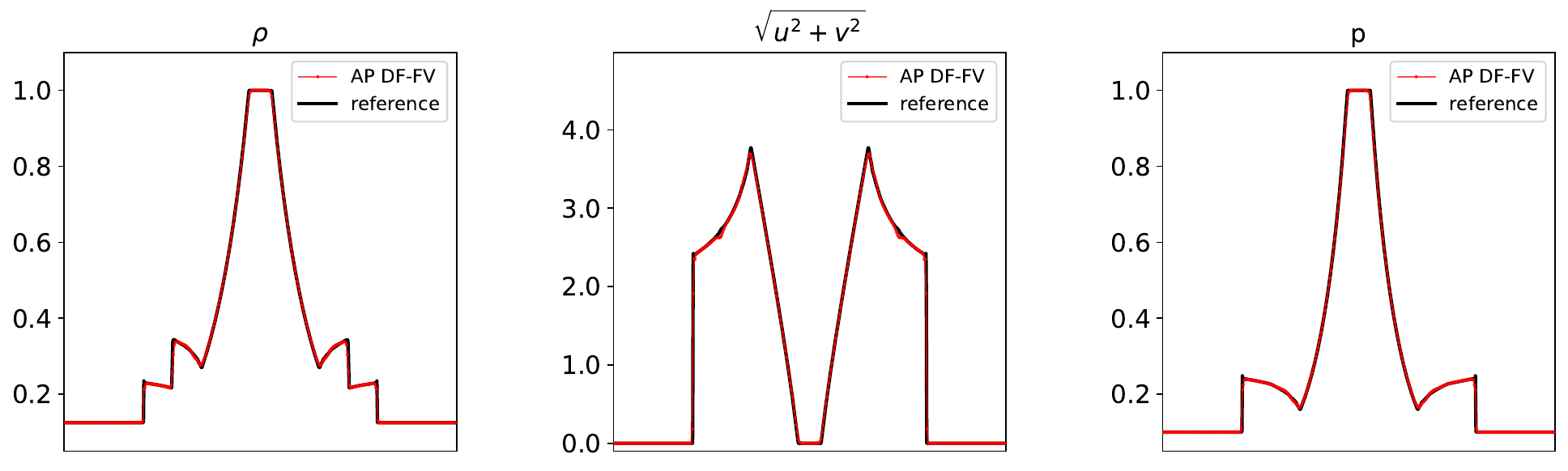}}
	\caption{\sf Example 5: 1-D slices of the computed solutions along $y=x$ for different values of $\ve$ and at different times: $\ve=0.9$,
		$t=0.2$ (top row), $\ve=0.6$, $t=0.15$ (middle row), and $\ve=0.3$, $t=0.08$ (bottom row).\label{fig410}}
\end{figure}

{\color{black}We stress that for the large (intermediate) Mach numbers considered in this example, the $\bm U$-solution is reliable and
	plays a crucial role in ensuring a correct handling of the discontinuities through the post-processing. We omit the plots of the
	$\bm U$-solution because they are visually indistinguishable from the ones reported. In this situation, one can see how crucial is the role
	of the post-processing. Without it, the nonconservative evolution of the primitive variables would lead to incorrect solutions. While in the
	case when $\ve=0.9$ or $0.6$, the weight $s(\ve)$ in \eref{star} is very close to $0$ and thus the post-processed solution is basically the
	conservative one, $\ve=0.3$ corresponds to a truly intermediate-Mach-number regime and it is instructive to look at the unreliable
	$\bm V$-solution obtained by the same scheme but without the post-processing. Such $\bm V$-solution is reported in Figure
	\ref{fig411}, where one can see a slight difference in the location of the discontinuities as well as other small numerical artifacts.
	Furthermore, we emphasize that in order to run the ``solely'' $\bm V$-simulation, it was necessary to lower $K_{\rm CFL}$ down to $0.01$.}
\begin{figure}[ht!]
	\centerline{\includegraphics[width=0.8\textwidth]{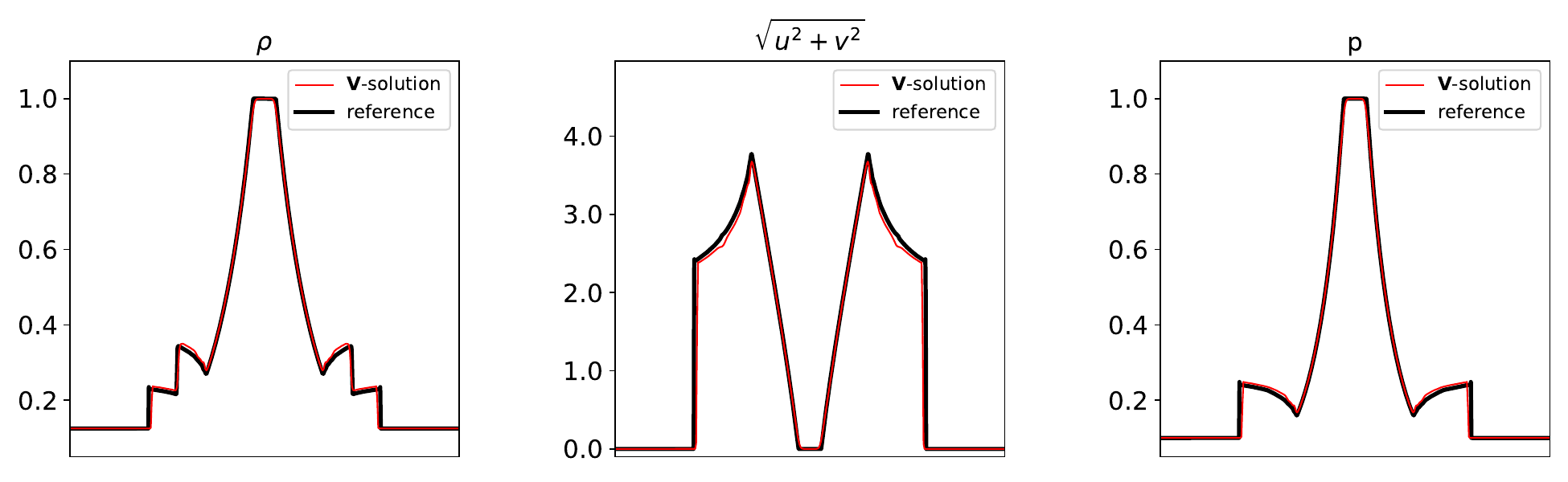}}
	\caption{\sf Example 5: 1-D slices of the ``solely'' $\bm V$-solution along $y=x$ for $\ve=0.3$, computed at $t=0.08$ without
		post-processing.\label{fig411}}
\end{figure}

\begin{remark}
	We stress that in this example, both terms in the numerator in \eref{4.48f} will vanish if $\rho_{\max}$ and $p_{\min}$ are computed using
	\eref{4.3}. While the modification \eref{4.9} ensures positivity of $\tilde c$, the resulting time steps might still be too big to guarantee
	stability of the AP DF-FV method. Therefore, we set $\dt=10^{-4}$ for the first 10 time steps for the simulations involving $\ve=0.6$ and
	$0.3$.
\end{remark}
\begin{remark}
	Let us remark that discontinuities are unlikely to occur in low-Mach-number flows. Consequently, the above tests with $\ve=0.6$ and $0.3$
	should be regarded as ``academic'' and are primarily intended to demonstrate that the proposed AP DF-FV scheme is capable of handling
	discontinuities even in the low-Mach-number regime.
\end{remark}

\section{Conclusion}\label{sec5}
We have presented a novel asymptotic-preserving (AP) numerical method for the compressible Euler equations that is effective across all
Mach-number regimes, including the low-Mach-number one, where standard explicit schemes become inefficient. The key idea is a new 
hyperbolic splitting, inspired by the flux-splitting approach introduced in \cite{haack2012all}. The new splitting is applied to a 
primitive (nonconservative) formulation of the Euler equations, which enables one to design an efficient semi-implicit (SI) time 
discretization. Our splitting isolates stiff linear terms, which are discretized semi-implicitly: this leads to a well-posed linear 
elliptic problem, which ensures the AP property of the resulting scheme.

To overcome the well-known difficulties associated with the use of nonconservative formulations in the presence of discontinuities, we
implement the proposed AP scheme within the recently introduced dual formulation framework \cite{CFKM,ACKM}. In this approach, the
conservative and primitive systems are solved simultaneously, and their resulting solutions are post-processed to ensure the correct
capturing of discontinuities while retaining the AP property of the primitive-based SI approach.

The proposed AP dual formulation finite-volume (DF-FV) method has been thoroughly validated on several benchmarks ranging from the fully
compressible to the nearly incompressible regime, demonstrating both high accuracy and robustness of the method. Future work will focus on
extending the AP DF-FV framework to more complex systems and on developing higher-order spatial and temporal discretizations.

\begin{acknowledgment}
The work of A. Chertock was supported in part by NSF grant DMS-2208438. The work of A. Kurganov was supported in part by NSFC grant
W2431004. The work of L. Micalizzi was supported in part by the LeRoy B. Martin, Jr. Distinguished Professorship Foundation.
\end{acknowledgment}

\bibliography{ref}
\bibliographystyle{siam}
\end{document}